\documentclass[a4paper]{article}
\usepackage[utf8]{inputenc}
\usepackage[english]{babel}
\usepackage{amsmath,amsthm,amssymb}
\usepackage{bbm}	%% used for \mathbbm
\usepackage[square,numbers]{natbib}	%% bibliography
\usepackage{graphicx, color, soul} %% soul for highlight
\usepackage[title]{appendix}
\usepackage[paperwidth=19.5cm,paperheight=25cm]{geometry}
\usepackage{url}
\usepackage{blkarray} %matrix with labels
\usepackage[colorlinks=true,linkcolor=blue,citecolor=blue]{hyperref}
\usepackage[capitalize,nameinlink]{cleveref}
\usepackage{authblk}

\usepackage[shortlabels]{enumitem}	%% change style of lists as "[(i)]"
\setlist{labelsep=.25in,leftmargin=*,labelindent=1cm,topsep=2pt,noitemsep}	%% adjust lists
\setlist[enumerate]{label=(\roman*)}

\usepackage{ragged2e}

\theoremstyle{plain}
\newtheorem{theorem}{Theorem}
\newtheorem{corollary}[theorem]{Corollary}
\newtheorem{proposition}[theorem]{Proposition}
\newtheorem{lemma}[theorem]{Lemma}

\theoremstyle{remark}
\newtheorem{remark}[theorem]{Remark}
\theoremstyle{definition}
\newtheorem{example}[theorem]{Example}

\newcommand{\Polya}{P\'{o}lya }

\usepackage{geometry}
\counterwithin{equation}{section}
\counterwithin{theorem}{section}

\title{Some developments of exchangeable measure-valued P\'{o}lya sequences}
\author[1]{Yoana R. Chorbadzhiyska\thanks{jchorbadzh@uni-sofia.bg; yoanarch@phys.uni-sofia.bg}}
\author[2,1]{Hristo Sariev\thanks{h.sariev@math.bas.bg; hsariev@uni-sofia.bg}}
\author[1,2]{Mladen Savov\thanks{msavov@fmi.uni-sofia.bg; mladensavov@math.bas.bg}}
\affil[1]{\normalsize Faculty of Mathematics and Informatics, Sofia University ``St. Kliment Ohridski'', 5 James Bourchier Blvd, Sofia 1164, Bulgaria\vspace{0.2cm}}
\affil[2]{\normalsize Institute of Mathematics and Informatics, Bulgarian Academy of Sciences, 8 Acad. Georgi Bonchev Str., Sofia 1113, Bulgaria}

\date{}

\begin{document}

\maketitle

\begin{abstract}
Measure-valued P\'{o}lya sequences (MVPS) are stochastic processes whose dynamics are governed by generalized P\'{o}lya urn schemes with infinitely many colors. Assuming a general reinforcement rule, exchangeable MVPSs can be viewed as extensions of Blackwell and MacQueen's P\'{o}lya sequence, which characterizes an exchangeable sequence whose directing random measure has a Dirichlet process prior distribution. Here, we show that the prior distribution of any exchangeable MVPS is a Dirichlet process mixture with respect to a latent parameter that is associated with the atoms of an emergent conditioning $\sigma$-algebra. As the mixing components have disjoint supports, the directing random measure can be interpreted as a random histogram with bins randomly located on these same atoms. Furthermore, we extend the basic exchangeable MVPS to include a null component in the reinforcement, which corresponds to the presence of a fixed component in the directing random measure. Finally, we examine the effects of relaxing exchangeability to conditional identity in distribution (c.i.d.) and find out that the two are equivalent for balanced MVPSs. The paper features a complementary study of some properties of probability kernels that underlies the analysis of exchangeable and c.i.d. MVPSs.
\end{abstract}


\section{Introduction}

The now classical \textit{\Polya sequence} lies at the heart of Bayesian nonparametric analysis, characterizing an exchangeable sequence of random variables with a \textit{Dirichlet process} (DP) prior distribution through its system of predictive distributions. More formally, a sequence $(X_n)_{n\geq1}$ of random variables, taking values in some space, say $\mathbb{X}=[0,1]$, is called a \Polya sequence (PS) if $\mathbb{P}(X_1\in\cdot)=\nu(\cdot)$ and, for each $n=1,2,\ldots$, the predictive distribution of $X_{n+1}$ given $X_1,\ldots,X_n$ is the probability measure
\begin{equation}\label{eq:intro:polya_sequence}
\mathbb{P}(X_{n+1}\in\cdot\mid X_1,\ldots,X_n)=\frac{\theta\nu(\cdot)+\sum_{i=1}^n\delta_{X_i}(\cdot)}{\theta+n},
\end{equation}
where $\theta>0$ is a positive constant, $\nu$ a probability measure on $\mathbb{X}$, and $\delta_x$ the unit mass at $x$; if $|\mathbb{X}|=k$, the model is also known as a $k$-color \Polya urn. It follows from Theorem 1 in \cite{blackwell1973ferguson} that $(X_n)_{n\geq1}$ is an exchangeable process whose \textit{directing random measure} has a DP prior distribution with parameters $(\theta,\nu)$. We recall that the directing random measure of an exchangeable sequence is the common weak limit of its empirical measure and predictive distributions, and refer to Section \ref{section:model:prelim} for more information.

Much of the subsequent work in the field of Bayesian nonparametrics builds on the exchangeable model with a DP prior, generalizing some of its many defining characteristics; see Section 4.4 and Figure 14.5 in \cite{ghosal2017}. For example, species sampling sequences were introduced by \cite{pitman1996} as an extension of the sampling procedure described by \eqref{eq:intro:polya_sequence}, whereas Gibbs-type processes \citep{lijoi2007,deblasi2015} represent a natural generalization of the random partition process generated by a PS, also known as the Chinese restaurant process. Moreover, the directing random measure with a DP prior distribution has motivated the study of the class of normalized random measures with independent increments \citep{regazzini2003}, DP mixture models \citep{lo1984}, and other important families of prior distributions; we refer to \cite{lijoi2010} for a comprehensive review. One other feature, highlighted by the predictive construction \eqref{eq:intro:polya_sequence}, is that the dynamics underlying the model can be interpreted as a sequence of draws from an urn that contains balls of infinitely many colors. In this framework, urn contents are described compactly by finite measures in the sense that, for any measurable set $B\subseteq\mathbb{X}$, the quantity $\theta\nu(B)$ records the initial \textit{mass} of balls whose colors lie in $B$. According to the urn scheme implied by \eqref{eq:intro:polya_sequence}, we pick the first ball from the normalized content distribution $\nu$ and, given that color $X_1$ is observed, reinforce the urn with another ball of the same color. Reinforcement here reduces to a summation of measures, so that we pick the next ball from the updated urn composition, $\theta\nu+\delta_{X_1}.$ Thus, after $n$ draws, the probability that the color of the $(n+1)$-st ball is in $B$ will be proportional to $\theta\nu(B)+\sum_{i=1}^{n}\delta_{X_i}(B)$.

One way of generalizing the above urn scheme is to consider an arbitrary reinforcement rule, which formally means replacing $\delta_x$ with a general finite measure $R_x$ on $\mathbb{X}$. The resulting class of \textit{measure-valued \Polya urn processes}, tracking urn contents, was first proposed by \cite{thacker2022} and \cite{mailler2017} as an extension of the generalized \Polya urn model to arbitrary color spaces. In this setting, the observation process $(X_n)_{n\geq1}$, also known as a \textit{measure-valued \Polya sequence} (MVPS), has predictive distributions given by
\begin{equation}\label{eq:intro:mvps}
\mathbb{P}(X_{n+1}\in\cdot\mid X_1,\ldots,X_n)=\frac{\theta\nu(\cdot)+\sum_{i=1}^nR_{X_i}(\cdot)}{\theta+\sum_{i=1}^nR_{X_i}(\mathbb{X})}.
\end{equation} 
While exchangeability is a feature of the model \eqref{eq:intro:polya_sequence}, MVPSs need not be exchangeable in general. Indeed, most studies of MVPSs prove, under ``irreducibility''-type assumptions on $R$, that the predictive distributions \eqref{eq:intro:mvps} have a deterministic weak limit; see, e.g., \cite{thacker2022,thacker2020,janson2021,janson2023,mailler2017,mailler2020}. By Lemma 8.2 in \cite{aldous1985exchangeability}, a stochastic process whose predictive distributions converge weakly is \textit{asymptotically exchangeable} with directing random measure the same predictive limit; thus, when the limit is deterministic, the process becomes asymptotically i.i.d. An example of an MVPS with a random predictive limit is the \textit{randomly reinforced \Polya sequence} (RRPS) by \cite{sariev2021,sariev2023,bassetti2010}, who consider a general ``diagonal'' reinforcement rule, $R_{X_i}=W_i\delta_{X_i}$, where they add a random number, $W_i$, of additional balls of the observed color. However, unless the $W_i$'s are constant (corresponding to the reinforcement of a PS), the resulting RRPS will not be exchangeable.

In this paper, we focus primarily on MVPSs that are exchangeable. Until recently, the only known examples of exchangeable MVPSs, apart from the PS, were the particular $k$-color urn models considered in \citet[p.\,1591]{hill1987} and \citet[Section 2]{dakkak2014}; thus, even the basic question of which $k$-color urns are exchangeable had remained open. Recent work by \cite{berti2023kernel,sariev2024,sariev2023sufficientness} has substantially clarified this issue, revealing several structural properties of the entire class. In particular, it is now understood that $(i)$ exchangeable MVPSs are necessarily \textit{balanced}, meaning that we always add the same total number of balls in the urn; $(ii)$ the reinforcement $R$ is a regular conditional distribution for $\nu$ given some sub-$\sigma$-algebra $\mathcal{G}$, that is, $R(\cdot)=\nu(\cdot\mid\mathcal{G})$; and $(iii)$ the directing random measure of any exchangeable MVPSs admits a stick-breaking representation analogous to that of the DP, where $\delta$ is now replaced by $R$. Although technical, property $(ii)$ is essential for all subsequent analysis and implies, for example, that all $k$-color exchangeable MVPSs have a particular block-diagonal reinforcement design (see Example \ref{results:hierarchical:example:k-color}). In fact, \cite{berti2023kernel} first developed the theory under the assumption that $R(\cdot)=\nu(\cdot\mid\mathcal{G})$, and later \cite{sariev2024} proved that this condition is ultimately necessary for exchangeability. These and other results from \cite{berti2023kernel,sariev2024,sariev2023sufficientness} are summarized in Section \ref{section:model:mvps}.

Our goal here is to provide additional insight into the structure of exchangeable MVPSs and to study some natural extensions of the basic model. We first show that exchangeable MVPSs are, at a more fundamental level, DP mixture models with respect to (w.r.t.) a latent parameter that is associated with the conditioning $\sigma$-algebra in $(ii)$. Because the mixing components have disjoint supports, the directing random measure of any exchangeable MVPS can be interpreted as a random histogram whose bins are located on the atoms of the same $\sigma$-algebra. As such, its prior distribution can be seen as a genuine nonparametric extension of the classical random histogram prior \cite[][Example 5.11]{ghosal2017} by randomizing the locations and the ``upper'' shape of the bins, assigning a Dirichlet process prior to the bin weights, and at the same time keeping a simple sampling scheme. On the other hand, in all previous studies, reinforcement is assumed to be strictly positive, $R_x(\mathbb{X})>0$, so in the urn analogy new balls are always added after each draw. Although exchangeability prevents balls from being removed from the urn (a fact that we prove in Section \ref{section:results:null}), it is still possible to have zero reinforcement at times, leaving the urn unchanged after observing certain colors. Here, we extend the results in $(ii)$ and $(iii)$ to include model specifications that explicitly allow $R_x(\mathbb{X})=0$ for all $x$ in some set $Z\subseteq\mathbb{X}$, and we show that this is equivalent to mixing the directing random measure of the exchangeable MVPS on $Z^c$ with the deterministic measure $\nu(\cdot\mid Z)$. Finally, we examine the effects of relaxing exchangeability to the weaker property of \textit{conditional identity in distribution}. A recent line of research in Bayesian nonparametrics (see Section \ref{section:results:beyond}) considers predictive constructions of conditionally identically distributed (c.i.d.) processes, with the aim of capturing certain types of asymmetries between observations or temporary disequilibrium in the dynamics of the system. We show that, when balanced, c.i.d. MVPSs are necessarily exchangeable, so that the structural constraints imposed by \eqref{eq:intro:mvps} rule out the non-stationary behavior otherwise allowed by conditional identity in distribution. Curiously, it is still possible to have unbalanced c.i.d.~MVPSs that are not exchangeable, but this requires a particular form of the reinforcement kernel $R$.

The rest of the paper is organized as follows. Section \ref{section:kernel} contains a preliminary study of relevant measure-theoretic structures, which includes a parametric representation of $\sigma$-algebras and a characterization of regular conditional distributions in terms of their averaging properties that may be of independent theoretical interest. These results are essential when studying the reinforcement kernels of MVPSs under the assumptions of exchangeability or, more generally, conditional identity in distribution. In Section \ref{section:model}, we define exchangeable MVPSs and state their known properties, along with new properties that can be inferred from \cite{berti2023kernel,sariev2024}. All results regarding different new developments of the basic exchangeable MVPS are contained in Section \ref{section:results}. Proofs are postponed to Section \ref{section:proofs}.

\subsection{Contributions}

The following points summarize the main contributions in the paper:\vspace{0.2cm}
\begin{itemize}
\item Section \ref{section:kernel:rcd} is a short study of some measure-theoretic properties of probability kernels (properness, stationarity, self-averaging), which appear in the analysis of exchangeable/c.i.d. MVPSs, with new results linking these properties to each other (Theorem \ref{result:prelim:kernels:rcd}, Remark \ref{result:prelim:kernels:rcd:remark}) and relating them to regular conditional distributions (Proposition \ref{result:prelim:kernels:total-rcd}, Corollary \ref{result:prelim:kernels:s+r-rcd}). While Proposition \ref{result:prelim:kernels:total-rcd} is known in the probability literature, Corollary \ref{result:prelim:kernels:s+r-rcd} is a consequence of Theorem \ref{result:prelim:kernels:rcd} and Proposition \ref{result:prelim:kernels:total-rcd} and, in fact, answers a question raised by \citet[][p.\,11]{berti2023kernel}. Moreover, Corollary \ref{result:prelim:kernels:s+r-rcd} is related to some classical results in functional analysis (see Remark \ref{result:prelim:kernels:s+r-rcd:remark}).\vspace{0.2cm}

\item Section \ref{section:model:mvps} is a survey on exchangeable MVPSs. Existing results from \cite{sariev2024} and \cite{berti2023kernel} are systematically presented in the more general framework of MVPSs with a potential null component (Theorems \ref{introTh_1} and \ref{introTh_2}). New results on the posterior distribution and discreteness of the directing random measure $\tilde{P}$, which are given in \cite{berti2023kernel} under the stronger assumption $R(\cdot)=\nu(\cdot\mid\mathcal{G})$, but extend immediately to the entire class of exchangeable MVPSs by invoking the characterization in \cite{sariev2024}, are stated in Proposition \ref{introProp_1} and Theorem \ref{introTh_3}. In addition, we discuss possible generalizations of the model to measure-valued species sampling sequences (Remark \ref{model:mvps:remark_sss}).\vspace{0.2cm}

\item In Section \ref{section:results:hierarchy}, we refine the results of Section \ref{section:model:mvps} by applying the theory from Section \ref{section:kernel}. In particular, the properties of the conditioning $\sigma$-algebra $\mathcal{G}$, which are given but not used in Section \ref{section:model:mvps}, allow us to find a suitable parameterization through which we derive a hierarchical decomposition of the directing random measure $\tilde{P}$ (Proposition \ref{atoms_general} and Theorem \ref{atoms_general:mixture}). As a consequence, we are able to interpret $\tilde{P}$ as a Dirichlet process mixture and, in particular, as a random histogram, which helps us connect exchangeable MVPSs to some standard Bayesian nonparametric constructions (Remark \ref{atoms_general:mixture:remark}). A series of examples illustrate these results. In addition, Theorem \ref{atoms_general:mixture} implicitly proves the converse statement in Theorem \ref{introTh_2}.\vspace{0.2cm}

\item The starting premise of Section \ref{section:results:null} is the generalization of the basic exchangeable MVPS model to the case where $R_x$ is a signed measure, meaning that we allow the removal of balls from the urn. We prove in Proposition \ref{results:null:non-negative} that the latter is impossible under tenability conditions, although it does not exclude the presence of a null component in the reinforcement (the set $Z$ from the introduction), i.e., colors that leave the urn unchanged. The rest of the section is devoted to the extension of the results in Sections \ref{section:model:mvps} and \ref{section:results:hierarchy} under the assumption $\nu(Z)>0$. The proof of the representation of $R$ under $\nu(Z)>0$ in Theorem \ref{representation} uses Theorem \ref{introTh_1} after restricting the process to $Z^c$, while the hierarchical decomposition results in Corollary \ref{atoms_null_case} and Theorem \ref{directing_rm} follow for the most part from Proposition \ref{atoms_general} and Theorem \ref{atoms_general:mixture}, respectively.\vspace{0.2cm}

\item In Section \ref{section:results:beyond}, we study another development of the model in Section \ref{section:model:mvps} by relaxing exchangeability to conditional identity in distribution. We prove in Proposition \ref{results:other:cid} that balanced c.i.d. MVPS are necessarily exchangeable. In the unbalanced case, we show that it is still possible to have c.i.d. MVPSs that are not exchangeable (Example \ref{example:other:cid}), but this requires a specific form of the reinforcement kernel $R$. In particular, we derive a representation for $R$ (Theorem \ref{results:other:cid-balanced}) in the case when $\mathbb{X}$ is finite, which is similar in spirit to Theorem \ref{introTh_1}, but its proof uses fundamentally different arguments.
\end{itemize}

\section{Measure-theoretic detour}\label{section:kernel}

Unless stated otherwise, all random quantities are defined on a common probability space $(\Omega,\mathcal{H},\mathbb{P})$, which we assume is rich enough to support any randomizing variable we need. From now on, $(\mathbb{X},\mathcal{X})$ is a standard Borel space, in which case $\mathcal{X}$ is countably generated (c.g.). For any sub-$\sigma$-algebra $\mathcal{G}\subseteq\mathcal{X},$  we will say that $\mathcal{G}$ is \textit{c.g. under $\nu$} if there exists $C\in\mathcal{G}$ such that $\nu(C)=1$ and $\mathcal{G}\cap C$ is c.g. We refer to \cite{kallenberg2021} for any unexplained measure-theoretic details.

\subsection{Atoms of $\sigma$-algebras}\label{section:model:atoms}

Let $\mathcal{G}\subseteq\mathcal{X}$ be a sub-$\sigma$-algebra. Then $\mathcal{G}$ can be characterized by the function that maps points in $\mathbb{X}$ to the atoms of $\mathcal{G}$. To that end, we define the \textit{$\mathcal{G}$-atom} at $x\in\mathbb{X}$ to be the set
\[[x]_\mathcal{G}:=\bigcap_{G\in\mathcal{G},\,x\in G}G,\]
which is the smallest set of points that are indistinguishable from $x$ by $\mathcal{G}$-measurable sets. Then $\Pi:=\{[x]_\mathcal{G}:x\in\mathbb{X}\}$ forms a partition of $\mathbb{X}$, and $G=\bigcup_{x\in G}[x]_\mathcal{G}$, for every $G\in\mathcal{G}$. In general, atoms need not be $\mathcal{G}$-measurable sets, but for a c.g. $\sigma$-algebra $\mathcal{G}=\sigma(G_1,G_2,\ldots)$, it holds $[x]_\mathcal{G}=\{y\in\mathbb{X}:\delta_x(G_n)=\delta_y(G_n),n\in\mathbb{N}\}\in\mathcal{G}$; thus, if $\mathcal{G}$ is c.g. under $\nu$, then $[x]_\mathcal{G}\in\mathcal{G}$ for $\nu$-almost every (a.e.) $x$.

Let us define the map $\pi:\mathbb{X}\rightarrow\Pi$ by
\[\pi(x):=[x]_\mathcal{G},\qquad\mbox{for }x\in\mathbb{X},\]
and
\[\mathcal{G}_\pi:=\{P\subseteq\Pi:\pi^{-1}(P)\in\mathcal{G}\}.\]
It is straightforward to check that $\mathcal{G}_\pi$ is a $\sigma$-algebra on $\Pi$, so by construction, $\pi$ is $\mathcal{G}\backslash\mathcal{G}_\pi$-measurable; thus, $\mathcal{G}\supseteq\sigma(\pi)\equiv\pi^{-1}(\mathcal{G}_\pi)$. On the other hand, for each $x\in\mathbb{X}$,
\[\pi^{-1}\bigl(\bigl\{[x]_\mathcal{G}\bigr\}\bigr)=\{y\in\mathbb{X}:[x]_\mathcal{G}=[y]_\mathcal{G}\}=[x]_\mathcal{G}\qquad\mbox{and}\qquad \pi\bigl([x]_\mathcal{G}\bigr)=\bigl\{[y]_\mathcal{G}:y\in[x]_\mathcal{G}\bigr\}=\bigl\{[x]_\mathcal{G}\bigr\}.\]
Let $G\in\mathcal{G}$. Then $G=\bigcup_{x\in G}[x]_\mathcal{G}=\bigcup_{x\in G}\pi^{-1}\bigl(\bigl\{[x]_\mathcal{G}\bigr\}\bigr)=\pi^{-1}\bigl(\bigcup_{x\in G}\bigl\{[x]_\mathcal{G}\bigr\}\bigr)$. But $G\in\mathcal{G}$, so $\bigcup_{x\in G}\bigl\{[x]_\mathcal{G}\bigr\}\in\mathcal{G}_\pi$; therefore,
\[\mathcal{G}=\sigma(\pi).\]
Note that this result says nothing about the measurability of $\mathcal{G}$-atoms.

Regarding $\mathcal{G}_\pi$, since $\pi(G)=\pi\bigl(\bigcup_{x\in G}[x]_\mathcal{G}\bigr)=\bigcup_{x\in G}\pi\bigl([x]_\mathcal{G}\bigr)=\bigcup_{x\in G}\bigl\{[x]_\mathcal{G}\bigr\}$, we have \begin{equation}\label{model:atoms:eq1}
\pi^{-1}(\pi(G))=G,\qquad\mbox{for all }G\in\mathcal{G};
\end{equation}
thus, $\pi(\mathcal{G})\subseteq\mathcal{G}_\pi$. Moreover, from standard results, $\pi(\mathcal{G})$ is closed w.r.t. countable unions. Let $G\in\mathcal{G}$. Since $\Pi$ forms a partition of $\mathbb{X}$, we have $(\pi(G))^c=\{[x]_\mathcal{G}:x\in G\}^c=\{[x]_\mathcal{G}:x\in G^c\}=\pi(G^c)\in \pi(\mathcal{G})$; therefore, $\pi(\mathcal{G})$ is a $\sigma$-algebra on $\Pi$. Let $P\in\mathcal{G}_\pi$. Then $\pi^{-1}(P)\in\mathcal{G}$, so $P=\{[x]_\mathcal{G}:x\in\pi^{-1}(P)\}=\pi(\pi^{-1}(P))\in\pi(\mathcal{G})$, from which we conclude
\[\mathcal{G}_\pi=\pi(\mathcal{G}).\]
Now, on the measurable space $(\Pi,\pi(\mathcal{G}))$, we introduce the image probability measure
\[\nu_\pi=\nu\circ\pi^{-1}.\]

\subsection{Properties of probability kernels}\label{section:kernel:rcd}

A \emph{transition kernel} $R$ on $\mathbb{X}$ is a function $R:\mathbb{X}\times\mathcal{X}\rightarrow\mathbb{R}_+$ such that $(i)$ the map $x\mapsto R(x,A)\equiv R_x(A)$ is $\mathcal{X}$-measurable, for all $A\in\mathcal{X}$; and $(ii)$ $R_x$ is a measure on $\mathbb{X}$, for all $x\in\mathbb{X}$. Moreover, a transition kernel $R$ is said to be \textit{finite} if $R_x(\mathbb{X})<\infty$ for all $x\in\mathbb{X}$, and is called a \textit{probability kernel} if $R_x(\mathbb{X})=1$ for all $x\in\mathbb{X}$. A \textit{random probability measure} is a probability kernel $\tilde{P}:\Omega\times\mathcal{X}\rightarrow[0,1]$ from $\Omega$ to $\mathbb{X}$.

Let $\nu$ be a probability measure on $\mathbb{X}$, and $\mathcal{G}\subseteq\mathcal{X}$ a sub-$\sigma$-algebra. A probability kernel $R$ on $\mathbb{X}$ is said to be a \emph{regular version of the conditional distribution} (r.c.d.) for $\nu$ given $\mathcal{G}$, denoted by
\[R(\cdot)=\nu(\cdot\mid\mathcal{G}),\]
if the following two conditions are satisfied: \textit{a)} $x \mapsto R_x(A)$ is $\mathcal{G}$-measurable, for all $A\in\mathcal{X}$; and \textit{b)} $\int_BR_x(A)\nu(dx)=\nu(A\cap B)$, for all $A\in\mathcal{X}$ and $B\in\mathcal{G}$. The assumptions on $(\mathbb{X},\mathcal{X})$ guarantee that an r.c.d. for $\nu$ given $\mathcal{G}$ exists and is unique up to a $\nu$-null set.

We next introduce several properties that probability kernels typically possess. Let $R$ be a probability kernel on $\mathbb{X}$. We say that $R$ is \textit{almost everywhere proper} w.r.t.~some sub-$\sigma$-algebra $\mathcal{G}\subseteq\mathcal{X},$ provided there exists $F\in\mathcal{G}$ such that $\nu(F)=1$ and
\begin{equation}\label{condition:proper}\tag{$A$}
R_x(A)=\delta_x(A),\qquad\mbox{for all }A\in\mathcal{G}\mbox{ and }x\in F;
\end{equation}
\textit{stationary} w.r.t. $\nu$, provided
\begin{equation}\label{condition:stationarity}\tag{$B$}
\int_\mathbb{X}R_x(A)\nu(dx)=\nu(A),\qquad\mbox{for all }A\in\mathcal{X};
\end{equation}
and \textit{self-averaging}, provided
\begin{equation}\label{condition:reversible}\tag{$C$}
\int_\mathbb{X}R_y(A)R_x(dy)=R_x(A),\qquad\mbox{for all }A\in\mathcal{X}\mbox{ and }\nu\mbox{-a.e. }x.
\end{equation}
Note that when $\mathcal{G}$ is c.g. under $\nu$, \eqref{condition:proper} becomes equivalent to the more easily verifiable condition
\begin{equation}\label{condition:total}
R_x([x]_\mathcal{G})=1\qquad\mbox{for }\nu\mbox{-a.e. }x,
\end{equation}
where the essential set belongs to $\mathcal{G}$ (see the proof of Theorem \ref{result:prelim:kernels:rcd}).

Conditions \eqref{condition:proper}-\eqref{condition:reversible} appear separately or in combination in many different contexts, such as in the study of Markov processes \cite[Example 4]{berti2014}, disintegrations of probability measures \citep{berti1999}, ergodic theory \cite[Theorem 6.2]{einsielder2011}, statistical mechanics \cite[Section 2]{preston1976}, \cite[p.\,538]{sokal1981}, and some predictive constructions of probability laws \citep{berti2021}, see also Section \ref{section:results:beyond}. For r.c.d.s, \eqref{condition:stationarity} and \eqref{condition:reversible} follow from standard results on conditional expectations, while \eqref{condition:proper} is an important property of ``well-behaved'' r.c.d.s, with \cite{blackwell1975} calling it an ``intuitive desideratum'' for r.c.d.s; see \cite{berti2007,seidenfeld2001} for a discussion of improper r.c.d.s. In fact, by \cite[][Theorem 1]{blackwell1975} and \cite[][p.\,650]{berti2007},
\begin{equation}\label{condition:cg_under}
\nu(\cdot\mid\mathcal{G})\mbox{ satisfies }\eqref{condition:proper}\qquad\Longleftrightarrow\qquad\mathcal{G}\mbox{ is c.g. under }\nu,
\end{equation}
so that the properness of an r.c.d. is fundamentally linked to the properties of the conditioning $\sigma$-algebra.

We proceed by studying the relationship between \eqref{condition:proper}-\eqref{condition:reversible}, which we will use to characterize almost everywhere proper r.c.d.s in terms of their averaging properties. In Sections \ref{section:model} and \ref{section:results}, we will see that probability kernels associated with exchangeable MVPSs satisfy \eqref{condition:stationarity} and \eqref{condition:reversible}, and we will examine the consequences of this characterization. The next result shows that \eqref{condition:proper} decomposes into a measurability statement regarding $R_{|\mathcal{G}}$ together with the following particularization of \eqref{condition:stationarity} and \eqref{condition:reversible} on $\mathcal{G}$:
\begin{equation}\label{condition:stationarity1}\tag{$B'$}
\int_\mathbb{X}R_x(A)\nu(dx)=\nu(A),\qquad\mbox{for all }A\in\mathcal{G},
\end{equation} 
\begin{equation}\label{condition:reversible1}\tag{$C'$}
\int_\mathbb{X}R_y(A)R_x(dy)=R_x(A),\qquad\mbox{for all }A\in\mathcal{G}\mbox{ and }\nu\mbox{-a.e. }x,
\end{equation}
where $R_{x|\mathcal{G}}(A):=R_x(A)$, for $A\in\mathcal{G}$, is the restriction of $R_x$ on $(\mathbb{X},\mathcal{G})$, for all $x\in\mathbb{X}$. 

\begin{theorem}\label{result:prelim:kernels:rcd}
Let $R$ be a probability kernel on $\mathbb{X}$, and $\mathcal{G}\subseteq\mathcal{X}$ a c.g. under $\nu$ sub-$\sigma$-algebra. Then $R$ satisfies \eqref{condition:proper} if and only if it satisfies \eqref{condition:stationarity1}, \eqref{condition:reversible1}, and $\mathcal{G}=\sigma(R_{|\mathcal{G}})$ a.e.$[\nu]$.
\end{theorem}

\begin{remark}\label{result:prelim:kernels:rcd:remark}
Suppose in Theorem \ref{result:prelim:kernels:rcd} that
\[\mathcal{G}\equiv\sigma(R):=\sigma(\{x\mapsto R_x(A):A\in\mathcal{X}\}).\]
Since $\mathcal{X}$ is c.g., then $\sigma(R)$ is c.g. as well. Moreover, the $\sigma(R)$-atoms have the form
\begin{equation}\label{condition:atoms}
[x]_{\sigma(R)}=\{y\in\mathbb{X}:R_y=R_x\},\qquad\mbox{for }x\in\mathbb{X}.
\end{equation}
If $R$ satisfies \eqref{condition:proper} w.r.t. $\sigma(R)$, then it satisfies \eqref{condition:stationarity1}, \eqref{condition:reversible1}, and $\sigma(R)=\sigma(R_{|\sigma(R)})$ a.e.$[\nu]$. In this case, however, we are able to say something more. Since \eqref{condition:proper} implies through \eqref{condition:total} and \eqref{condition:atoms} that $R_y=R_x$, for $R_x$-a.e. $y$ and $\nu$-a.e. $x$, then
\[\int_\mathbb{X}R_y(A)R_x(dy)=R_x(A),\qquad\mbox{for all }A\in\mathcal{X}\mbox{ and }\nu\mbox{-a.e. }x;\]
thus, \eqref{condition:proper} w.r.t. $\sigma(R)$ implies the stronger condition \eqref{condition:reversible}. Conversely, assuming only \eqref{condition:reversible}, there exists $F\in\sigma(R)$ such that $\nu(F)=1$ and $R_x(A)=\int_\mathbb{X}R_y(A)R_x(dy)$, for all $A\in\mathcal{X}$ and $x\in F$. Since the map $x\mapsto\int_\mathbb{X}R_y(A)R_x(dy)$ is $\sigma(R_{|\sigma(R)})$-measurable, we get $\sigma(R)\cap F=\sigma(R_{|\sigma(R)})\cap F$. As a result, the measurability assumption in Theorem \ref{result:prelim:kernels:rcd} is satisfied under \eqref{condition:reversible}, and we obtain
\[\eqref{condition:proper}\mbox{ w.r.t. }\sigma(R)\qquad\Longleftrightarrow\qquad\eqref{condition:stationarity1}\mbox{ w.r.t. }\sigma(R)\quad+\quad\eqref{condition:reversible}.\]
\end{remark}~ 

Example \ref{example:connections:countable} shows that \eqref{condition:reversible} alone is not equivalent to \eqref{condition:proper} w.r.t. $\sigma(R)$.

\begin{example}\label{example:connections:countable}
Let $\mathbb{X}$ be countable, $\nu(\{x\})>0$ for all $x\in\mathbb{X}$, and $R$ be a probability kernel on $\mathbb{X}$ satisfying 
\begin{equation}\label{condition:reversible+}\tag{C+}
\int_BR_y(A)R_x(dy)=\int_AR_y(B)R_x(dy),\qquad\mbox{for all }A,B\in\mathcal{X}\mbox{ and }\nu\mbox{-a.e. }x,
\end{equation}
which is stronger than \eqref{condition:reversible}. In this setting, probability kernels are best described in terms of a matrix $[r_{xy}]_{x,y\in\mathbb{X}}$, where $r_{xy}=R_x(\{y\})$, in which case \eqref{condition:reversible+} becomes
\begin{equation}\label{eq:connections:SR->T}
r_{xy}r_{yz}=r_{xz}r_{zy}\qquad\mbox{for all }x,y,z\in\mathbb{X}.
\end{equation}
%% the assumption $\nu(\{x\})>0$ for all $x\in\mathbb{X}$ is used here.
Fix $x\in\mathbb{X}$ such that $R_x([x]_{\sigma(R)})>0$. We can assume $r_{xx}>0$, without loss of generality; else, there exists at least one $x'\in[x]_{\sigma(R)}$ such that $r_{xx'}>0$, in which case $r_{x'x'}>0$ and $R_{x'}([x']_{\sigma(R)})=R_x([x]_{\sigma(R)})$. Let $y\notin[x]_{\sigma(R)}$ be such that $r_{xy}>0$. It follows from \eqref{eq:connections:SR->T} with $z=x$ that $r_{yx}=r_{xx}$. Take $z\in\mathbb{X}$. Then, $(i)$ if $z\in[x]_{\sigma(R)}$, we have $r_{xy}=r_{zy}>0$, so $r_{yz}=r_{xz}$ from \eqref{eq:connections:SR->T}; $(ii)$ if $z\notin[x]_{\sigma(R)}$ and $r_{xz}=0$, we have $r_{yz}=0$ from \eqref{eq:connections:SR->T}; $(iii)$ if $z\notin[x]_{\sigma(R)}$ and $r_{xz}>0$, we have $r_{yx}r_{xz}=r_{yz}r_{zx}$ from \eqref{eq:connections:SR->T}, and $r_{yx}=r_{xx}$ and, analogously, $r_{zx}=r_{xx}$ from before, so $r_{yz}=r_{xz}$. As a result $R_y=R_x$, which implies that $y\in[x]_{\sigma(R)}$, absurd. Therefore, $r_{xy}=0$ and, ultimately, $R_x([x]_{\sigma(R)})=1$.
%% Note that it is possible for general $\nu$ on $\mathbb{X}$ that $R$ satisfies \eqref{condition:reversible} but is not trivial, e.g., $\nu(\{2,3\})=0$ and $R=\begin{array}{[ccc]} \frac{1}{3} & \frac{1}{3} & \frac{1}{3} \\ 0 & \frac{1}{3} & \frac{2}{3} \\ 0 & \frac{2}{3} & \frac{1}{3}\end{array}, but is ok if $\mbox{supp}(R_x)\subseteq\mbox{supp}(\nu)$ for $x\in\mbox{supp}(\nu)$.

For $x\in\mathbb{X}$ such that $R_x([x]_{\sigma(R)})=0$ and $y\notin[x]_{\sigma(R)}$ with $r_{xy}>0$, we cannot exclude the possibility of $r_{yy}\neq r_{xy}$ to get a contradiction. For example, if $\mathbb{X}=\{1,2,3,4\}$, then
\[R=\left[\begin{array}{cccc} 0 & \frac{1}{3} & \frac{1}{3} & \frac{1}{3} \\ 0 & 1 & 0 & 0 \\ 0 & 0 & \frac{1}{2} & \frac{1}{2} \\ 0 & 0 & \frac{1}{2} & \frac{1}{2} \end{array}\right]\]
is easily seen to satisfy \eqref{eq:connections:SR->T}. However, $R_1([1]_{\sigma(R)})=r_{11}=0$, which ultimately means that \eqref{condition:reversible+} is not a sufficient condition for \eqref{condition:total}, and thus not for \eqref{condition:proper} w.r.t. $\sigma(R)$.
\end{example}

We now give a simple counterexample to the problem posed in Remark \ref{result:prelim:kernels:rcd:remark} of whether \eqref{condition:stationarity} can be recovered from \eqref{condition:proper} w.r.t. $\sigma(R)$ alone.

\begin{example}\label{example:connections:three_colors}
Let $\mathbb{X}=\{1,2,3\}$. Suppose that we have a probability kernel $R$ on $\mathbb{X}$, given in matrix form by
\[R=\left[\begin{array}{ccc} 1 & 0 & 0 \\ 0 & \frac{1}{2} & \frac{1}{2} \\ 0 & \frac{1}{2} & \frac{1}{2} \end{array}\right].\]
Then the $\sigma(R)$-atoms are $\{1\}$ and $\{2,3\}$, and $R_x([x]_{\sigma(R)})=1$, for all $x\in\mathbb{X}$; thus, \eqref{condition:proper} holds w.r.t. $\sigma(R)$. On the other hand, it is easily checked that $R$ will not satisfy \eqref{condition:stationarity} unless $\nu(\{2\})=\nu(\{3\})$.
\end{example}

Regarding the assumption in Theorem \ref{result:prelim:kernels:rcd} that $\mathcal{G}=\sigma(R_{|\mathcal{G}})$ a.e.$[\nu]$, note that its role in the proof of Theorem \ref{result:prelim:kernels:rcd} is to resolve some measurability issues, and is ultimately needed for the fact that $R_x([x]_\mathcal{G})=R_x([x]_{\sigma(R_{|\mathcal{G}})})$, for $\nu$-a.e. $x$. However, it is more natural to assume that $x\mapsto R_x(B)$ is $\mathcal{G}$-measurable, for $B\in\mathcal{G}$, so we may ask whether the weaker condition $\sigma(R_{|\mathcal{G}})\subseteq\mathcal{G}$ a.e.$[\nu]$ or, equivalently, that $R_{|\mathcal{G}}$ is constant on $\nu$-a.e. atom of $\mathcal{G}$, is sufficient for \eqref{condition:proper} along with \eqref{condition:stationarity1} and \eqref{condition:reversible1}. The following example shows that this is not the case even under the stronger conditions $\sigma(R)\subseteq\mathcal{G}$, \eqref{condition:stationarity}, and \eqref{condition:reversible+}. 

\begin{example}\label{example:connections:improper}
Let $(\mathbb{X},\mathcal{X})=(\mathbb{R},\mathcal{B}(\mathbb{R}))$, and $\pi$ be a \textit{diffuse} probability measure on $\mathbb{X}$, i.e. $\pi(\{x\})=0$, for all $x\in\mathbb{X}$. Let us define
\[\nu:=\frac{1}{2}(\pi+\delta_0),\qquad\mathcal{G}:=\mathcal{X},\qquad\mbox{and}\qquad R_x:=\nu,\quad\mbox{for }x\in\mathbb{X}.\]
%% Obviously, $[x]_\mathcal{G}=\{x\}\in\mathcal{G}$, for $x\neq 0$, as $\nu(\{x\})=0$. Since all $\{[x]_\mathcal{G}\}_{x\in\mathbb{X}}$ form a partition of $\mathbb{X}$, then $[0]_\mathcal{G}=\{0\}$.
Then $\mathcal{G}$ is c.g., $\sigma(R)=\{\emptyset,\mathbb{X}\}\subseteq\mathcal{G}$, and $[x]_\mathcal{G}=\{x\}\subseteq\mathbb{X}=[x]_{\sigma(R)}$, for all $x\in\mathbb{X}$. Moreover, $R$ satisfies \eqref{condition:stationarity}, \eqref{condition:reversible+}, and $R_x([x]_{\sigma(R)})=1$, for all $x\in\mathbb{X}$. However, $R_0([0]_\mathcal{G})=1/2$, so $R$ does not satisfy \eqref{condition:proper} w.r.t. $\mathcal{G}$.
\end{example}

Let us now consider the problem of determining, in terms of the properties \eqref{condition:proper}-\eqref{condition:reversible}, when a probability kernel $R$ on $\mathbb{X}$ is also an r.c.d. for $\nu$ given $\mathcal{G}$. Recall from \eqref{condition:cg_under} that \eqref{condition:proper} is a necessary condition when $\mathcal{G}$ is c.g. under $\nu$. In fact, it is not difficult to show that \eqref{condition:proper} becomes sufficient if, in addition, $R$ is stationary and $\mathcal{G}$-measurable; see also \cite[][p.\,741]{blackwell1975}, \cite[][Lemma 1]{berti2013}, \cite[][Proposition 5.19]{einsielder2011}.

\begin{proposition}\label{result:prelim:kernels:total-rcd}
Let $R$ be a probability kernel on $\mathbb{X}$. Then $R$ satisfies \eqref{condition:proper}, \eqref{condition:stationarity}, and $\sigma(R)\subseteq\mathcal{G}$ if and only if $R(\cdot)=\nu(\cdot\mid\mathcal{G})$ and $\mathcal{G}$ is c.g. under $\nu$.
\end{proposition}

Together, Theorem \ref{result:prelim:kernels:rcd}, Remark \ref{result:prelim:kernels:rcd:remark}, and Proposition \ref{result:prelim:kernels:total-rcd} imply the less obvious fact (see also Remark \ref{result:prelim:kernels:s+r-rcd:remark}) that \eqref{condition:stationarity} and \eqref{condition:reversible} are sufficient conditions for $R$ to be an almost everywhere proper r.c.d., and thus answer a question raised by \citet[][p.\,11]{berti2023kernel}. Necessity follows from standard results on conditional expectations.

\begin{corollary}\label{result:prelim:kernels:s+r-rcd}
A probability kernel $R$ on $\mathbb{X}$ is an almost everywhere proper r.c.d. for $\nu$ if and only if $R$ satisfies \eqref{condition:stationarity} and \eqref{condition:reversible}
\end{corollary}

\begin{example}\label{result:prelim:kernels:s+r-rcd:example}
Let $\mathbb{X}$ be countable, and $\nu(\{x\})>0$ for all $x\in\mathbb{X}$. In this case, probability kernels can be represented as stochastic matrices, $R=[r_{xy}]_{x,y\in\mathbb{X}}$, so that conditions \eqref{condition:stationarity} and \eqref{condition:reversible} imply that $\nu R=\nu$ and $R^2=R$, respectively. Since $\nu$ is an stationary distribution with full support, every state is recurrent and $\mathbb{X}$ decomposes into a disjoint union of closed classes of communication, $\mathbb{X}=\bigcup_{\alpha\in\Gamma}C_\alpha$. Let $C$ be one such class, and $x\in C$. Since $R$ is idempotent, $(r_{xy})_{y\in\mathbb{X}}$ is stationary for $R$, which implies that $R_x(C^c)=0$. But $C$ is irreducible, so $(r_{xy})_{y\in C}$ as a stationary distribution on $C$ is unique. Therefore, under a suitable permutation of states, $R$ is block-diagonal and such that the rows within each block are identical. Finally, note that $\nu$ is a convex combination of the rows of $R$, so from the design of $R$, we have $R_x(\cdot)=\nu(\cdot\mid C)$, for all $x\in C$.
\end{example}

\begin{remark}\label{result:prelim:kernels:s+r-rcd:remark}
As Example \ref{result:prelim:kernels:s+r-rcd:example} suggests, some of the results in the present section can be understood through the language of operator theory. In particular, \eqref{condition:stationarity} and \eqref{condition:reversible} can be restated as $\nu R=\nu$ (i.e., $R$ is mass-preserving) and $R^2=R$ (i.e., $R$ is a projector), respectively, so that parts of Corollary \ref{result:prelim:kernels:s+r-rcd} follow from the fact that Markov projectors are conditional expectations (see, e.g., \cite{douglas1965}, \citet[][Section II.6.10]{blackadar2006}, \cite{dodds1990}), although properness is a purely measure-theoretic concept.
\end{remark}

Finally, we state a technical lemma, which will be used later.

\begin{lemma}\label{atoms_general:mixture:lemma}
Let $R(\cdot)=\nu(\cdot\mid\mathcal{G})$ for some c.g. under $\nu$ sub-$\sigma$-algebra $\mathcal{G}\subseteq\mathcal{X}.$ Then $\sigma(R)=\mathcal{G}$ a.e.$[\nu]$, and $x\mapsto R_x([x]_\mathcal{G})$ is $\mathcal{G}$-measurable a.e.$[\nu]$.
\end{lemma}

\section{The model}\label{section:model}

\subsection{Preliminaries}\label{section:model:prelim}

A sequence $(X_n)_{n\geq 1}$ of $\mathbb{X}$-valued random variables is (infinitely) \emph{exchangeable} if, for each $n=2,3,\ldots$ and all permutations $\sigma$ of $\{1,\ldots,n\}$,
\[(X_1,\ldots,X_n)\overset{d}{=}(X_{\sigma(1)},\ldots,X_{\sigma(n)}).\]
By de Finetti's representation theorem for (infinitely) exchangeable sequences \cite[Theorem 3.1]{aldous1985exchangeability}, there exists a random probability measure $\tilde{P}$ on $\mathbb{X}$, called  the \emph{directing random measure} of the process $(X_n)_{n\geq 1}$, such that, given $\tilde{P}$, the random variables $X_1,X_2,\ldots$ are conditionally independent and identically distributed (i.i.d.) with marginal distribution $\tilde{P}$,
\begin{eqnarray}
X_n\mid\tilde{P} &\overset{i.i.d.}{\sim} & \tilde{P} \nonumber\\ \tilde{P} & \sim & Q, \nonumber
\end{eqnarray}
so modeling usually consists of choosing a \textit{prior} distribution $Q$ for $\tilde{P}$. In addition, $\tilde{P}$ is the almost sure (a.s.) weak limit of the empirical measure,
\begin{equation}\label{eq:model:prelim:empir_conv:weak}
\frac{1}{n}\sum_{i=1}^n\delta_{X_i}\overset{w}{\longrightarrow}\tilde{P}\qquad\mbox{a.s.},
\end{equation}
as $n\rightarrow\infty$. On the other hand, for every $A\in\mathcal{X}$, we have
\begin{equation}\label{eq:model:prelim:predict_mart}
\mathbb{P}(X_{n+1}\in A|X_1,\ldots,X_n)=\mathbb{E}[\tilde{P}(A)|X_1,\ldots,X_n]\qquad\mbox{a.s.},
\end{equation}
implying that the predictive distributions form a Doob martingale w.r.t. the directing random measure and the natural filtration of $(X_n)_{n\geq 1}$. Then, as $n\rightarrow\infty$,
\begin{equation}\label{eq:model:prelim:predict_conv}
\mathbb{P}(X_{n+1}\in A|X_1,\ldots,X_n)\overset{a.s.}{\longrightarrow}\tilde{P}(A),
\end{equation}
and, by monotone class and separability arguments,
\begin{equation}\label{eq:model:prelim:predict_conv:weak}
\mathbb{P}(X_{n+1} \in \cdot \mid X_1, \ldots, X_{n})\overset{w}{\longrightarrow}\tilde{P}(\cdot)\qquad\text{a.s.}
\end{equation}
Thus, in principle, we should be able to recover the prior distribution from \eqref{eq:model:prelim:predict_conv:weak} when choosing to model the process directly through its predictive distributions. Moreover, one can perform posterior analysis on $\tilde{P}$ using as input $\mathbb{P}(X_{n+1}\in\cdot\mid X_1,\ldots,X_n)$; see \cite[][Section 2.4]{fortini2025}. Such a predictive approach to Bayesian nonparametric modeling is deeply rooted in the philosophical foundations of Bayesian analysis and has recently enjoyed renewed interest; see, e.g., \cite{berti2025,fong2023,fortini2025}. Central to this approach is the following result, which provides necessary and sufficient conditions for a system of predictive distributions to be consistent with exchangeability.

\begin{theorem}[Theorem 3.1 and Proposition 3.2 in \cite{fortini2000exchangeability}]\label{result:prelim:predictive_cond}
A sequence $(X_n)_{n\geq1}$ of $\mathbb{X}$-valued random variables is exchangeable if and only if, for each $n=0,1,2,\ldots$ and every $A,B\in\mathcal{X}$,
\begin{equation}\label{eq:model:prelim:predict:cond1}
\mathbb{P}(X_{n+1}\in A,X_{n+2}\in B|X_1,\ldots,X_n)=\mathbb{P}(X_{n+1}\in B,X_{n+2}\in A|X_1,\ldots,X_n)\quad\mbox{a.s.},
\end{equation}
and
\begin{equation}\label{eq:model:prelim:predict:cond2}
\mathbb{P}(X_{n+1}\in A|X_1=x_1,\ldots,X_n=x_n)=\mathbb{P}(X_{n+1}\in A|X_1=x_{\sigma(1)},\ldots,X_n=x_{\sigma(n)}),
\end{equation}
for all permutations $\sigma$ of $\{1,\ldots,n\}$ and a.e. $(x_1,\ldots,x_n)\in\mathbb{X}^n$ w.r.t. the marginal distribution of $(X_1,\ldots,X_n)$, where the case $n=0$ is meant as an unconditional statement.
\end{theorem}

\subsection{Exchangeable MVPS}\label{section:model:mvps}

A sequence $(X_n)_{n\geq 1}$ of $\mathbb{X}$-valued random variables on $(\Omega,\mathcal{H},\mathbb{P})$ is called a \textit{measure-valued \Polya sequence} with parameters $\theta$, $\nu$ and $R$, denoted MVPS($\theta,\nu,R$), if $X_1\sim\nu$ and, for each $n=1,2,\ldots$,
\begin{equation}\label{model:mvps:predictive}
\mathbb{P}(X_{n+1}\in \cdot\mid X_1, \ldots X_n) = \frac{\theta\nu(\cdot)+\sum_{i=1}^nR_{X_i}(\cdot)}{\theta+\sum_{i=1}^nR_{X_i}(\mathbb{X})},
\end{equation}
where $\theta>0$, $\nu$ is a probability measure on $\mathbb{X}$, and $R$ a finite transition kernel on $\mathbb{X}$, called the reinforcement kernel. By the Ionescu-Tulcea theorem, the law of the process $(X_n)_{n\geq 1}$ is completely determined by the sequence $(\mathbb{P}(X_{n+1}\in\cdot\mid X_1,\ldots,X_n))_{n\geq0}$. When $R_x(\mathbb{X})=m$ for some $m>0$ and $\nu$-a.e. $x$, the MVPS is said to be \emph{balanced}, which in the urn analogy means that we add the same total number of balls each time. Such an assumption greatly simplifies the calculations and, as Theorem \ref{introTh_1} shows, becomes necessary under exchangeability. 

It is further possible to consider MVPSs with random reinforcement and/or ones that allow balls to be removed from the urn. In the former case, \cite[][Theorem 1.3]{janson2019} and \cite[][p.\,6]{sariev2021} show that randomly reinforced MVPSs can be regarded as deterministic MVPSs on an extended space. On the other hand, if $R$ is a signed transition kernel, then certain conditions of tenability have to be introduced to ensure that no balls are removed that do not exist; see Section \ref{section:results:null}, where we prove that reinforcement must be non-negative under exchangeability. In the sequel, all MVPSs will have a non-negative deterministic reinforcement kernel $R$, unless otherwise specified. As a new development, in Section \ref{section:results:null}, we will consider MVPSs that explicitly have a null component in the reinforcement, which we model using
\[Z:=\{x\in\mathbb{X}:R_x(\mathbb{X})=0\}.\]
If $\nu(Z)=0$, we will say that the MVPS has a strictly positive reinforcement.

The \textit{\Polya sequence} (PS) of \cite{blackwell1973ferguson}, which is a cornerstone of Bayesian nonparametric analysis, is an example of an MVPS$(\theta,\nu,R)$ with reinforcement kernel $R_x=\delta_x$. By Theorem 1 in \cite{blackwell1973ferguson}, any PS is exchangeable and its directing random measure $\tilde{P}$ has a Dirichlet process (DP) prior distribution with parameters $(\theta,\nu)$, denoted $\tilde{P}\sim\mathrm{DP}(\theta,\nu)$. Equivalently (see, e.g., Theorem 4.12 in \cite{ghosal2017}), $\tilde{P}$ is an a.s. discrete random probability measure with so-called stick-breaking weights,
\begin{equation}\label{Sethuraman}
\tilde{P}(\cdot)\overset{w}{=}\sum_{j\geq1} V_j\delta_{U_j}(\cdot),
\end{equation}
%% is there a clash between $\overset{w}{=}$ and $\overset{w}{\longrightarrow}$
where $V_1=W_1$ and $V_j=W_j\prod_{i=1}^{j-1}(1-W_i)$, for $j\geq2$, with $W_1,W_2,\ldots\overset{i.i.d.}{\sim}\textnormal{Beta}(1,\theta)$, and $U_1,U_2,\ldots\overset{i.i.d.}{\sim}\nu$ are independent of $(V_j)_{j\geq1}$.

We focus our study on the class of exchangeable MVPSs, viewed as an extension of the basic PS, though in Section \ref{section:results:beyond} we discuss model specifications that go beyond exchangeability. First, observe that the predictive distributions \eqref{model:mvps:predictive} are invariant under arbitrary permutations of the past observations, so that \eqref{eq:model:prelim:predict:cond2} is always true for MVPSs. Therefore, an MVPS will be exchangeable if and only if it satisfies the two-step-ahead invariance condition \eqref{eq:model:prelim:predict:cond1}. Then it is not hard to check that any MVPS, where $R(\cdot)=\nu(\cdot\mid\mathcal{G})$ is an r.c.d. for $\nu$ given some sub-$\sigma$-algebra $\mathcal{G}\subseteq\mathcal{X},$ is exchangeable; see Lemma 6 and Theorem 7 in \cite{berti2023kernel}. The converse result, though less obvious (see \citet[p.\,11,\,18]{berti2023kernel}), is also true. Indeed, \cite{sariev2024} prove that the reinforcement kernel of any exchangeable MVPS with strictly positive reinforcement is, up to a constant rescaling, an r.c.d. for $\nu$ given some sub-$\sigma$-algebra. To that end, they first show that exchangeable MVPSs are necessarily balanced. Then, under balancedness, they derive the identities $(i)$ $\int_AR_x(B)\nu(dx)=\int_BR_x(A)\nu(dx)$, and $(ii)$ $\int_AR_y(B)R_x(dy)=\int_BR_y(A)R_x(dy)$, for all $A,B\in\mathcal{X}$ and $\nu$-a.e. $x$, which are then used to construct the required $\sigma$-algebra. These results are collected in Theorem \ref{introTh_1}. Notice that $(i)$ and $(ii)$ are stronger than \eqref{condition:stationarity} and \eqref{condition:reversible}, respectively, so the last part of their argument can be simplified by exploiting the ideas behind Corollary \ref{result:prelim:kernels:s+r-rcd}.

\begin{theorem}[Proposition 3.1, Theorems 3.2 and 3.7, and Remark 4.1 in \cite{sariev2024}, Theorem 7 in \cite{berti2023kernel}]\label{introTh_1}
Let $(X_n)_{n\geq 1}$ be an exchangeable MVPS($\theta,\nu,R$).
\begin{enumerate}
\item If $(X_n)_{n\geq 1}$ is not i.i.d., there exists a constant $m>0$ such that
\begin{equation}\label{R=m}
R_x(\mathbb{X})=m\qquad\mbox{for }\nu\mbox{-a.e. }x\in Z^c.
\end{equation}
\item The sequence $(X_n)_{n\geq 1}$ is i.i.d. if and only if
\[\frac{R_x(\cdot)}{R_x(\mathbb{X})}=\nu(\cdot)\qquad\mbox{for }\nu\mbox{-a.e. }x\in Z^c.\]
\item  If $\nu(Z)=0$, then there exists a c.g. under $\nu$ sub-$\sigma$-algebra $\mathcal{G}\subseteq\mathcal{X}$ such that the normalized reinforcement kernel is an r.c.d. for $\nu$ given $\mathcal{G}$,
\begin{equation}\label{introduction_thrm}
\frac{R_x(\cdot)}{R_x(\mathbb{X})}=\nu(\cdot\mid\mathcal{G})(x)\qquad\mbox{for }\nu\mbox{-a.e. }x.
\end{equation}
\end{enumerate}
Conversely, if $(X_n)_{n\geq 1}$ is an MVPS($\theta,\nu,R$) such that $R(\cdot)=\nu(\cdot\mid\mathcal{G})$, for some (not necessarily c.g. under $\nu$) sub-$\sigma$-algebra $\mathcal{G}\subseteq\mathcal{X}$, then it is exchangeable.
\end{theorem}

According to Theorem \ref{introTh_1}, every exchangeable but not i.i.d. MVPS is balanced on $Z^c$. Since any i.i.d. MVPS$(\theta,\nu,R)$ is also i.i.d. MVPS$(\theta,\nu,\nu)$, we can reparametrize every exchangeable MVPS to satisfy $R_x(\mathbb{X})=1$ for all $x\in Z^c$; see also Remark 3.3 and Corollary 3.4 in \cite{sariev2024}. Moreover, from \eqref{model:mvps:predictive} we can easily check that such a parametrization is essentially unique, so we will call it the \textit{canonical representation} of the exchangeable MVPS and denote it by MVPS*($\theta,\nu,R$).

It is also worth noting that the fact that the conditioning $\sigma$-algebra $\mathcal{G}$ in Theorem \ref{introTh_1}\textit{(iii)} is c.g. under $\nu$ plays no role in \cite{sariev2024} and is simply an artifact of their proof. Nevertheless, this property becomes essential for the results in Section \ref{section:results} as it is the structure of $\mathcal{G}$ that allows us to derive the hierarchical decomposition of $\tilde{P}$ in Theorem \ref{atoms_general:mixture}; see also Remarks \ref{atoms_general:mixture:remark} and \ref{atoms_general:mixture:proof:remark}. 

A major consequence of Theorem \ref{introTh_1}\textit{(iii)} is that the results in \cite{berti2023kernel}, which are developed under the stronger assumption $R(\cdot)=\nu(\cdot\mid\mathcal{G})$ extend to the entire class of exchangeable MVPSs with $\nu(Z)=0$. Theorems \ref{introTh_2} and \ref{introTh_3}, and Proposition \ref{introProp_1} collect the most important facts about exchangeable MVPS with strictly positive reinforcement, providing in particular a complete description of the prior and posterior distributions, and showing that the convergence in \eqref{eq:model:prelim:predict_conv:weak} can be strengthened to convergence in total variation.

\begin{theorem}[Theorem 3.9 in \cite{sariev2024}]\label{introTh_2}
Let $(X_n)_{n\geq 1}$ be an exchangeable MVPS*($\theta,\nu,R$) with $\nu(Z)=0$ and directing random measure $\tilde{P}$. Then $\tilde{P}$ is equal in law to
\begin{equation}\label{eq:introTh_2:prior}
\tilde{P}(\cdot)\overset{w}{=}\sum_{j\geq1} V_jR_{U_j}(\cdot),
\end{equation}
where $(V_j)_{j\geq1}$ and $(U_j)_{j\geq1}$ are as in \eqref{Sethuraman}. Moreover, as $n\rightarrow\infty$,
\begin{equation}\label{eq:introTh_2:total_variation}
\sup_{A\in\mathcal{X}}\,\bigl|\mathbb{P}(X_{n+1}\in A|X_1,\ldots,X_n)-\tilde{P}(A)\bigr|\overset{a.s.}{\longrightarrow}0.
\end{equation}
\end{theorem}

Theorem \ref{atoms_general:mixture} in Section \ref{section:results:hierarchy} proves the converse statement: an exchangeable process directed by \eqref{eq:introTh_2:prior}, where $R(\cdot)=\nu(\cdot\mid\mathcal{G})$, for some c.g. under $\nu$ sub-$\sigma$-algebra $\mathcal{G}\subseteq\mathcal{X}$, has predictive structure \eqref{model:mvps:predictive}. On the other hand, \eqref{eq:introTh_2:prior} shows that, among the class of exchangeable MVPSs, the PS is the only one directed by a normalized random measure with independent increments. Moreover, $\tilde{P}$ in \eqref{eq:introTh_2:prior} may be viewed as a univariate example of a \textit{kernel stick-breaking Dirichlet process}, which was introduced by \cite{dunson2008} to model group data.

Unlike \eqref{eq:introTh_2:total_variation}, it is not necessarily true that the convergence of the empirical measure to $\tilde{P}$ in \eqref{eq:model:prelim:empir_conv:weak} can itself be extended to convergence in total variation. In fact (see Example 4 in \cite{berti2018}), for general exchangeable sequences, we will have
\[\sup_{A\in\mathcal{X}}\,\Bigl|\frac{1}{n}\sum_{i=1}^n\delta_{X_i}(A)-\tilde{P}(A)\Bigr|\overset{a.s.}{\longrightarrow}0,\qquad\mbox{as }n\rightarrow\infty,\]
if and only if $\tilde{P}$ is a.s. discrete, which in the case of an exchangeable MVPS is true if and only if $R_x$ is discrete for $\nu$-a.e. $x$. The latter fact is obtained from a combination of Theorem 10 in \cite{berti2023kernel} and Theorem \ref{introTh_1}, and is presented in the next proposition.

\begin{proposition}\label{introProp_1}
Let $(X_n)_{n\geq1}$ be an exchangeable MVPS*$(\theta,\nu,R)$ with $\nu(Z)=0$ and directing random measure $\tilde{P}$. Then $\tilde{P}$ is a.s. discrete/diffuse/absolutely continuous w.r.t. $\nu$ if and only if $R_x$ is discrete/diffuse/absolutely continuous w.r.t. $\nu$, for $\nu$-a.e. $x$.
\end{proposition}

Proposition \ref{introProp_1} further suggests that, in contrast to the PS and species sampling sequences in general, exchangeable MVPSs can be used to model continuous data, depending on the particular choice of $R$; see also \citet[p.\,3-4]{sariev2023sufficientness}. Indeed, species sampling sequences deal with categorical data by design, whereas for MVPSs with diffuse $R_x$, for example, notions like random partition and observation frequencies lose meaning at the level of the observed data $(X_1,\ldots,X_n)$ (see also Remark \ref{model:mvps:remark_sss}). Therefore, MVPSs can potentially make more efficient use of continuous data by further taking into account where each observation falls within $\mathbb{X}$.

\begin{remark}[Measure-valued species sampling sequences]\label{model:mvps:remark_sss}
Recall that a (proper) species sampling sequence \citep[][Definition 12]{pitman1996} is an exchangeable process $(X_n)_{n\geq1}$ with directing random measure given by
\begin{equation}\label{model:mvps:remark_sss:eq1}
\tilde{P}(\cdot)=\sum_{j\geq1}J_j\delta_{Z_j}(\cdot),
\end{equation}
for some sequence $(J_j)_{j\geq1}$ of non-negative random variables such that $\sum_{j\geq1}J_j=1$ a.s., and some sequence $(Z_j)_{j\geq1}$ that is i.i.d.$(\nu)$ and independent of $(J_j)_{j\geq1}$, where $\nu$ is a diffuse probability measure on $\mathbb{X}$. Equivalently \citep{hansen2000}, $(X_n)_{n\geq1}$ is characterized by predictive distributions of the form
\begin{equation}\label{model:mvps:remark_sss:eq2}
\mathbb{P}(X_{n+1}\in\cdot\mid X_1,\ldots,X_n)=\sum_{j=1}^{K_n}p_{j,n}\delta_{X_j^*}(\cdot)+r_n\nu(\cdot),
\end{equation}
for some measurable functions $p_{j,n}=p_j(n,K_n,\textbf{N}_n)$ and $r_n=r(n,K_n,\textbf{N}_n)$, which depend on the partition induced by $(X_1,\ldots,X_n)$, where $K_n$ and $(X_1^*,\ldots,X_{K_n}^*)$ denote, respectively, the number and distinct values of $(X_1,\ldots,X_n)$ in order of appearance, and $\textbf{N}_n=(N_{1,n},\ldots,N_{K_n,n})$, where $N_{j,n}$ is the number of times $X_j^*$ appears in $(X_1,\ldots,X_n)$, for $j=1,\ldots,K_n$. It follows from \cite{bacallado2017} that $(i)$ if $p_{j,n}=p(n,N_{j,n})$ and $r_n=r(n)$, then $\tilde{P}$ is a DP; $(ii)$ if $p_{j,n}=p(n,N_{j,n})$ and $r_n=r(n,K_n)$, then $\tilde{P}$ is a Pitman-Yor process; and $(iii)$ if $p_{j,n}=p(n,K_n,N_{j,n})$ and $r_n=r(n,K_n)$, then $\tilde{P}$ is a Gibbs-type process (see also \citet[][Proposition 1]{deblasi2015}).

In view of Theorems \ref{introTh_1} and \ref{introTh_2}, it is natural to ask whether, w.r.t.~the quantities in \eqref{model:mvps:remark_sss:eq1} and \eqref{model:mvps:remark_sss:eq2}, the directing measure of any exchangeable process $(X_n)_{n\geq1}$ satisfying
\begin{equation}\label{model:mvps:remark_sss:eq3}
\mathbb{P}(X_{n+1}\in\cdot\mid X_1,\ldots,X_n)=\sum_{j=1}^{K_n}p_{j,n}R_{X_j^*}(\cdot)+r_n\nu(\cdot),
\end{equation}
where $R(\cdot)=\nu(\cdot\mid\mathcal{G})$, for some sub-$\sigma$-algebra $\mathcal{G}\subseteq\mathcal{X}$, is of the form
\begin{equation}\label{model:mvps:remark_sss:eq4}
\tilde{P}(\cdot)\overset{w}{=}\sum_{j\geq1}J_jR_{Z_j}(\cdot).
\end{equation}
%Observe that $R_x$ has to be non-diffuse in order to observe ties in the data.
In the special case of an MVPS with $R(\cdot)=\nu(\cdot\mid\mathcal{G})$ or, equivalently, $p_{j,n}=p(n,N_{j,n})$ and $r_n=r(n)$, \citet[][Theorem 8]{berti2023kernel} show that $(R_{X_n})_{n\geq1}$ is itself a PS. The representation \eqref{model:mvps:remark_sss:eq4} then follows from \citet[][Theorem 9]{berti2023kernel}. Their argument, however, does not apply directly to general predictive weights, yet the structure of $R$ suggests a possible extension. Since $R(\cdot)=\nu(\cdot\mid\sigma(R))$ and $\sigma(R)$ is c.g., \eqref{condition:cg_under} implies that  $R_x([x]_{\sigma(R)})=1$, for $\nu$-a.e. $x$. Moreover, by \eqref{condition:atoms}, $R_y=R_x$, for all $y\in[x]_{\sigma(R)}$. Thus, a draw from $R_{X_j^*}$ a.s.~``adds'' another copy of the kernel $R_{X_j^*}$ to the urn. Consequently, for models in which $p_{j,n}$ and $r_n$ depend only on the partition induced by $(R_{X_1},\ldots,R_{X_n})$, and in which \eqref{model:mvps:remark_sss:eq3} is a convex combination of $\nu$ and the distinct kernels $R_1^*,\ldots,R_{\tilde{K}_n}^*$, the proof of \citet[][Theorem 8]{berti2023kernel} can be readily adapted to show that $(R_{X_n})_{n\geq1}$ is a species sampling sequence with the same predictive weights. The corresponding analogue of Theorem 9 in \cite{berti2023kernel} then requires the a.s. convergence of
\[\lim_{n\rightarrow\infty}\sum_{j=1}^{\tilde{K}_n}p_{j,n}R_j^*(A)=\int_\mathbb{X}R_x(A)\tilde{P}(dx).\]

For instance, consider the ``measure-valued'' Pitman-Yor predictive rule
\[\mathbb{P}(X_{n+1}\in\cdot\mid X_1,\ldots,X_n)=\sum_{j=1}^{\tilde{K}_n}\frac{\tilde{N}_{j,n}-\alpha}{\theta+n}R_j^*(\cdot)+\frac{\theta+\alpha\tilde{K}_n}{\theta+n}\nu(\cdot),\]
where $0\leq\alpha<1$, $\theta>-\alpha$, and $\tilde{N}_{j,n}:=\#\{i\leq n:R_{X_i}=R_j^*\}$, for $j=1,\ldots,\tilde{K}_n$. Then the predictive weights satisfy the above requirements and, by \citet[][Theorem 8]{berti2023kernel}, $(R_{X_n})_{n\geq1}$ is directed by the Pitman-Yor process. Moreover,
\[\sum_{j=1}^{\tilde{K}_n}p_{j,n}R_j^*(A)=\frac{1}{\theta+n}\sum_{i=1}^nR_{X_i}(A)-\frac{\alpha}{\theta+n}\sum_{j=1}^{\tilde{K}_n}R_j^*(A)\approx\frac{1}{n}\sum_{i=1}^nR_{X_i}(A),\]
since $\tilde{K}_n/n\overset{a.s.}{\longrightarrow}0$, as $n\rightarrow\infty$. It now follows from the law of large numbers for exchangeable sequences and \citet[][Theorem 9]{berti2023kernel} that $(X_n)_{n\geq1}$ is directed by \eqref{model:mvps:remark_sss:eq4}, where $(J_j)_{j\geq1}$ are the usual Pitman-Yor stick-breaking weights. 
%On the other hand, the question whether \eqref{model:mvps:remark_sss:eq3} is enough to show that $R(\cdot)=\nu(\cdot\mid\mathcal{G})$, for some sub-$\sigma$-algebra $\mathcal{G}\subseteq\mathcal{X}$, is addressed in Remark.
\end{remark}

Returning to exchangeable MVPSs, by combining Theorem 13 in \cite{berti2023kernel} and Theorem \ref{introTh_1}, we obtain the posterior distribution of the directing random measure of any exchangeable MVPS with strictly positive reinforcement, which enjoys a conjugacy property similar to that of the DP \citep[Theorem 4.6]{ghosal2017}.

\begin{theorem}\label{introTh_3}
Let $(X_n)_{n\geq1}$ be an exchangeable MVPS*$(\theta,\nu,R)$ with $\nu(Z)=0$ and directing random measure $\tilde{P}$. Then
\[\tilde{P}(\cdot)\mid X_1,\ldots,X_n\overset{w}{=}\sum_{j\geq1} V_j^*R_{U_j^*}(\cdot),\]
where $(V_j^*)_{j\geq1}$ and $(U_j^*)_{j\geq1}$ are as in \eqref{Sethuraman} w.r.t. the parameters $\bigl(\theta+n,\frac{\theta\nu+\sum_{i=1}^nR_{X_i}}{\theta+n}\bigr)$.
\end{theorem}

\section{Results}\label{section:results}

Section \ref{section:results} develops several extensions of the model introduced in the previous section. In Section \ref{section:results:hierarchy}, we derive a hierarchical decomposition for the directing random measure $\tilde{P}$ of any exchangeable MVPS with strictly positive reinforcement \eqref{eq:introTh_2:prior}. The result relies on the fact that the conditioning $\sigma$-algebra in the representation of the reinforcement $R$ in \eqref{introduction_thrm} is almost everywhere c.g. and uses the theory from Section \ref{section:kernel}. This decomposition provides further insight into the structure of $\tilde{P}$ and clarifies its connection to some standard Bayesian nonparametric constructions. Section \ref{section:results:null} studies a generalization of the model in Section \ref{section:model:mvps} allowing for arbitrary reinforcement kernels. Under certain tenability conditions, we show that exchangeability tolerates only the presence of a null component in $R$, and we extend the results in Sections \ref{section:model:mvps} and \ref{section:results:hierarchy} to cover this case. Finally, Section \ref{section:results:beyond} focuses on an extension of the model under conditional identity in distribution, which is a weaker condition than exchangeability. We show that the two conditions coincide for balanced MVPSs, whereas in the unbalanced case they may diverge.

\subsection{Hierarchical representation}\label{section:results:hierarchy}

The main purpose of the present section is to develop the results in Section \ref{section:model:mvps} by applying the theory from Section \ref{section:kernel} and making extensive use of the fact that conditioning sub-$\sigma$-algebra $\mathcal{G}$ is c.g. under $\nu$. In particular, we show that exchangeable processes directed by \eqref{eq:introTh_2:prior} with $R$ as in \eqref{introduction_thrm} have predictive structure \eqref{model:mvps:predictive}, by using a suitable parameterization of the $\sigma$-algebra $\mathcal{G}$. In fact, the same parameterization reveals that sampling from \eqref{eq:introTh_2:prior} is ultimately performed in two steps, and modeling essentially consists of choosing a partition of the space $\mathbb{X}$ and selecting a distribution over each set in the partition. The first proposition states that the observations of an exchangeable MVPS with strictly positive reinforcement form a PS on the atoms of $\mathcal{G}$.

\begin{proposition}\label{atoms_general}
Let $(X_n)_{n\geq 1}$ be an exchangeable MVPS*$(\theta,\nu,R)$ with strictly positive reinforcement. Take $\mathcal{G}$ to be the sub-$\sigma$-algebra in \eqref{introduction_thrm}, and define $\pi$ as in Section \ref{section:model:atoms} w.r.t. $\mathcal{G}$. Then $(\pi(X_n))_{n\geq 1}$ is a PS.
\end{proposition}

The next theorem extends the conclusions of Proposition \ref{atoms_general}, revealing the hierarchical structure behind the distributional results in Theorem \ref{introTh_2}. In particular, it shows that the directing random measure of an exchangeable MVPS with strictly positive reinforcement is determined on the atoms of the conditioning $\sigma$-algebra.

\begin{theorem}\label{atoms_general:mixture}
A sequence $(X_n)_{n\geq 1}$ of $\mathbb{X}$-valued random variables is an exchangeable MVPS with strictly positive reinforcement if and only if there exist $\theta>0$, a probability measure $\nu$ on $\mathbb{X}$, and a parameter $\pi$ taking values in some measurable space $(\Pi,\mathcal{P})$ such that $\mathcal{P}$ contains $\nu_\pi$-a.e. singleton of $\Pi$, $\sigma(\pi)$ is c.g. under $\nu$, and
\begin{eqnarray}\label{atoms_general:mixture:eq}
X_n \mid \tilde{p}_n,\tilde{Q} & \overset{ind.}{\sim} & \nu(\cdot\mid\pi=\tilde{p}_n) \nonumber\\
\tilde{p}_n\mid \tilde{Q} & \overset{i.i.d.}{\sim} & \tilde{Q} \\
\tilde{Q} & \sim & \textnormal{DP}(\theta,\nu_\pi) \nonumber
\end{eqnarray}
\end{theorem}

Similarly to Proposition \ref{atoms_general}, Theorem \ref{atoms_general:mixture} states that the directing random measure $\tilde{P}$ of any exchangeable MVPS $(X_n)_{n\geq1}$ with strictly positive reinforcement has a DP prior distribution at the level of the atoms of $\sigma(\pi)$. Within each $\sigma(\pi)$-atom, say $[x]_{\sigma(\pi)}$, $\tilde{P}$ is equal to the conditional distribution of $\nu$ given $[x]_{\sigma(\pi)}$, broadly speaking, and has full support on $[x]_{\sigma(\pi)}$, since $\sigma(\pi)$ is c.g. under $\nu$ (see Remark \ref{atoms_general:mixture:proof:remark}); thus, $X_n$ is sampled from $\nu$ on $[X_n]_{\sigma(\pi)}$, conditionally given $\pi(X_n)$. The assumption that $\sigma(\pi)$ is c.g. under $\nu$ should not be considered restrictive, as it holds, for example, when $\pi$ takes values in a standard Borel space, in which case also $\{p\}\in\mathcal{P}$ for all $p\in\Pi$.

\begin{remark}[Bayesian nonparametrics]\label{atoms_general:mixture:remark}
The hierarchical model in \eqref{atoms_general:mixture:eq} characterizes an exchangeable process directed by
\begin{equation}\label{atoms_general:mixture:remark:eq1}
\tilde{P}(\cdot)=\int_\Pi\nu(\cdot\mid\pi=p)\tilde{Q}(dp),
\end{equation}
so that $\tilde{P}$ has a DP \textit{mixture} prior distribution in the sense of \citep{lo1984}, with mixing components given by the conditional distributions of $\nu$ on the atoms $\{\pi=p\}$ of $\sigma(\pi)$. Since $\sigma(\pi)$ is c.g. under $\nu$, condition \eqref{condition:proper} and \eqref{condition:cg_under} imply
\[\nu(\pi=p|\pi=p)=1\qquad\mbox{for }\nu_\pi\mbox{-a.e. }p;\]
thus, $\{\nu(\cdot\mid\pi=p)\}_{p\in\Pi}$ are mutually singular. As a consequence, clustering is effectively support-determined: each observation belongs to the unique component whose support contains it. In fact, $\pi(X_n)\overset{a.s.}{=}\tilde p_n$ (see the proof of Theorem \ref{atoms_general:mixture}), so that the labels $\tilde{p}_1,\tilde{p}_2,\ldots$, tracking cluster membership, can be recovered from the sequence of observations through $\pi$. Moreover, at the level of the observed atoms, $\tilde{P}$ induces the usual Chinese restaurant partition structure, so that, for example, the number of observed distinct atoms grows on the order of $\log n$. When the conditional distributions $\nu(\cdot\mid\pi=p)$ are diffuse, however, these notions lose meaning at the level of the observed colors/dishes.

This behavior is different from that of standard DP mixtures, where overlapping kernels allow each observation to have non-negligible likelihood under multiple components. In that case, posterior cluster allocation can borrow information across nearby components, and the induced density can vary smoothly through their superposition. By contrast, the components in \eqref{atoms_general:mixture:remark:eq1} have disjoint supports, so cluster membership is fixed once we observe the process. Moreover, from \eqref{Sethuraman} and \eqref{atoms_general:mixture:remark:eq1}, we obtain
\[\tilde{P}(\cdot)\overset{w}{=}\sum_{j\geq1}V_j\nu(\cdot\mid\pi=p^*_j),\]
where $p_1^*,p^*_2,\ldots\overset{i.i.d.}{\sim}\nu_\pi$ are independent of $(V_j)_{j\geq1}$. Therefore, the resulting density is piece-wise, rather than a smooth mixture of overlapping kernels, and the model loses the cluster-allocation flexibility that makes standard DP mixtures effective for smooth density estimation. Nevertheless, this structure can be useful when the partition encoded by $\pi$ has an intrinsic meaning, for example when clusters correspond to distinct regimes, strata, or support regions.

Equivalently, $\tilde{P}$ can be interpreted as a random histogram (see, e.g., \citet[][Example 5.11]{ghosal2017}) whose bins are located on the atoms $\{\{\pi=p_j^*\}\}_{j\ge1}$, with probabilities given by the stick-breaking weights and shapes determined by the corresponding conditional distributions (see also Examples \ref{results:hierarchical:example:k-color} and \ref{results:hierarchical:example:histogram}). Compared with classical random histograms, the bins need not be fixed in advance or rectangular, so depending on the choice of $\nu$ and $\pi$, the model can accommodate irregular supports or problem-specific partitions. A possible extension, which we do not pursue here, is to place a prior on $\pi$ in order to randomize both the location and the shape of the bins.
\end{remark}

\begin{example}[$k$-color urns]\label{results:hierarchical:example:k-color}
When $|\mathbb{X}|=k$, MVPSs are known in the literature as \textit{generalized \Polya urn} models (GPU) \cite[p.\,5]{pemantle2007}, and $R$ is given in terms of a so-called reinforcement matrix. The classical $k$-color \Polya urn model itself corresponds to a GPU with a scalar diagonal reinforcement matrix and generates an exchangeable process with a $k$-dimensional Dirichlet distribution prior. In general, Example 3.11 in \cite{sariev2024} and Example 2 in \cite{sariev2023sufficientness} show that a GPU will be exchangeable if and only if its reinforcement matrix $R$ is block-diagonal and such that within each block $R$ is constant, equal to the conditional distribution for $\nu$ given that particular block; see also Example \ref{result:prelim:kernels:s+r-rcd:example}. In the context of Theorem \ref{atoms_general:mixture}, the latter means that if $(X_n)_{n\geq1}$ is an exchangeable GPU, then $\Pi=\{p_1,\ldots,p_m\}$, for some $1\leq m\leq k$, so that, letting $\pi(x):=p_j$ if and only if $x\in D_j$, for $j=1,\ldots,m$, we get
\[\nu(\cdot\mid\pi=p_j)=\nu(\cdot\mid D_j).\]
Moreover, $\bigl(\tilde{Q}(\{p_1\}),\ldots,\tilde{Q}(\{p_m\})\bigr)$ has a Dirichlet distribution with parameters $(\theta\nu_\pi(\{p_1\}),\ldots,\theta\nu_\pi(\{p_m\}))$, and $(X_n)_{n\geq1}$ has directing random measure
\[\tilde{P}(\cdot)=\sum_{j=1}^m\tilde{Q}(\{\pi_j\})\frac{\nu(\cdot\cap D_j)}{\nu(D_j)},\]
assuming, as usual, $\nu(D_j)>0$, for all $j=1,\ldots,m$.
\end{example}

\begin{example}[Dominated model]\label{results:hierarchical:example:histogram}
Let $(X_n)_{n\geq1}$ be an exchangeable MVPS*$(\theta,\nu,R)$ with strictly positive reinforcement such that $R_x$ is absolutely continuous w.r.t. $\nu$, for $\nu$-a.e. $x$. By Theorem 3.10 in \cite{sariev2024}, there exists a countable partition $D_1,D_2,\ldots\in\mathcal{X}$ such that
\[R_x(\cdot)=\sum_{k\geq1}\nu(\cdot\mid D_k)\cdot\mathbbm{1}_{D_k}(x)\qquad\mbox{for }\nu\mbox{-a.e. }x.\]
In particular, assuming that $\nu=\lambda$ is the Lebesgue measure on $\mathbb{X}=\mathbb{R}$, and $0<\lambda(D_k)<\infty$ for all $k\geq1$, we obtain the DP mixture
\begin{eqnarray}
X_n\mid\tilde{Q} &\overset{i.i.d.}{\sim} & \sum_{k\geq1}\tilde{Q}(D_k)\frac{\lambda(\cdot\cap D_k)}{\lambda(D_k)} \nonumber\\ \tilde{Q} & \sim & \mbox{DP}(\theta,\lambda), \nonumber
\end{eqnarray}
which corresponds to the usual random histogram model with DP-distributed weights that is commonly used in the estimation of cell probabilities \citep{lo1984}; see also Example 1 in \cite{sariev2023sufficientness}.
\end{example}

\begin{example}[Invariant Dirichlet process]\label{results:hierarchical:example:idp}
Let $\mathfrak{G}=\{g_1,\ldots,g_k\}$ be a finite group of measurable mappings on $\mathbb{X}$, $\theta>0$ a positive constant, and $\nu$ a \textit{$\mathfrak{G}$-invariant} probability measure on $\mathbb{X}$, i.e. $\nu\circ g^{-1}=\nu$ for all $g\in\mathfrak{G}$. Define by
\[\mathcal{G}:=\{A\in\mathcal{X}:A=g^{-1}(A)\mbox{ for all }g\in\mathfrak{G}\}\]
the $\sigma$-algebra of $\mathfrak{G}$-invariant subsets of $\mathbb{X}$. Then $[x]_\mathcal{G}=\{y\in\mathbb{X}:g(y)=x,g\in\mathfrak{G}\}$. It follows from \cite[][Theorem 1]{tiwari1988} that
\begin{equation}\label{results:hierarchical:example:idp:eq}
\tilde{P}(\cdot)\overset{w}{=}\sum_{j\geq1}V_j\Bigl(\frac{1}{k}\sum_{i=1}^k\delta_{g_i(U_j)}(\cdot)\Bigr)
\end{equation}
has a so-called \textit{invariant Dirichlet process} (IDP) prior distribution, where $(V_j)_{j\geq1}$ and $(U_j)_{j\geq1}$ are as in \eqref{Sethuraman}. IDPs have been introduced by \cite{dalal1979} as extensions of the basic DP to account for inherent symmetries in the data; see also Example 17 in \cite{berti2023kernel} and Section 4.6.1 in \cite{ghosal2017}. In fact, by Theorem 1 in \cite{dalal1979}, realizations of $\tilde{P}$ are a.s. $\mathfrak{G}$-invariant probability measures. On the other hand, omitting the details,
\[\nu(\cdot\mid\mathcal{G})(x)=\frac{1}{k}\sum_{i=1}^k\delta_{g_i(x)}(\cdot)\qquad\mbox{for }\nu\mbox{-a.e. }x,\]
so the exchangeable process with directing random measure \eqref{results:hierarchical:example:idp:eq} is an MVPS$(\theta,\nu,R)$ with reinforcement kernel $R(\cdot)=\nu(\cdot\mid\mathcal{G})$.
\end{example}

\begin{example}[Symmetrized Dirichlet process]
%% $\nu$ being symmetric is essential
Let $\mathbb{X}=\mathbb{R}$, and $\nu$ be a symmetric probability measure on $\mathbb{X}$. Suppose that $(X_n)_{n\geq1}$ satisfies \eqref{atoms_general:mixture:eq} w.r.t. $\pi(x):=|x|$, for $x\in\mathbb{X}$. Then $\sigma(\pi)\equiv\{A\in\mathcal{X}:A=-A\}$ is c.g., where $-A=\{-x:x\in A\}$, and $[x]_{\sigma(\pi)}=\{x,-x\}$. It follows for every $A\in\mathcal{X}$ that
\[\nu(A|\sigma(\pi))(x)=\nu(-A|\sigma(\pi))(x)\qquad\mbox{for }\nu\mbox{-a.e. }x;\]
thus, if $\tilde{P}$ is the directing random measure of $(X_n)_{n\geq1}$, then \eqref{eq:introTh_2:prior} implies that realizations of $\tilde{P}$ are a.s. symmetric distributions on $\mathbb{X}$. Moreover,
\[\nu(\cdot\mid\sigma(\pi))(x)=\frac{1}{2}\bigl(\delta_x(\cdot)+\delta_{-x}(\cdot)\bigr)\qquad\mbox{for }\nu\mbox{-a.e. }x,\]
so the above model is a particular example of an IDP, also known as a \textit{symmetrized Dirichlet process} \citep{doss1984}; see also Example 3 in \cite{berti2023kernel} and Example 5 in \cite{sariev2023sufficientness}. Similarly, one can modify the DP to pick rotationally invariant or exchangeable measures on $\mathbb{X}=\mathbb{R}^k$, \cite[Examples 4.34 and 4.35]{ghosal2017}.
\end{example}

\begin{example}
Let $\mathbb{X}=\mathbb{R}^m$, for $m\geq2$, and $\|\cdot\|$ be the Euclidean norm on $\mathbb{X}$. Suppose that $(X_n)_{n\geq1}$ satisfies \eqref{atoms_general:mixture:eq} w.r.t. $\pi(x):=\|x\|$, for $x\in\mathbb{X}$. Then $[x]_{\sigma(\pi)}=\{y\in\mathbb{X}:\|y\|=\|x\|\}$, so that each $\nu(\cdot\mid\sigma(\pi))(x)$ is supported on its own spherical surface in $\mathbb{X}$ centered at zero. Example 16 in \cite{berti2023kernel} studies the particular model 
\[\nu(\cdot)=\int_0^\infty\mathcal{U}_t(\cdot)e^{-t}dt,\]
where $\mathcal{U}_t$ is the uniform distribution on $\{\pi=t\}$, with $\mathcal{U}_0=\delta_0$. In that case,
\[\nu(\cdot\mid\sigma(\pi))(x)=\mathcal{U}_{\|x\|}(\cdot)\qquad\mbox{for }\nu\mbox{-a.e. }x,\]
so that the mixing probability distributions in \eqref{atoms_general:mixture:eq} are uniform on the particular spherical surfaces. Moreover, $\nu([x]_{\sigma(\pi)})=0$ for all $x\in\mathbb{X}$, which implies that the reinforcement $\nu(\cdot\mid\sigma(\pi))(x)$ at each $x$ and the initial measure $\nu$ are mutually singular.
\end{example}

\subsection{Exchangeable MVPSs with null part}\label{section:results:null}

A natural extension of the urn model \eqref{model:mvps:predictive} is obtained by allowing the reinforcement $R$ be a finite signed kernel, i.e., $R_x$ is a finite signed measure on $\mathbb{X}$, for all $x\in\mathbb{X}$, thereby permitting the removal of balls from the urn. In order to rule out the removal of non-existent balls, we impose the following condition of \textit{tenability}:
\begin{equation}\label{model:mvps:tenable}
\theta\nu+\sum_{i=1}^nR_{X_i}\quad\mbox{is a.s. a non-negative measure, for all }n\in\mathbb{N}.
\end{equation}
The next result states that, under exchangeability and \eqref{model:mvps:tenable}, reinforcement must be non-negative.

\begin{proposition}\label{results:null:non-negative}
Let $(X_n)_{n\geq1}$ be an MVPS$(\theta,\nu,R)$ such that $R$ is a finite signed kernel satisfying \eqref{model:mvps:tenable}. If $(X_n)_{n\geq1}$ is exchangeable, then $R_x(B)\geq0$, for all $B\in\mathcal{X}$ and $\nu$-a.e. $x$.
\end{proposition}

Note that Proposition \ref{results:null:non-negative} does not exclude the possibility of $\nu(Z)>0$, where we recall that
\[Z:=\{x\in\mathbb{X}:R_x(\mathbb{X})=0\}.\]
The introduction of $Z$ potentially allows us to account for the presence of control variables or to model situations in which we deliberately want to exclude the effect of certain observations. The next theorem, which generalizes Theorem \ref{introTh_1}\textit{(iii)}, states that the reinforcement kernel $R$ of any exchangeable MVPS is necessarily a mixture of two components, one independent of $x$ and corresponding to $\nu$ restricted to $Z$, and the other emerging from the representation of $R$ in \eqref{introduction_thrm} when $(X_n)_{n\geq 1}$ is restricted to $Z^c$.

\begin{theorem}
\label{representation}
If $(X_n)_{n\geq1}$ is an exchangeable MVPS$(\theta,\nu,R)$ such that $0<\nu(Z)<1$, then there exists a c.g. under $\nu$ sub-$\sigma$-algebra $\mathcal{G}\subseteq\mathcal{X}$ such that $Z^c\in\mathcal{G}$ and
\begin{equation} \label{R_x}
\frac{R_x(\cdot)}{R_x(\mathbb{X})}=\nu(Z^c)\nu(\cdot\mid\mathcal{G})(x)+\nu(Z)\nu(\cdot\mid Z)\qquad\mbox{for }\nu\mbox{-a.e. }x\in Z^c.
\end{equation}
Conversely, if $(X_n)_{n\geq1}$ is a balanced MVPS on $Z^c$ with reinforcement kernel \eqref{R_x} for some (not necessarily c.g. under $\nu$) sub-$\sigma$-algebra $\mathcal{G}\subseteq\mathcal{X}$ such that $Z^c\in\mathcal{G}$, then it is exchangeable.
\end{theorem}

\begin{remark}\label{results:null:remark:Z^c}
Notice that, for the reinforcement kernel in \eqref{R_x}, the assumption $Z^c\in\mathcal{G}$ implies through \eqref{condition:proper} that
\[\nu(Z^c|\mathcal{G})(x)=\delta_x(Z^c)=1,\qquad\mbox{for }\nu\mbox{-a.e. }x\in Z^c.\]
Then $\mathbb{P}(X_{n+1}\in Z^c|X_1,\ldots,X_n)=\nu(Z^c)$ a.s. on $\{X_1\in Z^c,\ldots,X_n\in Z^c\}$, so
\[\mathbb{P}(X_1\in Z^c,\ldots,X_n\in Z^c)=\bigl(\nu(Z^c)\bigr)^n.\]
On the other hand, $R_x(\cdot\cap Z)=R_x(\mathbb{X})\,\nu(\cdot\cap Z)$ for $\nu$-a.e. $x\in Z^c$, so in the urn analogy, when a color $x$ in $Z^c$ is observed, the colors in $Z$ are reinforced proportional to the amount they were initially present in the urn, which is necessary to preserve exchangeability.
\end{remark}

\begin{example}[PS with a null component]
Let $(X_n)_{n\geq1}$ be an MVPS$(\theta,\nu,R)$, where
\[R_x(\cdot)=\biggl\{\begin{array}{cl}\delta_x(\cdot) & \mbox{for }x\in Z^c,\\ 0 & \mbox{for }x\in Z,\end{array}\]
for some $Z\in\mathcal{X}$ such that $0<\nu(Z)<1$, that is, starting from the classical PS, we introduce a null component by choosing to reinforce only the colors that lie in $Z^c$. It follows from Remark \ref{results:null:remark:Z^c} that the resulting process is not exchangeable; in fact, it is an example of a \textit{dominant \Polya sequence} \citep{sariev2023}.

Now, consider instead the reinforcement kernel
\[R_x(\cdot)=\biggl\{\begin{array}{cl}\nu(Z^c)\delta_x(\cdot)+\nu(Z)\nu(\cdot\mid Z) & \mbox{for }x\in Z^c,\\ 0 & \mbox{for }x\in Z.\end{array}\]
In that case, $R$ satisfies \eqref{R_x} w.r.t. $\mathcal{G}=\mathcal{X}$,
\[\mathbb{P}(X_{n+1}\in\cdot\mid X_1,\ldots,X_n)=\nu(Z)\nu(\cdot\mid Z)+\nu(Z^c)\frac{\theta\nu(\cdot\mid Z^c)+\sum_{i=1}^n\delta_{X_i}(\cdot)\mathbbm{1}_{Z^c}(X_i)}{\theta+\sum_{i=1}^n\mathbbm{1}_{Z^c}(X_i)},\]
and $(X_n)_{n\geq1}$ is exchangeable.

From the point of view of cluster growth, the relevant quantity of interest for both models is the number of distinct values observed in $Z^c$. Under standard diffuseness assumptions, this number grows on the order of $\log n$ in both models, as in the classical PS; see Proposition 3.4 in \cite{sariev2023} and the proof of Theorem \ref{representation}. The effect of the null component on the innovation mechanism, however, is different. In both models,
\[\mathbb{P}(X_{n+1}\mbox{ is new in }Z^c|X_1,\ldots,X_n)=\frac{\theta\nu(Z^c)}{\theta+\sum_{i=1}^n\mathbbm{1}_{Z^c}(X_i)}.\]
However, in the exchangeable case, we have $\mathbb{P}(X_{n+1}\in Z^c|X_1,\ldots,X_n)=\nu(Z^c)$, so $\sum_{i=1}^n\mathbbm{1}_{Z^c}(X_i)/n\rightarrow\nu(Z^c)$ a.s., as $n\rightarrow\infty$, and, by the conditional Borel-Cantelli lemma, the number of distinct colors in $Z^c$ behaves like $\theta\log(\nu(Z^c)n)$. By contrast, in the dominant case, $\sum_{i=1}^n\mathbbm{1}_{Z^c}(X_i)/n\rightarrow1$ a.s., as $n\rightarrow\infty$, so the number of distinct colors in $Z^c$ behaves like $\theta\nu(Z^c)\log n$. Thus, in the exchangeable model, the null component reduces the number of visits to $Z^c$ and slows the reinforcement effect, producing a downward shift in the number of innovations at finite sample sizes. In the dominant model, almost all observations eventually fall in $Z^c$, and the reinforcement-driven part takes over more rapidly. Nevertheless, the presence of a null component forces innovations to occur relatively less often in the long run.
\end{example}

From Theorem \ref{representation}, we obtain the following extension of Proposition \ref{atoms_general} for the case when $0<\nu(Z)<1$.

\begin{corollary}\label{atoms_null_case}
Let $(X_n)_{n\geq 1}$ be an exchangeable MVPS*$(\theta,\nu,R)$ such that $0<\nu(Z)<1$. Take $\mathcal{G}$ to be the sub-$\sigma$-algebra in \eqref{R_x}, and define $\pi$ as in Section \ref{section:model:atoms} w.r.t. $\mathcal{G}$. Then $(\pi(X_n))_{n\geq 1}$ is an exchangeable MVPS$(\theta,\nu_\pi,R_\pi)$, where
\begin{equation}\label{atoms_null_case:eq}
(R_\pi)_p(\cdot)=\begin{cases}\nu(Z^c)\delta_p(\cdot)+\nu(Z)\nu_\pi(\cdot\mid\pi(Z)) \hspace{15pt}&\text{if $p\in \pi(Z^c)$},\\ 0&\text{if $p\in \pi(Z)$}.\end{cases}
\end{equation}
\end{corollary}

Let us now consider the form of the directing random measure of any exchangeable MVPS having a null reinforcement component. The next theorem is a straightforward extension of the results in Theorems \ref{introTh_2} and \ref{atoms_general:mixture}.

\begin{theorem}\label{directing_rm}
Let $(X_n)_{n\geq 1}$ be an exchangeable MVPS*($\theta,\nu,R$) with $0<\nu(Z)<1$ and directing random measure $\tilde{P}$.
\begin{enumerate}
\item Then, as $n\rightarrow\infty$,
\[\sup_{A\in\mathcal{X}}\,\bigl|\mathbb{P}(X_{n+1}\in A|X_1,\ldots,X_n)-\tilde{P}(A)\bigr|\overset{a.s.}{\longrightarrow}0.\]
Moreover, $\tilde{P}$ has the form
\[\tilde{P}(\cdot)=\nu(Z^c)\tilde{P}(\cdot\mid Z^c)+\nu(Z)\nu(\cdot \mid Z)\qquad\mbox{a.s.},\]
where
\[\tilde{P}(\cdot \mid Z^c)\overset{w}{=}\sum_{j\geq1}V_jR_{U_j}(\cdot \mid Z^c),\]
with $(V_j)_{j\geq 1}$ and $(U_j)_{j\geq 1}$ as in \eqref{Sethuraman} w.r.t. the parameters $(\theta,\nu(\cdot\mid Z^c))$.\vspace{0.2cm}
\item There exists some parameter $\pi$ such that
% if and only if statement?
\begin{eqnarray}
%\label{directing_rm:mixture:eq}
X_n \mid \tilde{p}_n,\xi_n,\tilde{Q} & \overset{ind.}{\sim} & \begin{cases} \nu(\cdot\mid \pi=\tilde{p}_n) & \mbox{if }\xi_n=1, \nonumber\\ \nu(\cdot\mid Z) & \mbox{if }\xi_n=0, \end{cases}\\
(\tilde{p}_n,\xi_n)\mid \tilde{Q} & \overset{i.i.d.}{\sim} & \tilde{Q}\times\textnormal{Ber}(\nu(Z^c)) \nonumber\\
\tilde{Q} & \sim & \textnormal{DP}\bigl(\theta,\nu_\pi\bigl(\cdot\mid\pi(Z^c)\bigr)\bigr) \nonumber
\end{eqnarray}
\end{enumerate}
\end{theorem}

By Theorem \ref{directing_rm}, the directing random measure $\tilde{P}$ of any exchangeable MVPS $(X_n)_{n\geq1}$ with $\nu(Z)>0$ is a mixture of two components with disjoint supports, a DP mixture component on $Z^c$ and the deterministic probability measure $\nu(\cdot\mid Z)$. Therefore, given $\tilde{P}$, we draw each $X_n$ by first flipping a coin with ``probability of success'' $\nu(Z^c)$ to decide whether to choose $X_n$ from $\tilde{P}(\cdot\mid Z^c)$ or, alternatively, from $\nu(\cdot\mid Z)$. With respect to the random histogram interpretation in Remark \ref{atoms_general:mixture:remark}, the introduction of a null part implies that the histogram will have a bin over $Z$ with fixed bin weight $\nu(Z)$.

\subsection{Conditionally identically distributed MVPSs}\label{section:results:beyond}

Here, unlike Sections \ref{section:results:hierarchy} and \ref{section:results:null}, we consider non-exchangeable MVPSs $(X_n)_{n\geq1}$. In fact, by Corollary \ref{result:prelim:kernels:s+r-rcd} and Theorem \ref{introTh_1},
\begin{equation}\label{beyond:eq1}
(X_n)_{n\geq1}\mbox{ is exchangeable }\qquad\Longleftrightarrow\qquad R\mbox{ satisfies }\eqref{condition:stationarity}\mbox{ and }\eqref{condition:reversible},
\end{equation}
so that exchangeability is an emergent property when $R$ is (self-)averaging. On the other hand, by Proposition 2.1 in \cite{kallenberg1988}, a stochastic process $(Y_n)_{n\geq1}$ is exchangeable if and only if it is both stationary and \textit{conditionally identically distributed} (c.i.d.), i.e., for each $n=0,1,\ldots$,
\[(Y_1,\ldots,Y_n,Y_{n+1})\overset{d}{=}(Y_1,\ldots,Y_n,Y_{n+2}).\]
By \citet[][eq.\,(5)]{berti2004limit}, the latter is equivalent to $(\mathbb{P}(Y_{n+1}\in A|Y_1,\ldots Y_n))_{n\geq0}$ being a martingale, for all $A\in\mathcal{X}$, so there exists a random probability measure $\tilde{P}$, known again as the \textit{directing random measure} of $(Y_n)_{n\geq1}$, which is the limit of its predictive distributions. It follows from this result and Lemma 8.2 in \cite{aldous1985exchangeability} that $(Y_n)_{n\geq1}$ is asymptotically exchangeable with limit directing random measure $\tilde{P}$. Thus, conditional identity in distributions permits temporal asymmetries and path-dependent effects, although the process achieves stability in the long-run (see, e.g., the class of generalized species sampling sequences by \cite{bassetti2010} and Example \ref{example:other:cid}). Since c.i.d. processes retain part of the predictive structure associated with exchangeability, it is natural to investigate how the the reinforcement kernel of an MVPS is affected under this weaker assumption. The following result addresses the case of a balanced MVPS.

\begin{proposition}\label{results:other:cid}
Let $(X_n)_{n\geq1}$ be a balanced MVPS$(\theta,\nu,R)$. Then $(X_n)_{n\geq1}$ is c.i.d. if and only if $R$ satisfies \eqref{condition:stationarity} and \eqref{condition:reversible}.
\end{proposition}

Proposition \ref{results:other:cid} implies through \eqref{beyond:eq1} that, for a balanced MVPS $(X_n)_{n\geq1}$,
\[(X_n)_{n\geq1}\mbox{ is exchangeable }\qquad\Longleftrightarrow\qquad (X_n)_{n\geq1}\mbox{ is c.i.d.}\]
Therefore, \eqref{model:mvps:predictive} imposes structural constraints on the reinforcement kernel that rule out the type of temporal distortions, which are otherwise compatible with the conditional identity in distribution. Consequently, when $R$ is balanced, stationarity follows from both the predictive structure \eqref{model:mvps:predictive} and the c.i.d. assumption. Thus, in particular, every balanced c.i.d. GPU is exchangeable.

A more recent direction of research in Bayesian nonparametrics (see, e.g., \cite{berti2021,fong2023,battiston2025,fortini2025}), looks at recursive predictive constructions of probability laws that give rise to c.i.d. processes, attempting to model situations where exchangeability is violated due to innate asymmetries, forms of selection and competition, or the presence of temporary disequilibrium; see also \cite{bassetti2010,sariev2023}. From this perspective, MVPSs form a particularly tractable class of predictive schemes that can be viewed as marking a natural boundary between genuinely non-exchangeable c.i.d. behavior and exchangeability. For balanced MVPS, letting $P_n(\cdot)=\mathbb{P}(X_{n+1}\in\cdot\mid X_1,\ldots,X_n)$, we get
\begin{equation}\label{beyond:eq2}
P_n(\cdot)=\frac{\theta+n-1}{\theta+n}P_{n-1}(\cdot)+\frac{1}{\theta+n}R_{X_n}(\cdot).
\end{equation}
As an extension, \cite{berti2021} consider the system of predictive distributions
\begin{equation}\label{beyond:eq3}
P_n(\cdot)=q_nP_{n-1}(\cdot)+(1-q_n)R_{X_n}(\cdot),
\end{equation}
where $q_n:\mathbb{X}^n\rightarrow[0,1]$ is an $\mathcal{X}^n$-measurable function, and $R$ a probability kernel on $\mathbb{X}$. By Theorem 5 in \cite{berti2021}, if $R$ satisfies \eqref{condition:stationarity} w.r.t. $\nu=P_0$ and \eqref{condition:reversible} (or, equivalently, $R$ is an r.c.d. for $\nu$, by Corollary \ref{result:prelim:kernels:s+r-rcd}), then \eqref{beyond:eq3} generates a c.i.d. process. A natural question, which we do not pursue here, is to identify conditions on $q_n$ that lead to an exchangeable sequence. By Proposition \ref{results:other:cid} and \eqref{beyond:eq2}, the choice $q_n=1-1/(\theta+n)$ is sufficient, so a complete answer would clarify what aspects of MVPS structure are responsible for making c.i.d. MVPSs exchangeable.
%Such a result can be further framed in terms of a ``sufficientness'' postulate, for which relevant works are \cite{zabell1997,hansen2000,sariev2023sufficientness}.

Returning to c.i.d. MVPSs with a general reinforcement kernel $R$, the following simple example shows that it is still possible to have an unbalanced MVPS that is c.i.d. but not exchangeable.

\begin{example}\label{example:other:cid}
Let $(X_n)_{n\geq1}$ be an MVPS$(1,\nu,R)$ on $\mathbb{X}=\{1,\ldots,4\}$, where $\nu=(\nu_1,\nu_2,\nu_1,\nu_2)$, for some $\nu_1,\nu_2\in(0,1)$, and
\[R=\left[\begin{array}{cccc} \nu_1 & \nu_2 & 0 & 0 \\ 2\nu_1 & 2\nu_2 & 0 & 0 \\ 0 & 0 & \nu_1 & \nu_2 \\ 0 & 0 & 2\nu_1 & 2\nu_2 \end{array}\right].\]
Fix $n\in\{0,1,\ldots\}$. Define
\[T_n^{(1)}:=\sum_{i=1}^nR_{X_i}(\mathbb{X})\cdot\mathbbm{1}_{\{X_i=1,2\}},\qquad T_n^{(2)}:=\sum_{i=1}^nR_{X_i}(\mathbb{X})\cdot\mathbbm{1}_{\{X_i=3,4\}},\qquad D_n:=T_n^{(1)}+T_n^{(2)}.\]
Notice that $R_x(\{y\})=2\nu_yR_x(\mathbb{X})\cdot\mathbbm{1}_{\{x=1,2\}}$ for $y=1,2$, and $R_x(\{y\})=2\nu_yR_x(\mathbb{X})\cdot\mathbbm{1}_{\{x=3,4\}}$ for $y=3,4$. Then
\begingroup\allowdisplaybreaks
\begin{align*}
\mathbb{E}[P_{n+1}(\{1\})|X_1,\ldots,X_n]&=\sum_{x=1}^4\frac{\nu_1+\sum_{i=1}^nR_{X_i}(\{1\})+R_x(\{1\})}{1+\sum_{i=1}^nR_{X_i}(\mathbb{X})+R_x(\mathbb{X})}P_n(\{x\})\\
&\begin{aligned}\;=\biggl(\frac{\nu_1+2\nu_1T_n^{(1)}+\nu_1}{1+D_n+\frac{1}{2}}&\frac{\nu_1+2\nu_1T_n^{(1)}}{1+D_n}+\frac{\nu_1+2\nu_1T_n^{(1)}+2\nu_1}{1+D_n+1}\frac{\nu_2+2\nu_2T_n^{(1)}}{1+D_n}\\
&+\frac{\nu_1+2\nu_1T_n^{(1)}}{1+D_n+\frac{1}{2}}\frac{\nu_1+2\nu_1T_n^{(2)}}{1+D_n}+\frac{\nu_1+2\nu_1T_n^{(1)}}{1+D_n+1}\frac{\nu_2+2\nu_2T_n^{(2)}}{1+D_n}\biggr)\end{aligned}\\
&\begin{aligned}\;=\frac{\nu_1+2\nu_1T_n^{(1)}}{1+D_n}\biggl(\nu_1\frac{1+2T_n^{(1)}+1}{1+D_n+\frac{1}{2}}&+\nu_2\frac{1+2T_n^{(1)}+2}{1+D_n+1}\\
&+\nu_1\frac{1+2T_n^{(2)}}{1+D_n+\frac{1}{2}}+\nu_2\frac{1+2T_n^{(2)}}{1+D_n+1}\biggr)\end{aligned}\\
&=P_n(\{1\})\biggl(\nu_1\frac{2+2D_n+1}{1+D_n+\frac{1}{2}}+\nu_2\frac{2+2D_n+2}{1+D_n+1}\biggr)\\
&=P_n(\{1\}),
\end{align*}
\endgroup
where we have used that $\nu_1+\nu_2=\frac{1}{2}$. Similarly, $\mathbb{E}[P_{n+1}(\{x\})|X_1,\ldots,X_n]=P_n(\{x\})$, for all $x\in\mathbb{X}$. Therefore, $(P_n(A))_{n\geq0}$ is a martingale, for all $A\subseteq\mathbb{X}$, which implies that $(X_n)_{n\geq1}$ is c.i.d. But $R_1(\mathbb{X})\neq R_2(\mathbb{X})$, so the model is unbalanced and, by Theorem \ref{introTh_1}, $(X_n)_{n\geq1}$ is not exchangeable.
\end{example}

The next result shows that the particular block-diagonal form of the reinforcement kernel $R$ in Example \ref{example:other:cid} is not accidental, but is the only one that allows a c.i.d. MVPS to be unbalanced on a finite state space.

\begin{theorem}\label{results:other:cid-balanced}
Let $\mathbb{X}$ be finite, and $(X_n)_{n\geq1}$ an MVPS$(\theta,\nu,R)$ with strictly positive reinforcement such that $\nu(\{x\})>0$, for all $x\in\mathbb{X}$. Then $(X_n)_{n\geq1}$ is c.i.d. if and only if there exists a partition $B_1,\ldots,B_m$ of $\mathbb{X}$, for some $1\leq m\leq|\mathbb{X}|$, such that
\begin{enumerate}
\item for each $j=1,\ldots,m$ and all $x\in B_j$,
\[\frac{R_x(\cdot)}{R_x(\mathbb{X})}=\nu(\cdot\mid B_j);\]
\item for each $j=1,\ldots,m$ and all $a\in(0,\infty)$,
\[\nu(R(\mathbb{X})=a|B_j)=\nu(R(\mathbb{X})=a).\]
\end{enumerate}
\end{theorem}

Although the conclusions of Theorem \ref{results:other:cid-balanced} are comparable to those of Theorem \ref{introTh_1}(\textit{2.}), the arguments leading to them are substantially different. In particular, by applying the maximal Azuma-Hoeffding inequality, we show that the support of the law of the directing random measure of any finite c.i.d. MVPS is convex and that its extreme points are the normalized kernels $R_x$. Using this fact, we show that, up to a constant rescaling, $R$ satisfies \eqref{condition:stationarity} and \eqref{condition:reversible}, from which we derive its structural form, see also Example \ref{result:prelim:kernels:s+r-rcd:example}.

Similarly to exchangeable GPUs (Example \ref{results:hierarchical:example:k-color}), Theorem \ref{results:other:cid-balanced}(\textit{1}.) states that, after normalization, the reinforcement kernel $R$ has a block-diagonal structure, where each block is equal to the conditional probability of $\nu$, given that a color from the same block is observed. In this case, however, $R(\mathbb{X})$ need not be constant. Nevertheless, Theorem \ref{results:other:cid-balanced}(\textit{2}.) restricts its variability by requiring that, within each block, the conditional distribution of the values of $R(\mathbb{X})$ coincides with the unconditional one. In particular, in Example \ref{example:other:cid}, we have $R_1(\mathbb{X})=R_3(\mathbb{X})=\nu_1+\nu_2=\frac{1}{2}$ and $R_2(\mathbb{X})=R_4(\mathbb{X})=1$, so for all $a\in(0,\infty)$,
\begingroup\allowdisplaybreaks
\begin{align*}
\nu(R(\mathbb{X})=a|\{1,2\})&=\mathbbm{1}_{\{R_1(\mathbb{X})=a\}}\frac{\nu_1}{\nu_1+\nu_2}+\mathbbm{1}_{\{R_2(\mathbb{X})=a\}}\frac{\nu_2}{\nu_1+\nu_2}\\
&=\frac{1}{2}\mathbbm{1}_{\{R_1(\mathbb{X})=a\}}\frac{\nu_1}{\frac{1}{2}}+\frac{1}{2}\mathbbm{1}_{\{R_3(\mathbb{X})=a\}}\frac{\nu_1}{\frac{1}{2}}+\frac{1}{2}\mathbbm{1}_{\{R_2(\mathbb{X})=a\}}\frac{\nu_2}{\frac{1}{2}}+\frac{1}{2}\mathbbm{1}_{\{R_4(\mathbb{X})=a\}}\frac{\nu_2}{\frac{1}{2}}\\
&=\nu(R(\mathbb{X})=a),
\end{align*}
\endgroup
and, analogously, $\nu(R(\mathbb{X})=a|\{3,4\})=\nu(R(\mathbb{X})=a)$.

When $\mathbb{X}$ is not finite, we expect an analogue of Theorem \ref{results:other:cid-balanced} to remain valid, and its proof, although technically more involved, to follow similar lines. However, even if $R$ was shown to satisfy \eqref{condition:stationarity} and \eqref{condition:reversible}, up to a constant rescaling, it is not yet clear how to derive a counterpart of Theorem \ref{results:other:cid-balanced}\textit{(2.)}, and establishing such a result would likely require an extension of the theory developed in Section \ref{section:kernel:rcd}. A related statement from functional analysis is given, for example, by Corollary 3.7 in \cite{dodds1990}, which shows that a strictly positive, order continuous projection $R$ satisfying $R\textbf{1}>0$ is a weighted conditional expectation.

\section{Proofs}\label{section:proofs}

We denote $\mathbb{N}=\{1,2,\ldots\}$ and $\mathbb{N}_0=\{0,1,2,\ldots\}$.

\subsection{Proofs of the results in Section \ref{section:kernel:rcd}}

\begin{proof}[Proof of Theorem \ref{result:prelim:kernels:rcd}]
Suppose that $R$ satisfies \eqref{condition:stationarity1} and \eqref{condition:reversible1}, and that $\mathcal{G}\cap C_0=\sigma(R_{|\mathcal{G}})\cap C_0$ is c.g. for some $C_0\in\mathcal{G}$ such that $\nu(C_0)=1$. Define $C_n:=\{x\in C_{n-1}:R_x(C_{n-1})=1\}$, for $n\in\mathbb{N}$. Then
\[C_1=\{R(C_0)=1\}\cap C_0=\{R_{|\mathcal{G}}(C_0)\}\cap C_0\in\mathcal{G}\cap C_0\subseteq\mathcal{G},\]
and, by induction, $C_n=\{R(C_{n-1})=1\}\cap C_{n-1}\cap C_0\in\mathcal{G}$, for all $n\in\mathbb{N}$. It now follows from \eqref{condition:stationarity1} that
\begingroup\allowdisplaybreaks
\begin{align*}
1=\nu(C_0)&=\int_\mathbb{X}R_x(C_0)\nu(dx)\\
&=1-\int_{\mathbb{X}}(1-R_x(C_0))\nu(dx)=1-\int_{C_1^c}(1-R_x(C_0))\nu(dx).
\end{align*}
\endgroup
But $R_x(C_0)<1$ for $x\in C_1^c$, so $\nu(C_1^c)=0$; otherwise, the term on the right-hand side of the equation becomes strictly less than $1$. Proceeding by induction, we get $\nu(C_n)=1$ for all $n\in\mathbb{N}_0$; thus, letting $C^*:=\bigcap_{n=0}^\infty C_n$, we have $C^*\in\mathcal{G}$, $\nu(C^*)=1$, $R_x(C^*)=1$ for all $x\in C^*$, and that $\mathcal{G}\cap C^*=\sigma(R_{|\mathcal{G}})\cap C^*$ is c.g.

Let us define
\[R^*_x(B):=R_x(B),\qquad\mbox{for }B\in\mathcal{G}\cap C^*\mbox{ and }x\in C^*.\]
Then $R^*:C^*\times\mathcal{G}\cap C^*\rightarrow[0,1]$ is a probability kernel on $C^*$, $\mathcal{G}\cap C^*=\sigma(R^*)$, and since $\mathcal{G}\cap C^*$ is c.g., for all $x\in C^*$,
\[[x]_\mathcal{G}=[x]_\mathcal{G}\cap C^*=[x]_{\mathcal{G}\cap C^*}=[x]_{\sigma(R^*)}=\{y\in C^*:R^*_y=R_x^*\}\in\sigma(R^*).\]
%% Since $C^*\in\mathcal{G}$, then $[x]_\mathcal{G}=\{y\in\mathbb{X}:R^*_y\equiv R^*_x\}$, for $x\in C^*$.
Moreover, for all $A\in\mathcal{G}\cap C^*$,
\begingroup\allowdisplaybreaks
\begin{gather}\label{proof:connections:reverse<->total:eq0}
\begin{gathered}
\int_{C^*}R^*_x(A)\nu(dx)=\nu(A),\\
\int_{C^*}R_y^*(A)R_x^*(dy)=R_x^*(A),\quad\mbox{for }\nu\mbox{-a.e. }x\in C^*.
\end{gathered}
\end{gather}
\endgroup
Using again the fact that $\mathcal{G}\cap C^*$ is c.g., we obtain that, as measures on $(C^*,\mathcal{G}\cap C^*)$,
\begin{equation}\label{proof:connections:reverse<->total:eq2}
\int_{C^*}R_x^*(dy)\nu(dx)=\nu(dy).
\end{equation}
Let $A\in\mathcal{G}\cap C^*$. It follows from \eqref{proof:connections:reverse<->total:eq0} and \eqref{proof:connections:reverse<->total:eq2} that
\begingroup\allowdisplaybreaks
\begin{align*}
\int_{C^*}\biggl\{\int_{C^*}\bigl(R_y^*(A)&-R_x^*(A)\bigr)^2R_x^*(dy)\biggr\}\nu(dx)\\
&\begin{aligned}=\int_{C^*}\biggl\{\int_{C^*}\bigl(R_y^*(A)\bigr)^2R_x^*(dy)\biggr\}\nu(dx)&+\int_{C^*}\biggl\{\int_{C^*}\bigl(R_x^*(A)\bigr)^2R_x^*(dy)\biggr\}\nu(dx)\\
&-2\int_{C^*}R_x^*(A)\biggl\{\int_{C^*}R_y^*(A)R_x^*(dy)\biggr\}\nu(dx)\end{aligned}\\
&=\int_{C^*}\bigl(R_x^*(A)\bigr)^2\nu(dx)+\int_{C^*}\bigl(R_x^*(A)\bigr)^2\nu(dx)-2\int_{C^*}\bigl(R_x^*(A)\bigr)^2\nu(dx)\\
&=0.
\end{align*}
\endgroup
Therefore, $\int_{C^*}(R_y^*(A)-R_x^*(A))^2R_x^*(dy)=0$ for $\nu$-a.e. $x\in C^*$, so that $R_y^*(A)=R_x^*(A)$ for $R_x^*$-a.e. $y$. Since $\mathcal{G}\cap C^*$ is c.g., we obtain $R_x^*([x]_{\sigma(R^*)})=R_x^*(\{y\in C^*:R_y^*=R_x^*\})=1$, for $\nu$-a.e. $x\in C^*$. By a monotone class argument (see the proof of Lemma \ref{atoms_general:mixture:lemma}), $x\mapsto R_x^*([x]_{\sigma(R^*)})$ is $\mathcal{G}\cap C^*$-measurable, so $\{x\in C^*:R_x^*([x]_{\sigma(R^*)})=1\}\in\mathcal{G}\cap C^*$ and
\[\nu(\{x\in C^*:R_x([x]_{\mathcal{G}})=1\})=\nu(\{x\in C^*:R_y^*([x]_{\sigma(R^*)})=1\})=1,\]
which implies that there exists $G\in\mathcal{G}$ such that $\nu(G)=1$ and $R_x([x]_{\mathcal{G}})=1$, for all $x\in G$. It follows for every $A\in\mathcal{G}$ and $x\in G$ that $R_x(A)\geq R_x([x]_\mathcal{G})=1$ when $x\in A$, and $R_x(A)=1-R_x(A^c)\leq 0$ when $x\in A^c$. Thus, $R_x(A)=\delta_x(A)$, for all $A\in\mathcal{X}$ and $x\in G$, that is, $R$ satisfies \eqref{condition:proper}.\\

Conversely, if $R$ satisfies \eqref{condition:proper}, then there exists $F\in\mathcal{G}$ such that $\nu(F)=1$ and $R_x(A)=\delta_x(A)$, for all $A\in\mathcal{X}$ and $x\in F$. It follows that $\mathcal{G}\cap F=\sigma(R_{|\mathcal{G}})\cap F$ and
\[\int_\mathbb{X}R_x(A)\nu(dx)=\int_F\delta_x(A)\nu(dx)=\nu(A),\qquad\mbox{for all }A\in\mathcal{G}.\]
%also $\int_BR_x(A)\nu(dx)=\nu(B\cap A)$ and $\int_ER_x(A)\nu(dx)=\int_AR_x(E)\nu(dx)$, for $A,E,\in\mathcal{G}$ and $B\in\mathcal{X}$.
On the other hand, from \eqref{condition:cg_under}, there exists $C\in\mathcal{G}$ such that $\nu(C)=1$ and $\mathcal{G}\cap C$ is c.g. Then, arguing as in the first part, we can find $C^*\in\mathcal{G}$ such that $\nu(C^*)=1$, $R_x(C^*)=1$ for all $x\in C^*$, and $C^*\subseteq C\cap F$; simply apply the same arguments w.r.t. $C_0:=C\cap F$ and $C_n:=\{x\in C_{n-1}:R_x(C_{n-1})=1\}$, $n\in\mathbb{N}$, which are $\mathcal{G}$-measurable from \eqref{condition:proper}. It follows that $\mathcal{G}\cap C^*=\sigma(R_{|\mathcal{G}})\cap C^*=\sigma(R^*)$, where $R^*:C^*\times\mathcal{G}\cap C^*\rightarrow[0,1]$ is the probability kernel on $C^*$, defined by $R^*_x(B):=R_x(B)$, for $B\in\mathcal{G}\cap C^*$ and $x\in C^*$. Since $\mathcal{G}\cap C^*$ is c.g., then $[x]_\mathcal{G}=\{y\in C^*:R^*_y=R^*_x\}$, for all $x\in C^*$. Moreover, $R_x([x]_\mathcal{G})=1$, for all $x\in C^*$, so we obtain
\[\int_\mathbb{X}R_y(A)R_x(dy)=\int_{[x]_\mathcal{G}\cap C^*}R^*_y(A\cap C^*)R_x(dy)=R_x(A),\quad\mbox{for all }A\in\mathcal{G}\mbox{ and }x\in C^*.\]
%also $\int_BR_y(A)R_x(dy)=R_x(A)R_x(B)=R_x(A\cap B)$ and $\int_\mathbb{X}R_y(A)R_x(dy)=R_x(A)$, for $A\in\mathcal{G}$, $B\in\mathcal{X}$ and $x\in F$.
\end{proof}

\begin{proof}[Proof of Proposition \ref{result:prelim:kernels:total-rcd}]
Suppose that $R$ satisfies \eqref{condition:proper} on $F\in\mathcal{G}$, where $\nu(F)=1$. Let $A\in\mathcal{G}$ and $B\in\mathcal{X}$. Fix $x\in F$. If $x\in A$, then $R_x(A)=1$, so $R_x(A\cap B)=R_x(B)$; otherwise, if $x\in A^c$, then $R_x(A)=0$, so $R_x(A\cap B)=0$. Therefore, $R_x(A\cap B)=R_x(B)\delta_x(A)$. It now follows from \eqref{condition:stationarity} and $\nu(F)=1$ that
\[\int_AR_x(B)\nu(dx)=\int_\mathbb{X}R_x(B)\delta_x(A)\nu(dx)=\int_FR_x(A\cap B)\nu(dx)=\nu(A\cap B).\]
By assumption, $x\mapsto R_x(B)$ is $\mathcal{G}$-measurable, for all $B\in\mathcal{X}$, so $R(\cdot)=\nu(\cdot\mid\mathcal{G})$. To complete the proof, recall from \eqref{condition:cg_under} that $\nu(\cdot\mid\mathcal{G})$ satisfies \eqref{condition:proper} if and only if $\mathcal{G}$ is c.g. under $\nu$.
\end{proof}

\begin{proof}[Proof of Lemma \ref{atoms_general:mixture:lemma}]
Since $\mathcal{G}$ is c.g. under $\nu$, by \eqref{condition:cg_under}, there exists $C\in\mathcal{G}$ such that $\nu(C)=1$, $\mathcal{G}\cap C$ is c.g., and $R_x(A)=\delta_x(A)$, for all $A\in\mathcal{G}$ and $x\in C$. In fact, arguing as in the first part of the proof of Theorem \ref{result:prelim:kernels:rcd}, we can assume without loss of generality that $R_x(C)=1$, for all $x\in C$. Let $A\in\mathcal{G}$. Then
\[A\cap C=\{\delta(A)=1\}\cap C=\{R(A)=1\}\cap C\in\sigma(R)\cap C.\]
But $\sigma(R)\subseteq\mathcal{G}$, so $\mathcal{G}\cap C=\sigma(R)\cap C$, which implies that $[x]_{\mathcal{G}}=[x]_{\mathcal{G}\cap C}=[x]_{\sigma(R)\cap C}$, for all $x\in C$.

By hypothesis, $\mathcal{G}\cap C=\sigma(E_1,E_2,\ldots)$, for some $\pi$-class $\{E_1,E_2,\ldots\}\in\mathcal{G}\cap C$ on $C$. Let us define
\[D:=\bigl\{(x,y)\in C^2:R_x^*=R_y^*\bigr\},\]
and denote by $D_x$ the $x$-section of $D$, where $R^*:C\times\mathcal{G}\cap C\rightarrow[0,1]$ is the probability kernel on $C$, defined by $R_x^*(B):=R_x(B)$, for $B\in\mathcal{G}\cap C$ and $x\in C$. Then
\[[x]_\mathcal{G}=[x]_{\sigma(R)\cap C}=[x]_{\sigma(R^*)}=\{y\in C:R_y^*=R_x^*\}=D_x,\qquad\mbox{for all }x\in C.\]
On the other hand, standard results imply that $D=\{(x,y)\in C^2:R_x^*(E_n)=R_y^*(E_n),n\in\mathbb{N}\}$. Since $(x,y)\mapsto(R_x^*(E_n),R_y^*(E_n))$ is $\mathcal{G}\cap C\otimes\mathcal{G}\cap C$-measurable and the diagonal of $[0,1]^2$ is a measurable set, then $D\in\mathcal{G}\cap C\otimes\mathcal{G}\cap C$. Finally, define
\[\mathcal{A}:=\Bigl\{E\in\mathcal{G}\cap C\otimes\mathcal{G}\cap C:x\mapsto\int_C\mathbbm{1}_E(x,y)R_x^*(dy)\mbox{ is }\mathcal{G}\cap C\mbox{-measurable}\Bigr\}.\]
Let $A,B\in\mathcal{G}\cap C$. It follows that $x\mapsto\int_C\mathbbm{1}_{A\times B}(x,y)R_x^*(dy)=R_x^*(B)\delta_x(A)$ is $\mathcal{G}\cap C$-measurable, so $A\times B\in\mathcal{A}$. In addition, it is easily seen that $\mathcal{A}$ is a $\lambda$-class, so by Dynkin's lemma, $\mathcal{A}=\mathcal{G}\cap C\otimes\mathcal{G}\cap C$. Therefore, $D\in\mathcal{A}$ and $x\mapsto R_x([x]_\mathcal{G})=\int_C\mathbbm{1}_D(x,y)R_x^*(dy)$ is $\mathcal{G}\cap C$-measurable.
\end{proof}

\subsection{Proofs of the results in Section \ref{section:results:hierarchy}}

\begin{proof}[Proof of Proposition \ref{atoms_general}]
It follows from Theorem \ref{introTh_1} that $R(\cdot)=\nu(\cdot\mid\mathcal{G})$ for some c.g. under $\nu$ sub-$\sigma$-algebra $\mathcal{G}\subseteq\mathcal{X}$. By \eqref{condition:cg_under}, there exists $C\in\mathcal{G}$ such that $\nu(C)=1$ and
\begin{equation}\label{atoms_general:proof:eq:proper}
R_x(G)=\delta_x(G),\qquad\mbox{for all }G\in\mathcal{G}\mbox{ and }x\in C.
\end{equation}
Using Boole's inequality and the fact that $(X_n)_{n\geq1}$ are identically distributed with marginal distribution $\nu$, we obtain for all $n\in\mathbb{N}$,
\[\mathbb{P}(X_1\in C,\ldots,X_n\in C)\geq\sum_{i=1}^n\mathbb{P}(X_i\in C)-(n-1)=1.\]

Let $\pi$ and $\Pi$ be as in Section \ref{section:model:atoms} w.r.t. $\mathcal{G}$. Then $\pi$ is $\mathcal{G}\backslash\pi(\mathcal{G})$-measurable. Define $X_n':=\pi(X_n)$, for $n\in\mathbb{N}$. It follows from the exchangeability of $(X_n)_{n\geq1}$ that $(X_n')_{n\geq1}$ is an exchangeable sequence of $\Pi$-valued random variables. Moreover, by \eqref{model:atoms:eq1},
\[\{X_n'\in\pi(B)\}=\{X_n\in B\},\qquad\mbox{for all }B\in\mathcal{G}.\]
Let $n\in\mathbb{N}$ and $B_1,\ldots,B_{n+1}\in\mathcal{G}$. Using \eqref{atoms_general:proof:eq:proper}, we obtain
\begingroup\allowdisplaybreaks
\begin{align*}
\mathbb{E}\bigl[\mathbbm{1}_{\pi(B_1)}(X_1')\cdots\mathbbm{1}_{\pi(B_n)}&(X_n')\cdot\mathbb{P}(X_{n+1}'\in\pi(B_{n+1})|X_1',\ldots,X_n')\bigr]\\
&=\mathbb{E}\bigl[\mathbbm{1}_{B_1}(X_1)\cdots\mathbbm{1}_{B_n}(X_n)\cdot\mathbb{P}(X_{n+1}\in B_{n+1}|X_1,\ldots,X_n)\bigr]\\
&=\mathbb{E}\biggl[\mathbbm{1}_{B_1}(X_1)\cdots\mathbbm{1}_{B_n}(X_n)\cdot\frac{\theta\nu(B_{n+1})+\sum_{i=1}^nR_{X_i}(B_{n+1})}{\theta+n}\cdot\mathbbm{1}_{C^n}(X_1,\ldots,X_n)\biggr]\\
&=\mathbb{E}\biggl[\mathbbm{1}_{\pi(B_1)}(X_1')\cdots\mathbbm{1}_{\pi(B_n)}(X_n')\cdot\frac{\theta\nu_\pi(\pi(B_{n+1}))+\sum_{i=1}^n\delta_{X_i'}(\pi(B_{n+1}))}{\theta+n}\biggr].
\end{align*}
\endgroup
Therefore, $(X_n')_{n\geq1}$ is an exchangeable MVPS$(\theta,\nu_\pi,\delta)$ and, by Theorem \ref{introTh_2}, has directing random measure \eqref{Sethuraman} w.r.t. the parameters $(\theta,\nu_\pi)$, that is, $(X_n')_{n\geq1}$ is a PS.
\end{proof}

\begin{proof}[Proof of Theorem \ref{atoms_general:mixture}]
Let $(X_n)_{n\geq1}$ be an exchangeable MVPS*$(\theta,\nu,R)$ with $\nu(Z)=0$ and directing random measure $\tilde{P}$. It follows from Theorem \ref{introTh_1} that $R(\cdot)=\nu(\cdot\mid\mathcal{G})$, for some sub-$\sigma$-algebra $\mathcal{G}\subseteq\mathcal{X}$, which is c.g. under $\nu$. Let $\pi$ and $\Pi$ be as in Section \ref{section:model:atoms} w.r.t. $\mathcal{G}$. Then $\pi$ is $\mathcal{G}\backslash\pi(\mathcal{G})$-measurable, $\sigma(\pi)=\mathcal{G}$, and $\{[x]_{\mathcal{G}}\}=\pi([x]_\mathcal{G})\in\pi(\mathcal{G})$ for all $x$ in some $E\in\mathcal{G}$ such that $\nu(E)=1$.
%In case $(X_n)_{n\geq1}$ is i.i.d., notice that $\sigma(\pi)=\{\emptyset,\mathbb{X}\}$, so $\nu(\cdot\mid\sigma(\pi))=\nu(\cdot)$, and \eqref{atoms_general:mixture:eq} is trivially true.
Suppose that $(Y_n)_{n\geq1}$ satisfies \eqref{atoms_general:mixture:eq} w.r.t. $\pi$. Let $n\in\mathbb{N}$ and $A_1,\ldots,A_n\in\mathcal{X}$. It follows from \eqref{Sethuraman} w.r.t. the parameters $(\theta,\nu_\pi)$ that
\begingroup\allowdisplaybreaks
\begin{align*}
\mathbb{P}(Y_1\in A_1,\ldots,Y_n\in A_n)&=\mathbb{E}\biggl[\prod_{i=1}^n\nu(A_i|\pi=\tilde{p}_i)\biggr]\\
&=\mathbb{E}\biggl[\prod_{i=1}^n\int_{\Pi}\nu(A_i|\pi=p)\tilde{Q}(dp)\biggr]\\
&=\mathbb{E}\biggl[\prod_{i=1}^n\sum_{j\geq1}V_j\int_\Pi\nu(A_i|\pi=p)\delta_{U_j}(dp)\biggr]\\
&=\mathbb{E}\biggl[\prod_{i=1}^n\sum_{j\geq1}V_j\int_\Pi\nu(A_i|\pi=p)\delta_{\pi(U_j^*)}(dp)\biggr]\\
&=\mathbb{E}\biggl[\prod_{i=1}^n\sum_{j\geq1}V_j\,\nu\bigl(A_i|\pi=\pi(U_j^*)\bigr)\biggr]=\mathbb{E}\biggl[\prod_{i=1}^n\sum_{j\geq1}V_j\,\nu\bigl(A_i|\mathcal{G}\bigr)(U_j^*)\Bigr]=\mathbb{E}\biggl[\prod_{i=1}^n\tilde{P}(A_i)\biggr],
\end{align*}
\endgroup
for some $U_1^*,U_2^*,\ldots\overset{i.i.d.}{\sim}\nu$ independent of $(V_j)_{j\geq1}$. Therefore, $(X_n)_{n\geq1}\overset{d}{=}(Y_n)_{n\geq1}$. Using a suitable randomization, see, e.g., Theorem 8.17 in \cite{kallenberg2021}, we can find $\bigl((\tilde{p}_n^*)_{n\geq1},\tilde{Q}^*\bigr)$ such that $\bigl((X_n)_{n\geq1},(\tilde{p}_n^*)_{n\geq1},\tilde{Q}^*\bigr)\overset{d}{=}\bigl((Y_n)_{n\geq1},(\tilde{p}_n)_{n\geq1},\tilde{Q}\bigr)$, that is, $(X_n)_{n\geq1}$ satisfies the distributional statement \eqref{atoms_general:mixture:eq}.\\
%% In the fifth equation, we use $(\pi(U_j^*))_{j\geq1}\overset{d}{=}(U_j)_{j\geq1}$ and for the second to last that, for every $B\in\sigma(\pi)$, say $B=\pi^{-1}(B^*)$ for $B^*\in\pi(\mathcal{G})$,
%\begingroup\allowdisplaybreaks\begin{align*}
%\int_{\{U_j^*\in B\}}\nu(A_i|\sigma(\pi))(U_j^*(\omega))\mathbb{P}(d\omega)&=\int_B\nu(A_i|\sigma(\pi))(x)\nu(dx)=\int_{B^*}\nu(A_i|\pi=p)\nu_\pi(dp)\\
%&=\int_B\nu(A_i|\pi=\pi(x))\nu(dx)=\int_{\{U_j^*\in B\}}\nu(A_i|\pi=\pi(U_j^*(\omega)))\mathbb{P}(d\omega).
%\end{align*}\endgroup
%As $\nu(A_i|\sigma(\pi))(U_j^*)$, $\nu(A_i|\pi=\pi(U_j^*))$ are $(U_j^*)^{-1}(\sigma(\pi))$-measurable, $\nu(A_i|\sigma(\pi))(U_j^*)=\nu(A_i|\pi=\pi(U_j^*))$ a.s.

Regarding the converse result, suppose that $(X_n)_{n\geq1}$ satisfies \eqref{atoms_general:mixture:eq}, where $\pi:\mathbb{X}\rightarrow\Pi$ is $\sigma(\pi)\backslash\mathcal{P}$-measurable, $\sigma(\pi)$ is c.g. under $\nu$, and $(\Pi,\mathcal{P},\nu_\pi)$ is a probability space such that $\{p\}\in\mathcal{P}$ for $\nu_\pi$-a.e. $p$. It follows from the first part that $(X_n)_{n\geq1}$ is an exchangeable sequence with directing random measure $\tilde{P}(\cdot)=\int_\Pi\nu(\cdot\mid\pi=p)\tilde{Q}(dp)$. Using \eqref{eq:model:prelim:predict_mart}, we obtain from the posterior distribution of a Dirichlet process, see, e.g., \cite[][eq.\,(5.3)]{ghosal2017}, that, for every $A\in\mathcal{X}$,
\begingroup\allowdisplaybreaks
\begin{align}
\mathbb{P}(X_{n+1}\in A|X_1,\ldots,X_n)&=\mathbb{E}[\tilde{P}(A)|X_1,\ldots,X_n]\nonumber\\
&=\mathbb{E}\biggl[\mathbb{E}\Bigl[\int_\Pi\nu(A|\pi=p)\tilde{Q}(dp)\;\bigr|\; \tilde{p}_1,X_1,\ldots,\tilde{p}_n,X_n\Bigr]\;\bigr|\; X_1,\ldots,X_n\biggr]\nonumber\\
&=\mathbb{E}\biggl[\frac{\int_\Pi\nu(A|\pi=p)(\theta\nu_{\pi})(dp)+\sum_{i=1}^n\nu(A|\pi=\tilde{p}_i)}{\theta+n}\;\bigr|\; X_1,\ldots,X_n\biggr]\nonumber\\
&=\frac{\theta\int_\Pi\nu(A|\pi=p)\nu_{\pi}(dp)+\sum_{i=1}^n\mathbb{E}\bigl[\nu(A|\pi=\tilde{p}_i)|X_1,\ldots,X_n\bigr]}{\theta+n}\qquad\mbox{a.s.}\label{atoms_general:mixture:proof:eq}
\end{align}
\endgroup
%% \int_\Pi\nu(A|\pi=p)\nu_\pi(dp)=\int_\mathbb{X}\nu(A|\sigma(\pi))(x)\nu(dx)=\nu(A)
By hypothesis, $\sigma(\pi)$ is c.g. under $\nu$, so \eqref{condition:total} and \eqref{condition:cg_under} imply the existence of a set $C\in\sigma(\pi)$ such that $\nu(C)=1$, $[x]_{\sigma(\pi)}\in\sigma(\pi)$ and $\nu([x]_{\sigma(\pi)}|\sigma(\pi))(x)=1$, for all $x\in C$. By Lemma \ref{atoms_general:mixture:lemma}, $x\mapsto\nu([x]_{\sigma(\pi)}|\sigma(\pi))(x)$ is $\sigma(\pi)$-measurable a.e.$[\nu]$. Moreover,
\begingroup\allowdisplaybreaks
\begin{align*}
[x]_{\sigma(\pi)}&=\bigcup_{x\in\pi^{-1}(P):P\in\mathcal{P}}\pi^{-1}(P)\\
&=\pi^{-1}\biggl(\bigcup_{\pi(x)\in P\in\mathcal{P}}P\biggr)=\pi^{-1}\bigl([\pi(x)]_{\mathcal{P}}\bigr)=\pi^{-1}\bigl(\{\pi(x)\}\bigr)=\{\pi=\pi(x)\},
\end{align*}
\endgroup
for $\nu$-a.e. $x$, since $\{p\}\in\mathcal{P}$ for $\nu_\pi$-a.e. $p$. From these facts and \eqref{atoms_general:mixture:eq}, we obtain, for each $i=1,\ldots,n$,
%% $\pi$ is arbitrary, do not confuse it as the one from Section 2.3
%% $[x]_{\sigma(\pi)}\subseteq\{\pi=\pi(x)\}$ by hypothesis and definition of $[x]_{\sigma(\pi)}$ ( measurability is needed!); for the converse relation, we show that $\{\pi=\pi(x)\}\subseteq\{y\in\mathbb{X}:\delta_x(\pi^{-1}(B))=\delta_y(\pi^{-1}(B)),B\in\mathcal{P}\}=[x]_{\sigma(\pi)}$
\begingroup\allowdisplaybreaks
\begin{align*}
\mathbb{P}(\pi(X_i)=\tilde{p}_i)&=\mathbb{E}\bigl[\mathbb{P}(\pi(X_i)=\tilde{p}_i|\tilde{p}_i)\bigr]\\
&=\mathbb{E}\bigl[\nu(\pi=\tilde{p}_i|\pi=\tilde{p}_i)\bigr]\\
&\overset{(a)}{=}\int_\Pi\nu(\pi=p|\pi=p)\nu_\pi(dp)\\
&\overset{(b)}{=}\int_\mathbb{X}\nu(\pi=\pi(x)|\pi=\pi(x))\nu(dx)=\int_C\nu([x]_{\sigma(\pi)}|\sigma(\pi))(x)\nu(dx)=1,
\end{align*}
\endgroup
where in $(a)$ and $(b)$ we have used the change of variables formula, noting that $p\mapsto\mathbbm{1}_{\{\pi=p\}}(y)$ is $\nu_\pi$-a.e. measurable from the assumption that $\mathcal{P}$ contains $\nu_\pi$-almost every singleton of $\Pi$. Proceeding from \eqref{atoms_general:mixture:proof:eq},
%% from 
%\begingroup\allowdisplaybreaks\begin{align*}
%\mathbb{E}\bigl[\mathbb{P}(\pi(X_i)=\tilde{p}_i|\tilde{p}_i)\bigr]&=\int_\Omega\mathbb{P}(X_i\in\pi^{-1}(\{\tilde{p}_i(\omega)\})|\tilde{p}_i)(\omega)\mathbb{P}(d\omega)=\int_\Omega\nu(\pi^{-1}(\{\tilde{p}_i(\omega)\})|\pi=\tilde{p}_i(\omega))\mathbb{P}(d\omega)\\
%&=\int_\Omega\Bigl(\int_\mathbb{X}\mathbbm{1}_{\{\pi-\tilde{p}_i(\omega)\}}(y)\nu(dy|\pi=\tilde{p}_i(\omega))\Bigr)\mathbb{P}(d\omega)=\int_\Pi\nu(\pi=p|\pi=p)\nu_\pi(dp)\\
%&=\int_\Pi\Bigl(\int_\mathbb{X}\mathbbm{1}_{\{\pi=p\}}(y)\nu(dy|\pi=p)\Bigr)\nu_\pi(dp)=\int_\mathbb{X}\Bigl(\int_\mathbb{X}\mathbbm{1}_{\{\pi=\pi(x)\}}(y)\nu(dy|\pi=\pi(x))\Bigr)\nu(dx)\\
%&=\int_\mathbb{X}\nu(\pi=\pi(x)|\pi=\pi(x))\nu(dx)\\
%&=\int_{\mathbb{X}^2}\mathbbm{1}_{\pi=\pi(x)}(y)\bigl(\nu(dy|\sigma(\pi))(x)\nu(dx)\bigr)=\int_{C^*}\nu([x]_{\sigma(\pi)}|\sigma(\pi))(x)\nu(dx)=1
%\end{align*}\endgroup
%% $\nu(\pi=\pi(x)|\pi=\pi(x))=\nu([x]_{\sigma(\pi)}|\sigma(\pi))(x)$ for $\nu$-a.e. $x$, so $x\mapsto\nu(\pi=\pi(x)|\pi=\pi(x))$ is measurable on the complete probability space.
\begingroup\allowdisplaybreaks
\begin{align*}
\mathbb{P}(X_{n+1}\in A|X_1,\ldots,X_n)&=\frac{\theta\int_\Pi\nu(A|\pi=p)\nu_{\pi}(dp)+\sum_{i=1}^n\mathbb{E}\bigl[\nu(A|\pi=\tilde{p}_i)|X_1,\ldots,X_n\bigr]}{\theta+n}\\
&=\frac{\theta\nu(A)+\sum_{i=1}^n\mathbb{E}\bigl[\nu(A|\pi=\pi(X_i))|X_1,\ldots,X_n\bigr]}{\theta+n}\\
&=\frac{\theta\nu(A)+\sum_{i=1}^n\nu(A|\sigma(\pi))(X_i)}{\theta+n}\qquad\mbox{a.s.},
\end{align*}
\endgroup
that is, $(X_n)_{n\geq1}$ is an exchangeable MVPS with parameters $\bigl(\theta,\nu,\nu(\cdot\mid\sigma(\pi))\bigr)$.
\end{proof}

\begin{remark}\label{atoms_general:mixture:proof:remark}
In proving the converse statement of Theorem \ref{atoms_general:mixture}, the assumption that $\sigma(\pi)$ is c.g. under $\nu$ or, equivalently, that $\nu(\cdot\mid\sigma(\pi))$ satisfies \eqref{condition:proper} is essential. First, observe that $\sigma\bigl(\nu(\cdot\mid\sigma(\pi))\bigr)\subseteq\sigma(\pi)$, so $[x]_{\sigma(\pi)}\subseteq\bigl\{y\in\mathbb{X}:\nu(\cdot\mid\sigma(\pi))(y)=\nu(\cdot\mid\sigma(\pi))(x)\bigr\}$, for all $x\in\mathbb{X}$. As a result, \eqref{condition:proper} implies through \eqref{condition:total} that $\nu(\cdot\mid\sigma(\pi))$ is ``block-diagonal'' in the sense that, for $\nu$-a.e. $x,y$ belonging to the same $\sigma(\pi)$-atom, the measures $\nu(\cdot\mid\sigma(\pi))(y)=\nu(\cdot\mid\sigma(\pi))(x)$ are identical and have full support on $[x]_{\sigma(\pi)}$. Since $X_i$ is sampled from $\nu(\cdot\mid\pi=\tilde{p}_i)$, this fact guarantees us that $\pi(X_i)$ and $\tilde{p}_i$ carry the same information.
\end{remark}

\subsection{Proofs of the results in Section \ref{section:results:null}}

\begin{proof}[Proof of Proposition \ref{results:null:non-negative}]
Let $B\in\mathcal{X}$. From \eqref{model:mvps:tenable}, $\mathbb{P}(\theta\nu(B)+\sum_{i=1}^nR_{X_i}(B)\geq0)=1$, for all $n\in\mathbb{N}$. Fix $\epsilon>0$. Define $G_\epsilon:=\{x\in\mathbb{X}:R_x(B)<-\epsilon\}$ and $N:=\bigl\lceil\frac{\theta\nu(B)}{\epsilon}\bigr\rceil+1$. Assume that $\nu(G_\epsilon)>0$. Letting $\tilde{P}$ be the directing random measure of $(X_n)_{n\geq1}$, we obtain from Jensen's inequality that
\begingroup\allowdisplaybreaks
\begin{align*}
\mathbb{P}\Bigl(\theta\nu(B)+\sum_{i=1}^NR_{X_i}(B)<0\Bigr)&\geq\mathbb{P}(X_1\in G_\epsilon,\ldots,X_N\in G_\epsilon)\\
&=\mathbb{E}\bigl[\tilde{P}(G_\epsilon)^N\bigr]\geq\mathbb{E}[\tilde{P}(G_\epsilon)]^N=(\nu(G_\epsilon))^N>0,
\end{align*}
\endgroup
absurd, unless $\nu(G_\epsilon)=0$. Therefore, taking $\epsilon\downarrow0$, we get $R_x(B)\geq0$ for $\nu$-a.e. $x$.
\end{proof}

\begin{proof}[Proof of Theorem \ref{representation}]
Suppose that $(X_n)_{n\geq1}$ is an exchangeable but not i.i.d. MVPS such that $0<\nu(Z)<1$; otherwise, if $(X_n)_{n\geq1}$ is i.i.d., Theorem \ref{introTh_1} implies \eqref{R_x} w.r.t. $\mathcal{G}=\{\emptyset,Z,Z^c,\mathbb{X}\}$. By \eqref{R=m}, there exists a constant $m>0$ such that $R_x(\mathbb{X})=m$ for all $x\in Z^c$, without loss of generality. Our strategy for proving \eqref{R_x} is to consider $R_x(\cdot\cap Z)$ and $R_x(\cdot\cap Z^c)$ separately. 

Regarding $R_x(\cdot\cap Z)$, let $A,B\in \mathcal{X}$. By exchangeability, $(X_1,X_2)\overset{d}{=}(X_2,X_1)$, so
\begingroup\allowdisplaybreaks
\begin{align*}
\int_{A\cap Z^c}\frac{\theta\nu(B\cap Z)+R_x(B\cap Z)}{\theta+m}\nu(dx)&=\int_{A\cap Z^c}\mathbb{P}(X_2\in B\cap Z|X_1=x)\mathbb{P}(X_1\in dx)\\
&\hspace{2cm}=\mathbb{P}(X_1\in A\cap Z^c,X_2\in B\cap Z)\\
&\hspace{2cm}=\mathbb{P}(X_1\in B\cap Z,X_2\in A\cap Z^c)=\int_{B\cap Z}\nu(A\cap Z^c)\nu(dx),
\end{align*}
\endgroup
which after some simple algebra yields
\[\int_{A\cap Z^c}R_x(B\cap Z)\nu(dx)=\int_{A\cap Z^c}m\,\nu(B\cap Z)\nu(dx);\]
thus, since $A$ is arbitrary and $\mathcal{X}$ is c.g., we obtain, as measures on $\mathbb{X}$,
\begin{equation}\label{representation:eq1}
R_x(\cdot \cap Z)=m\,\nu(\cdot\cap Z)\qquad\mbox{for }\nu\mbox{-a.e. }x\in Z^c.
\end{equation}
Then, in particular, $R_x(Z)=m\,\nu(Z)$ and $R_x(Z^c)=m\,\nu(Z^c)$, for $\nu$-a.e. $x\in Z^c$.

Regarding $R_x(\cdot\cap Z^c)$, we will first focus on the sequence $(X_n)_{n\geq1}$ restricted to $Z^c$, which we will show to be an MVPS with strictly positive reinforcement, and then reason back to the whole sequence $(X_n)_{n\geq1}$. To that end, observe that
\[\mathbb{P}(X_{n+1}\in Z^c|X_1,\ldots,X_n)=\frac{\theta\nu(Z^c)+\sum_{i=1}^nR_{X_i}(Z^c)}{\theta+\sum_{i=1}^nR_{X_i}(\mathbb{X})}\geq\frac{\theta\nu(Z^c)}{\theta+n\cdot m},\]
% we use here that $R_x(\mathbb{X})=m$ for ALL $x\in Z^c$, otherwise we have to use an argument as in Proposition 1.
so $\sum_{n=1}^\infty\mathbb{P}(X_{n+1}\in Z^c|X_1,\ldots,X_n)=\infty$, since $\nu(Z^c)>0$. It follows from the conditional Borel-Cantelli lemma, see, e.g., Theorem 1 in \cite{dubins1965sharper}, that $\sum_{n=1}^\infty\mathbbm{1}_{Z^c}(X_n)=\infty$ a.s., which implies $\mathbb{P}(X_n\in Z^c\mbox{ i.o.})=1$.

Let us define
\[T_0:=0\qquad\mbox{and}\qquad T_n:=\inf\{l>T_{n-1}:X_l\in Z^c\},\qquad\mbox{for }n\geq1.\]
It follows from above that $T_n<\infty$ a.s., so $\Omega^*:=\bigcap_{n=1}^\infty\{T_n<\infty\}$ satisfies $\mathbb{P}(\Omega^*)=1$. To keep the notation simple, we will assume, without loss of generality, that $(\Omega,\mathcal{H},\mathbb{P})=(\Omega^*, \mathcal{H}\cap\Omega^*,\mathbb{P}(\cdot\mid\Omega^*))$. Then $Y_n:=X_{T_n}$ is a well-defined $Z^c$-valued random variable, for all $n\in\mathbb{N}$.

We proceed by showing that the process $(Y_n)_{n\geq1}$ is an exchangeable MVPS with parameters $(\theta^*, \nu^*, R^*)$, where $\theta^*=\theta\nu(Z^c)$, $\nu^*(\cdot)=\nu(\cdot \mid Z^c)$, and $R_x^*(\cdot)=R_x(\cdot)$ on $(Z^c,\mathcal{X}\cap Z^c)$, for $x\in Z^c$. Let $A_1,\ldots,A_n,B\in\mathcal{X}\cap Z^c$, and $\sigma$ be a permutation of $\{1,\ldots,n\}$. It follows from the exchangeability of $(X_n)_{n\geq 1}$ that
\begingroup\allowdisplaybreaks
\begin{align*}
\mathbb{P}(Y_1 \in A_1, \ldots, Y_n \in A_n) &=\sum_{k_1<\cdots<k_n} \mathbb{P}(X_{T_1} \in A_1, \ldots, X_{T_n} \in A_n, T_1=k_1, \ldots, T_n=k_n)\\
&=\sum_{k_1<\cdots<k_n} \mathbb{P}(X_1 \in Z, \ldots, X_{k_1-1} \in Z, X_{k_1} \in A_1, X_{k_1+1} \in Z, \ldots, X_{k_n} \in A_n)\\
&=\sum_{k_1<\cdots<k_n} \mathbb{P}(X_{1} \in Z, \ldots, X_{k_1-1} \in Z, X_{k_1} \in A_{\sigma(1)}, X_{k_1+1} \in Z, \ldots, X_{k_n} \in A_{\sigma(n)})\\
&=\mathbb{P}(Y_1\in A_{\sigma(1)},\ldots,Y_n\in A_{\sigma(n)}).
\end{align*}
\endgroup
On the other hand, 
\begingroup\allowdisplaybreaks
\begin{align*}
\mathbb{E}\bigl[\mathbbm{1}_{A_1}(Y_1)\cdots&\mathbbm{1}_{A_n}(Y_n)\cdot\mathbb{P}(Y_{n+1}\in B|Y_1,\ldots,Y_n)\bigr]\\
&=\sum_{k_1<\cdots<k_{n+1}}\mathbb{E}\bigl[\mathbbm{1}_{A_1}(X_{k_1})\cdots\mathbbm{1}_{A_n}(X_{k_n})\mathbbm{1}_B(X_{k_{n+1}})\mathbbm{1}_{\{T_1=k_1,\ldots,T_{n+1}=k_{n+1}\}}\bigr]\\
&=\sum_{k_1<\cdots<k_{n+1}}\mathbb{E}\bigl[\mathbbm{1}_{A_1}(X_{k_1})\cdots\mathbbm{1}_{A_n}(X_{k_n})\mathbbm{1}_{B\cap Z^c}(X_{k_{n+1}})\mathbbm{1}_{\{T_1=k_1,\ldots,T_n=k_n\}}\mathbbm{1}_{\{T_{n+1}\geq k_{n+1}\}}\bigr]\\
&\begin{aligned}\,\overset{(a)}{=}&\sum_{k_1<\cdots<k_{n+1}}\mathbb{E}\bigl[\mathbbm{1}_{A_1}(X_{k_1})\cdots\mathbbm{1}_{A_n}(X_{k_n})\cdot\mathbb{P}(X_{k_{n+1}}\in B\cap Z^c|X_1,\ldots,X_{k_{n+1}-1})\\
&\hspace{7.5cm}\times\mathbbm{1}_{\{T_1=k_1,\ldots,T_n=k_n\}}\mathbbm{1}_{\{T_{n+1}\geq k_{n+1}\}}\bigr]\end{aligned}\\
&\begin{aligned}\;=&\sum_{k_1<\cdots<k_{n+1}}\mathbb{E}\biggl[\mathbbm{1}_{A_1}(X_{k_1})\cdots\mathbbm{1}_{A_n}(X_{k_n})\cdot\frac{\theta\nu(B\cap Z^c)+\sum_{i=1}^{k_{n+1}-1}R_{X_i}(B\cap Z^c)}{\theta+\sum_{i=1}^{k_{n+1}-1}R_{X_i}(\mathbb{X})}\\
&\hspace{7.5cm}\times\mathbbm{1}_{\{T_1=k_1,\ldots,T_n=k_n\}}\mathbbm{1}_{\{T_{n+1}\geq k_{n+1}\}}\biggr]\end{aligned}\\
&\begin{aligned}\;=&\sum_{k_1<\cdots<k_n}\mathbb{E}\biggl[\mathbbm{1}_{A_1}(X_{k_1})\cdots\mathbbm{1}_{A_n}(X_{k_n})\cdot\frac{\theta\nu(B\cap Z^c)+\sum_{j=1}^nR_{X_{k_j}}(B\cap Z^c)}{\theta+\sum_{j=1}^nR_{X_{k_j}}(\mathbb{X})}\\
&\hspace{6cm}\times\mathbbm{1}_{\{T_1=k_1,\ldots,T_n=k_n\}}\Bigl(\sum_{m=0}^\infty\mathbbm{1}_{\{T_{n+1}>k_n+m\}}\Bigr)\biggr]\end{aligned}\\
&\begin{aligned}\;=\sum_{k_1<\cdots<k_n}\mathbb{E}\biggl[\mathbbm{1}_{A_1}(X_{k_1})\cdots\mathbbm{1}_{A_n}&(X_{k_n})\cdot\frac{\theta\nu(B\cap Z^c)+\sum_{j=1}^nR_{X_{k_j}}(B\cap Z^c)}{\theta+\sum_{j=1}^nR_{X_{k_j}}(\mathbb{X})}\\
&\times\mathbbm{1}_{\{T_1=k_1,\ldots,T_n=k_n\}}\Bigl(\sum_{m=0}^\infty\mathbbm{1}_{\{X_{k_n+1}\in Z,\ldots,X_{k_n+m}\in Z\}}\Bigr)\biggr]\end{aligned}\\
&\begin{aligned}\;=\sum_{k_1<\cdots<k_n}&\mathbb{E}\biggl[\mathbbm{1}_{A_1}(X_{k_1})\cdots\mathbbm{1}_{A_n}(X_{k_n})\cdot\frac{\theta\nu(B\cap Z^c)+\sum_{j=1}^nR_{X_{k_j}}(B\cap Z^c)}{\theta+\sum_{j=1}^nR_{X_{k_j}}(\mathbb{X})}\\
&\times\mathbbm{1}_{\{T_1=k_1,\ldots,T_n=k_n\}}\Bigl(\sum_{m=0}^\infty\mathbb{P}(X_{k_n+1}\in Z,\ldots,X_{k_n+m}\in Z|X_1,\ldots,X_{k_n})\Bigr)\biggr]\end{aligned}\\
&\begin{aligned}\,\overset{(b)}{=}&\sum_{k_1<\cdots<k_n}\mathbb{E}\biggl[\mathbbm{1}_{A_1}(X_{k_1})\cdots\mathbbm{1}_{A_n}(X_{k_n})\cdot\frac{\theta\nu(B\cap Z^c)+\sum_{j=1}^nR_{X_{k_j}}(B\cap Z^c)}{\theta+\sum_{j=1}^nR_{X_{k_j}}(\mathbb{X})}\\
&\hspace{7cm}\times\mathbbm{1}_{\{T_1=k_1,\ldots,T_n=k_n\}}\Bigl(\sum_{m=0}^\infty\bigl(\nu(Z)\bigr)^m\Bigr)\biggr]\end{aligned}\\
&=\mathbb{E}\biggl[\mathbbm{1}_{A_1}(X_{T_1})\cdots\mathbbm{1}_{A_n}(X_{T_n})\cdot\frac{\theta\nu(B\cap Z^c)+\sum_{j=1}^nR_{X_{T_j}}(B\cap Z^c)}{\theta\nu(Z^c)+\sum_{j=1}^nR_{X_{T_j}}(\mathbb{X})\nu(Z^c)}\biggr]\\
&\overset{(c)}{=}\mathbb{E}\biggl[\mathbbm{1}_{A_1}(Y_1)\cdots\mathbbm{1}_{A_n}(Y_n)\cdot\frac{\theta^*\nu^*(B)+\sum_{j=1}^nR_{Y_j}^*(B)}{\theta^*+\sum_{j=1}^nR_{Y_j}^*(Z^c)}\biggr],
\end{align*}
\endgroup
where $(a)$ follows from $\{T_{n+1}\geq k_{n+1}\}\in\sigma(X_1,\ldots,X_{k_{n+1}-1})$; $(b)$ is a result of $\mathbb{P}(X_{n+1}\in Z|X_1,\ldots,X_n)=\nu(Z)$ a.s., using that $R_x(Z)=m\,\nu(Z)$ for $\nu$-a.e. $x\in Z^c$; and $(c)$ since $R_x^*(Z^c)=R_x(Z^c)=m\,\nu(Z^c)$ for $\nu$-a.e. $x\in Z^c$.
%% technically, (b) and (c) use an argument similar to the beginning of Proposition 3.1
Therefore, by Theorem \ref{introTh_1}, there exists a c.g. under $\nu^*$ sub-$\sigma$-algebra $\mathcal{G}^*$ of $\mathcal{X}\cap Z^c$ on $Z^c$ such that, normalized, $R^*$ is an r.c.d. for $\nu^*$ given $\mathcal{G}^*$,
\[\frac{R_x^*(\cdot)}{R_x^*(Z^c)}=\nu^*(\cdot\mid \mathcal{G}^*)(x)\qquad\mbox{for }\nu^*\mbox{-a.e. }x.\]

Let us define
\[\mathcal{G}:=\{A\cup\emptyset, A\cup Z: A\in \mathcal{G}^*\}.\]
Then $\mathcal{G}$ is a sub-$\sigma$-algebra of $\mathcal{X}$ on $\mathbb{X}$. Moreover, for all $B\in\mathcal{X}$, $x\mapsto\frac{R_x(B\cap Z^c)}{R_x(Z^c)}\mathbbm{1}_{Z^c}(x)$ is $\mathcal{G}$-measurable.
% since \[\biggl\{\frac{R(B\cap Z^c)}{R(Z^c)}\mathbbm{1}_{Z^c}\leq t\biggr\}=\biggl(\biggl\{\frac{R(B\cap Z^c)}{R(Z^c)}\leq t\biggr\}\cap Z^c\biggr)\cup Z,\qquad\mbox{for every }t\in[0,1].\]
Let $A\in\mathcal{G}$ and $B\in\mathcal{X}$. Then $A\cap Z^c\in\mathcal{G}^*$ and $Z^c\in\mathcal{G}$, so
\begingroup\allowdisplaybreaks
\begin{align*}
\int_{A} \mathbbm{1}_{Z^c}(x) \frac{R_x(B\cap Z^c)}{R_x(Z^c)} \nu(dx)&= \nu(Z^c) \int_{A\cap Z^c} \frac{R^*_x(B\cap Z^c)}{R_x^*(Z^c)}\nu^*(dx)\\
&=\nu(Z^c)\nu^*(A\cap Z^c\cap B)=\nu(A\cap Z^c \cap B)=\int_A\mathbbm{1}_{Z^c}(x)\nu(B|\mathcal{G})(x)\nu(dx).    
\end{align*}
\endgroup
Since $A$ is arbitrary and $\mathcal{X}$ is c.g., we obtain, as measures on $\mathbb{X}$,
\begin{equation}\label{representation:eq2}
\frac{R_x(\cdot\cap Z^c)}{R_x(Z^c)}=\nu(\cdot\mid\mathcal{G})(x)\qquad\mbox{for }\nu\mbox{-a.e. }x\in Z^c.
\end{equation}
Together, \eqref{representation:eq1}-\eqref{representation:eq2} and the fact that $R_x(Z^c)=m\,\nu(Z^c)$ for $\nu$-a.e. $x\in Z^c$ imply that
\[R_x(\cdot) = R_x(\cdot \cap Z^c) + R_x( \cdot \cap Z) = m\,\nu(Z^c)\nu(\cdot\mid\mathcal{G})(x) + m\,\nu(Z)\nu(\cdot\mid Z),\;\mbox{for }\nu\mbox{-a.e. }x\in Z^c.\]
Finally, recall that $\mathcal{G}^*$ is c.g. under $\nu^*$, that is, there exists $C^*\in\mathcal{G}^*$ such that $\nu^*(C^*)=1$ and $\mathcal{G}^*\cap C^*=\sigma(D_1\cap C^*,D_2\cap C^*,\cdots)$, for some $D_1,D_2,\ldots\in\mathcal{G}^*$. Define $C:=C^*\cup Z$. Then $C\in\mathcal{G}$ and $\nu(C)=1$, as $C^*\subseteq Z^c$. Moreover,
\[\mathcal{G}\cap C=\sigma(Z,D_1\cap C^*,D_2\cap C^*,\ldots).\]

Regarding the converse statement, suppose that $R_x(\mathbb{X})=1$ for all $x\in Z^c$, without loss of generality. It follows from Theorem \ref{result:prelim:predictive_cond} and the discussion in Section \ref{section:model:mvps} that the MVPS with reinforcement kernel \eqref{R_x} will be exchangeable if and only if $\mathbb{P}(X_{n+1}\in A, X_{n+2}\in B| X_1, \ldots, X_n)$ is symmetric w.r.t. $A$ and $B$, for each $n=0, 1, \ldots$ and every $A, B \in \mathcal{X}$. In the case of $n=0$, it holds
\begingroup\allowdisplaybreaks
\begin{align*}
\mathbb{P}(X_1\in A,X_2\in B)&=\int_{A}\frac{\theta\nu(B)+R_x(B)}{\theta+R_x(\mathbb{X})}\nu(dx)\\
&=\int_{A\cap Z}\frac{\theta\nu(B)}{\theta}\nu(dx)+\int_{A\cap Z^c}\frac{\theta\nu(B)+\nu(Z^c)\nu(B|\mathcal{G})(x)+\nu(B\cap Z)}{\theta+1}\nu(dx)\\
&\begin{aligned}\;=&\frac{1}{\theta+1}\biggl((\theta+1)\nu(A\cap Z)\nu(B)+\theta\nu(A\cap Z^c)\nu(B)\\
&\hspace{2cm}+\nu(Z^c)\int_{A\cap Z^c}\nu(B|\mathcal{G})(x)\nu(dx)+\nu(B\cap Z)\nu(A\cap Z^c)\biggr)\end{aligned}\\
&\begin{aligned}\,\overset{(a)}{=}&\frac{1}{\theta+1}\biggl(\theta\nu(A)\nu(B)+\nu(B\cap Z^c)\nu(A\cap Z)+\nu(B\cap Z)\nu(A\cap Z)\\
&\hspace{2cm}+\nu(Z^c)\int_{Z^c}\nu(A|\mathcal{G})(x)\nu(B|\mathcal{G})(x)\nu(dx)+\nu(B\cap Z)\nu(A\cap Z^c)\biggr),\end{aligned}
\end{align*}
\endgroup
where we have used in $(a)$ that
\[\int_{A\cap Z^c}\nu(B|\mathcal{G})(x)\nu(dx)=\mathbb{E}_{\nu}[\nu(A\cap Z^c|\mathcal{G})\nu(B|\mathcal{G})]=\mathbb{E}_{\nu}[\mathbbm{1}_{Z^c}\cdot\nu(A|\mathcal{G})\nu(B|\mathcal{G})],\]
which follows from standard results on conditional expectations and that $Z^c\in\mathcal{G}$.

In the case of $n\geq 1$, the same considerations yield, for each $i=1,\ldots,n$,
\begingroup\allowdisplaybreaks
\begin{align*}
\int_{A\cap Z^c}\nu(B|\mathcal{G})(x) R_{X_i}(dx) &=\nu(Z^c)\int_{A\cap Z^c}\nu(B|\mathcal{G})(x) \nu(dx|\mathcal{G})(X_i)\cdot\mathbbm{1}_{Z^c}(X_i)\\
&=\nu(Z^c)\nu(A|\mathcal{G})(X_i)\nu(B|\mathcal{G})(X_i)\cdot\mathbbm{1}_{Z^c}(X_i)\qquad\mbox{a.s.},
\end{align*}
\endgroup
so arguing in a similar but lengthy way as before, we can show that $\mathbb{P}(X_{n+1}\in A, X_{n+2}\in B| X_1, \ldots, X_n)=\mathbb{P}(X_{n+1}\in B, X_{n+2}\in A| X_1, \ldots, X_n)$.
\end{proof}

\begin{proof}[Proof of Corollary \ref{atoms_null_case}]
The proof is identical to that of Proposition \ref{atoms_general}.\\
%% The only ``non-immediate'' part is in showing that $\nu(B_{n+1}\cap Z)=\nu(Z)\nu_\pi(\pi(B_{n+1})|\pi(Z))$, which follows from the fact that $\pi(B_{n+1}\cap Z)=\pi(B_{n+1})\cap\pi(Z)$.

Note, however, that the particular sub-$\sigma$-algebra $\mathcal{G}=\{A\cup\emptyset, A\cup Z:A\in\mathcal{G}^*\}$, constructed in the proof of Theorem \ref{representation}, has atoms of the form
\[[x]_\mathcal{G}=\begin{cases} [x]_{\mathcal{G}^*} &\mbox{for }x\in Z^c,\\ Z & \mbox{for }x\in Z.
\end{cases}\]
In that case, $\pi(B\cap Z)=\{Z\}$ when $B\cap Z\neq\emptyset$, and $\pi(B\cap Z)=\{\emptyset\}$ when $B\cap Z=\emptyset$, for every $B\in\mathcal{G}$. As a result, the representation \eqref{atoms_null_case:eq} w.r.t. that particular $\mathcal{G}$ becomes
\[(R_\pi)_p(\cdot)=\begin{cases}\nu(Z^c)\delta_p(\cdot)+\nu(Z)\delta_Z(\cdot)&\mbox{for }p\neq Z,\\ 0&\mbox{for }p=Z.\end{cases}\]
\end{proof}

\begin{proof}[Proof of Theorem \ref{directing_rm}]
It follows from Theorem \ref{representation} that $R$ satisfies \eqref{R_x} for some sub-$\sigma$-algebra $\mathcal{G}\subseteq\mathcal{X}$ such that $Z^c\in\mathcal{G}$. Moreover, recall from the proof of Theorem \ref{representation} that $\sum_{n=1}^\infty\mathbbm{1}_{Z^c}(X_n)=\infty$ a.s. and $T_n=\inf\{l>T_{n-1}:X_l\in Z^c\}$ is an a.s. finite random variable, for all $n\in\mathbb{N}$, where $T_0=0$. Regarding \textit{(i)}, let $B\in\mathcal{X}$. It follows from \eqref{eq:model:prelim:predict_conv} and \eqref{R_x} that, on a set of probability one,
\begingroup\allowdisplaybreaks
\begin{align*}
\tilde{P}(B\cap Z)&=\lim_{n\rightarrow\infty}\mathbb{P}(X_{n+1}\in B\cap Z|X_1,\ldots,X_n)\\
&=\lim_{n\rightarrow\infty}\frac{\theta\nu(B\cap Z)+\sum_{i=1}^n\nu(B\cap Z)\cdot\mathbbm{1}_{Z^c}(X_i)}{\theta+\sum_{i=1}^n\mathbbm{1}_{Z^c}(X_i)}=\nu(B\cap Z);
\end{align*}
\endgroup
thus, in particular, $\tilde{P}(Z)=\nu(Z)>0$ a.s. and $\tilde{P}(Z^c)=\nu(Z^c)>0$ a.s. Since $\mathcal{X}$ is c.g., we obtain, as measures on $\mathbb{X}$,
\[\tilde{P}(\cdot)=\tilde{P}(\cdot\cap Z^c)+\tilde{P}(\cdot\cap Z)=\nu(Z^c)\tilde{P}(\cdot\mid Z^c)+\nu(Z)\nu(\cdot\mid Z)\qquad\mbox{a.s.}\]
Furthermore, letting $M_n:=\sum_{i=1}^n\mathbbm{1}_{Z^c}(X_i)$, for $n\in\mathbb{N}$, we get
\begingroup\allowdisplaybreaks
\begin{align}
\tilde{P}(B|Z^c)&=\frac{1}{\nu(Z^c)}\lim_{n\rightarrow\infty}\mathbb{P}(X_{n+1}\in B\cap Z^c|X_1,\ldots,X_n)\nonumber\\
&=\lim_{n\rightarrow\infty}\frac{\theta\nu(B\cap Z^c)+\sum_{i=1}^n\bigl(\nu(Z^c)\nu(B\cap Z^c|\mathcal{G})(X_i)+\nu(Z)\nu(B\cap Z^c|Z)\bigr)\cdot\mathbbm{1}_{Z^c}(X_i)}{\theta\nu(Z^c)+\nu(Z^c)\sum_{i=1}^n\mathbbm{1}_{Z^c}(X_i)}\nonumber\\
&=\lim_{n\rightarrow\infty}\frac{\theta\nu(B\cap Z^c)+\sum_{i=1}^n\nu(Z^c)\nu(B\cap Z^c|\mathcal{G})(X_i)\cdot\mathbbm{1}_{Z^c}(X_i)}{\theta\nu(Z^c)+\nu(Z^c)\sum_{i=1}^n\mathbbm{1}_{Z^c}(X_i)}\nonumber\\
&=\lim_{n\rightarrow\infty}\frac{\theta^*\nu^*(B\cap Z^c)+\sum_{j=1}^{M_n}\nu(Z^c)\nu(B\cap Z^c|\mathcal{G})(X_{T_j})}{\theta^*+\nu(Z^c)M_n}\qquad\mbox{a.s.}, \label{directing_rm:proof:eq}
\end{align}
\endgroup
where $\theta^*:=\theta\nu(Z^c)$ and $\nu^*(B):=\nu(B|Z^c)$. It was already shown in the proof of Theorem \ref{representation} that $(X_{T_n})_{n\geq1}$ is an exchangeable MVPS$(\theta^*,\nu^*,R^*)$, where $R_x^*(\cdot)=R_x(\cdot)=\nu(Z^c)\nu(\cdot\mid\mathcal{G})(x)$ on $(Z^c,\mathcal{X}\cap Z^c)$, for $\nu$-a.e. $x\in Z^c$. Therefore, by Theorem \ref{introTh_2}, the directing random measure $\tilde{P}^*$ of $(X_{T_n})_{n\geq1}$ satisfies
\begin{equation}\label{directing_rm:proof:eq2}
\sup_{A\in\mathcal{X}}\,\bigl|\mathbb{P}(X_{T_{n+1}}\in A\cap Z^c|X_{T_1},\ldots,X_{T_n})-\tilde{P}^*(A\cap Z^c)\bigr|\overset{a.s.}{\longrightarrow}0,
\end{equation}
as $n\rightarrow\infty$, and is equal in law to
\[\tilde{P}^*(\cdot)\overset{w}{=}\sum_{j\geq1}V_j\frac{R^*_{U_j}(\cdot)}{\nu(Z^c)},\]
with $(V_j)_{j\geq1}$ and $(U_j)_{j\geq1}$ as in \eqref{Sethuraman} w.r.t. the parameters $\bigl(\frac{\theta^*}{\nu(Z^c)},\nu(\cdot\mid Z^c)\bigr)$. On the other hand, we have $M_n\overset{a.s.}{\longrightarrow}\infty$, as $n\rightarrow\infty$, so from \eqref{directing_rm:proof:eq},
\[\tilde{P}(B|Z^c)\overset{a.s.}{=}\tilde{P}^*(B\cap Z^c).\]
Using that $\frac{R_{U_j}^*(\cdot)}{\nu(Z^c)}=\frac{R_{U_j}(\cdot)}{R_{U_j}(Z^c)}$ a.s. on $(Z^c,\mathcal{X}\cap Z^c)$, we obtain
\[\tilde{P}(\cdot\mid Z^c)\overset{w}{=}\sum_{j\geq1}V_jR_{U_j}(\cdot\mid Z^c).\]
Finally, it follows from the calculations around \eqref{directing_rm:proof:eq} that, for every $A\in\mathcal{X}$,
\begingroup\allowdisplaybreaks
\begin{align*}
\mathbb{P}(X_{n+1}\in A|X_1,\ldots,X_n)-\tilde{P}(A)&=\mathbb{P}(X_{n+1}\in A\cap Z^c|X_1,\ldots,X_n)-\tilde{P}(A\cap Z^c)\\
&=\nu(Z^c)\bigl(\mathbb{P}(X_{T_{M_{n+1}}}\in A\cap Z^c|X_{T_1},\ldots,X_{T_{M_n}})-\tilde{P}^*(A\cap Z^c)\bigr)\qquad\mbox{a.s.},
\end{align*}
\endgroup
so the convergence of the predictive distributions of $(X_n)_{n\geq1}$ to $\tilde{P}$ in total variation follows from \eqref{directing_rm:proof:eq2}.

The proof of \textit{(ii)} is identical to Theorem \ref{atoms_general:mixture}, using the results found in \textit{(i)}.
%The last statement itself follows from the additional fact that $Z^c\in\sigma(\pi)$, and so $[x]_{\sigma(\pi)}\subseteq Z^c$ for all $x\in Z^c$.
\end{proof}

\subsection{Proofs of the results in Section \ref{section:results:beyond}}

\begin{proof}[Proof of Proposition \ref{results:other:cid}]
Suppose that $(X_n)_{n\geq1}$ is c.i.d. Let $A,B\in\mathcal{X}$. Then
\begin{equation}\label{eq:cid:proof:eq1}
\nu(A)=\mathbb{P}(X_1\in A)=\mathbb{P}(X_2\in A)=\int_\mathbb{X}\frac{\theta\nu(A)+R_x(A)}{\theta+1}\nu(dx),
\end{equation}
which after some simple algebra becomes
\begin{equation}\label{eq:cid:equality1-proof}
\int_\mathbb{X}R_x(A)\nu(dx)=\nu(A).
\end{equation}
On the other hand, we have
\begingroup\allowdisplaybreaks
\begin{align*}
\mathbb{P}(X_1\in A,X_2\in B)&=\int_A\frac{\theta\nu(B)+R_x(B)}{\theta+1}\nu(dx)=\frac{\theta\nu(A)\nu(B)+\int_AR_x(B)\nu(dx)}{\theta+1},
\end{align*}
\endgroup
and
\begingroup\allowdisplaybreaks
\begin{align*}
\mathbb{P}(X_1\in A,X_3\in B)&=\int_A\int_\mathbb{X}\frac{\theta\nu(B)+R_x(B)+R_y(B)}{\theta+2}\frac{\theta\nu(dy)+R_x(dy)}{\theta+1}\nu(dx)\\
&\begin{aligned}\;=\frac{1}{(\theta+2)(\theta+1)}&\biggl\{\theta^2\nu(A)\nu(B)+\theta\int_AR_x(B)\nu(dx)+\theta\nu(A)\int_\mathbb{X}R_y(B)\nu(dy)\\
&+\theta\nu(A)\nu(B)+\int_AR_x(B)\nu(dx)+\int_A\Bigl(\int_\mathbb{X}R_y(B)R_x(dy)\Bigr)\nu(dx)\biggr\}.\end{aligned}
\end{align*}
\endgroup
Since $(X_1,X_2)\overset{d}{=}(X_1,X_3)$, using \eqref{eq:cid:equality1-proof}, we obtain
\begingroup\allowdisplaybreaks
\begin{align*}
(\theta+2)&\biggl(\theta\nu(A)\nu(B)+\int_AR_x(B)\nu(dx)\biggr)\\
&=\theta^2\nu(A)\nu(B)+\theta\int_AR_x(B)\nu(dx)+2\theta\nu(A)\nu(B)+\int_AR_x(B)\nu(dx)+\int_A\Bigl(\int_\mathbb{X}R_y(B)R_x(dy)\Bigr)\nu(dx),
\end{align*}
\endgroup
or, after simplification, $\int_AR_x(B)\nu(dx)=\int_A(\int_\mathbb{X}R_y(B)R_x(dy))\nu(dx)$,
which implies that
\[R_x(B)=\int_\mathbb{X}R_y(B)R_x(dy),\qquad\mbox{for }\nu\mbox{-a.e. }x.\]\\

Conversely, suppose that $R$ satisfies \eqref{condition:stationarity} and \eqref{condition:reversible}. Repeating the argument in \eqref{eq:cid:proof:eq1} in reverse order, we get $X_1\overset{d}{=}X_2$. Moreover, for every $A\in\mathcal{X}$,
\begingroup\allowdisplaybreaks
\begin{align*}
\mathbb{P}(X_3\in A)&=\int_{\mathbb{X}^2}\frac{\theta\nu(A)+R_{x_1}(A)+R_{x_2}(A)}{\theta+2}\mathbb{P}(X_1\in dx_1,X_2\in dx_2)\\
&=\frac{1}{\theta+2}\biggl(\theta\nu(A)+\int_\mathbb{X}R_{x_1}(A)\nu(dx_1)+\int_\mathbb{X}R_{x_2}(A)\nu(dx_2)\biggr)=\nu(A);
\end{align*}
\endgroup
therefore, by induction, $(X_n)_{n\geq1}$ are identically distributed with marginal distribution $\nu$.

Fix $n\in\mathbb{N}$ and $A\in\mathcal{X}$. Let $C\in\mathcal{X}$ with $\nu(C)=1$ be the essential set in \eqref{condition:reversible}. Since $(X_n)_{n\geq1}$ are identically distributed with marginal distribution $\nu$, we have $\mathbb{P}(X_1\in C,\ldots,X_n\in C)=1$, see the proof of Proposition \ref{atoms_general}. It follows from \eqref{condition:stationarity} and \eqref{condition:reversible} that, for $\mathbb{P}_{(X_1,\ldots,X_n)}$-a.e. $(x_1,\ldots,x_n)\in\mathbb{X}^n$,
\begingroup\allowdisplaybreaks
\begin{align*}
\mathbb{P}(X_{n+2}\in A|X_1=x_1,&\ldots,X_n=x_n)=\int_\mathbb{X}\frac{\theta\nu(A)+\sum_{i=1}^{n+1}R_{x_i}(A)}{\theta+n+1}\mathbb{P}(X_{n+1}\in dx_{n+1}|X_1=x_1,\ldots,X_n=x_n)\\
&=\frac{1}{\theta+n+1}\bigg(\theta\nu(A)+\sum_{i=1}^nR_{x_i}(A)+\int_\mathbb{X}R_{x_{n+1}}(A)\frac{\theta\nu(dx_{n+1})+\sum_{i=1}^nR_{x_i}(dx_{n+1})}{\theta+n}\biggr)\\
&=\frac{1}{\theta+n+1}\bigg(\theta\nu(A)+\sum_{i=1}^nR_{x_i}(A)+\frac{\theta\nu(A)+\sum_{i=1}^nR_{x_i}(A)}{\theta+n}\biggr)\\
&=\frac{\theta\nu(A)+\sum_{i=1}^nR_{x_i}(A)}{\theta+n}\\
&=\mathbb{P}(X_{n+1}\in A|X_1,\ldots,X_n),
\end{align*}
\endgroup
which concludes the proof.
\end{proof}

\begin{proof}[Proof of Theorem \ref{results:other:cid-balanced}]
Suppose that $(X_n)_{n\geq1}$ is a c.i.d. MVPS. Let $\mathbb{X}=\{1,\ldots,k\}$. Assume that $(X_n)_{n\geq1}$ is not i.i.d.; otherwise, the result follows from the proof of Proposition 3.1 in \cite{sariev2024}. Define $f(x):=R_x(\mathbb{X})$, for $x\in\mathbb{X}$. By Lemma 2.1 and Theorem 2.2 in \cite{berti2004limit}, there exists a random probability measure $\tilde{P}$ on $\mathbb{X}$ such that
\[\lim_{n\rightarrow\infty}\mathbb{E}[g(X_{n+1})|X_1,\ldots,X_n]=\tilde{P}(g)\quad\mbox{and}\quad\lim_{n\rightarrow\infty}\frac{1}{n}\sum_{i=1}^ng(X_i)=\tilde{P}(g),\quad\mbox{a.s.},\]
for all functions $g:\mathbb{X}\rightarrow\mathbb{R}$, where we use the notation $\mu(g)=\int_\mathbb{X}g(x)\mu(dx)$ for any measure $\mu$ on $\mathbb{X}$. In particular, we have from the fact that $(X_n)_{n\geq1}$ are identically distributed with marginal distribution $\nu$,
\begin{equation}\label{proof:results:other:cid-balanced:eq0}
\mathbb{E}[\tilde{P}(g)]=\mathbb{E}\bigl[\lim_{n\rightarrow\infty}\mathbb{E}[g(X_{n+1})|X_1,\ldots,X_n]\bigr]=\lim_{n\rightarrow\infty}\mathbb{E}[g(X_1)]=\nu(g).
\end{equation}
Moreover,
\begingroup\allowdisplaybreaks
\begin{align}\label{proof:results:other:cid-balanced:eq1}
\begin{aligned}
\tilde{P}(R(g))&=\lim_{n\rightarrow\infty}\frac{1}{n}\sum_{i=1}^nR_{X_i}(g)\\
&=\lim_{n\rightarrow\infty}\frac{\theta\nu(g)+\sum_{i=1}^nR_{X_i}(g)}{\theta+\sum_{i=1}^nf(X_i)}\frac{\sum_{i=1}^nf(X_i)}{n}\\
&=\lim_{n\rightarrow\infty}\mathbb{E}[g(X_{n+1})|X_1,\ldots,X_n]\cdot\lim_{n\rightarrow\infty}\frac{1}{n}\sum_{i=1}^nf(X_i)=\tilde{P}(g)\tilde{P}(f)\quad\mbox{a.s.}
\end{aligned}
\end{align}
\endgroup
If $\tilde{P}\sim Q$, then, by continuity,
\begin{equation}\label{proof:results:other:cid-balanced:eq2}
p(R(g))=p(g)p(f),\qquad\mbox{for all }p\in\mbox{supp}(Q)\mbox{ and }g:\mathbb{X}\rightarrow\mathbb{R}.
\end{equation}
Given these preliminary results, we will prove the necessity of the representation of $R$ in Theorem \ref{results:other:cid-balanced} in several steps, first examining the support of $Q$, then showing that $R$ has a specific ``block-diagonal'' form, and finally proving that the distribution of $f$ is constant across blocks.\\

\noindent\textit{Step 1 (support of $Q$).} Define $P_n(\cdot):=\mathbb{P}(X_{n+1}\in\cdot\mid X_1,\ldots,X_n)$ and $\bar{R}_x:=R_x/f(x)$, for $x\in\mathbb{X}$. Then the convex hull $\mbox{conv}\{\bar{R}_x:x\in\mathbb{X}\}=\{\sum_{i=1}^k\lambda_i\bar{R}_{i}:\lambda_i\geq0,\sum_{i=1}^k\lambda_i=1\}$ is closed. On the other hand, as $n\rightarrow\infty$,
%Observe that $\{P_1,\ldots,P_k\}\subseteq\mathbb{R}^k$. Let $(y_n)_{n=1}^\infty\subseteq\mbox{conv}\{P_j,j=1,\ldots,k\}$ be such that $y_n\rightarrow y$, where $y_n=\sum_{j=1}^k\lambda_j^{(n)}P_j$ for $\lambda_j^{(n)}\geq0$ such that $\sum_{j=1}^k\lambda_j^{(n)}=1$. But $\lambda^{(n)}=(\lambda_1^{(n)},\ldots,\lambda_k^{(n)})\in[0,1]^n$ is bounded, so by the Bolzano-Weierstrass theorem, there exists a subsequence $\lambda^{(n_m)}$ such that $\lambda_j^{(n_m)}\rightarrow\lambda=(\lambda_1,\ldots,\lambda_k)\in[0,1]^k$, which satisfies $\sum_{j=1}^k\lambda_j=1$ from passing to the limit. By continuity, $y_{n_m}\rightarrow\sum_{j=1}^k\lambda_jP_j$, so $y=\sum_{j=1}^k\lambda_jP_j$, and thus $y\in\mbox{conv}\{P_j,j=1,\ldots,k\}$. Equivalently, the map $T:\Delta_{k-1}\rightarrow\mathbb{R}^k:(\lambda_1,\ldots,\lambda_k)\rightarrow\sum_{j=1}^k\lambda_jP_j$ is continuous and $\Delta_{k-1}$ is closed and bounded (thus compact), so $T(\Delta_{k-1})=\mbox{conv}\{P_j,j=1,\ldots,k\}$ is compact in $\mathbb{R}^k$, and thus closed.
\begingroup\allowdisplaybreaks
\begin{align*}
P_n(\{j\})&=\frac{\theta}{\theta+\sum_{l=1}^nf(X_l)}\nu(\{j\})+\sum_{i=1}^n\frac{f(X_i)}{\theta+\sum_{l=1}^nf(X_l)}\bar{R}_{X_i}(\{j\})\\
&=\frac{\theta}{\theta+\sum_{j=1}^nf(X_j)}\nu(\{j\})+\sum_{x\in\mathbb{X}}\frac{\sum_{i=1}^nf(x)\cdot\mathbbm{1}_{\{X_i=x\}}}{\theta+\sum_{l=1}^nf(X_l)}\bar{R}_x(\{j\})\overset{a.s.}{\longrightarrow}\sum_{x\in\mathbb{X}}\frac{f(x)\tilde{P}(\{x\})}{\tilde{P}(f)}\bar{R}_x(\{j\}),
\end{align*}
\endgroup
which implies that
\[\mbox{supp}(Q)\subseteq\mbox{conv}\{\bar{R}_x:x\in\mathbb{X}\}.\]

Take $\epsilon>0$ and $t=(t_1,\ldots,t_k)\subseteq\mbox{conv}\{\bar{R}_x:x\in\mathbb{X}\}$. Let $d$ metrize the weak topology on the space of probability measures. Since $\mathbb{X}$ is finite, $d$ coincides with the total variation norm. Define $g_j(\textbf{n}):=\sum_{i=1}^{k}n_iR_{i}(\{j\})/\sum_{i=1}^kn_if(i)$, for $j=1,\ldots,k$ and $\textbf{n}=(n_1,\ldots,n_k)\in\mathbb{N}^k$. Then $\sum_{i=1}^{k}n_iR_{i}/\sum_{i=1}^kn_if(i)\in\mbox{conv}\{\bar{R}_x:x\in\mathbb{X}\}$, so there exist $\textbf{n}_\epsilon=(n_{1,\epsilon},\ldots,n_{k,\epsilon})\in\mathbb{N}^k$ such that
\[\max_{1\leq j\leq k}|g_j(\textbf{n}_\epsilon)-t_j|<\frac{\epsilon}{k}.\]
Moreover, $\mathbb{P}(X_{|\textbf{n}|+1}=j|X_1=x_1,\ldots,X_{|\textbf{n}|}=x_{|\textbf{n}|})-g_j(\textbf{n})=O\bigl(\theta/\sum_{i=1}^{|\textbf{n}|}f(X_i)\bigr)$, where $|\textbf{n}|:=\sum_{i=1}^kn_i$ with $n_i=\#\{l:x_l=i\}$, for $i=1,\ldots,k$. Then, letting $C_{|\textbf{n}_\epsilon|}:=\{d(P_{|\textbf{n}|_\epsilon}-t)<\epsilon\}$ and taking a multiple of $|\textbf{n}|_\epsilon$ if necessary, which would leave the $g_j(\textbf{n}_\epsilon)$'s unchanged, we have $\mathbb{P}(C_{|\textbf{n}_\epsilon|})>0$.

Now, since $(X_n)_{n\geq1}$ is c.i.d., the sequence $(P_{n+m}(\{j\}))_{m\geq0}$ is a martingale w.r.t. $(\sigma(X_1,\ldots,X_{n+m}))_{m\geq0}$, for every $j=1,\ldots,k$ and $n\in\mathbb{N}$. Moreover, 
\begingroup\allowdisplaybreaks
\begin{align*}
\bigl|P_{n+m+1}(\{j\})-P_{n+m}(\{j\})\bigr|&=\bigl|P_{n+m}(\{j\})-\bar{R}_{X_{n+m+1}}(\{j\})\bigr|\frac{f(X_{n+m+1})}{\theta+\sum_{i=1}^{n+m+1}f(X_i)}\leq\frac{\bar{f}}{\theta+(n+m+1)\underline{f}},
\end{align*}
\endgroup
for all $m\in\mathbb{N}$, where $\bar{f}=\max_{1\leq j\leq k}f(j)$ and $\underline{f}=\min_{1\leq j\leq k}f(j)$. By the maximal Azuma-Hoeffding inequality, see, e.g., \cite[][Corollary 6.9 and Section 6(c)]{mcdiarmid1989} and \cite{roch2022},
\[\mathbb{P}\Bigl(\sup_{m\geq1}\bigl|P_{n+m}(\{j\})-P_n(\{j\})\bigr|>2\epsilon/k\bigr| X_1,\ldots,X_n\Bigr)\leq2\exp\biggl\{-\frac{2\epsilon^2}{k^2\sum_{m=n+1}^\infty\bigl(\frac{\bar{f}}{\theta+m\underline{f}}\bigr)^2}\biggr\},\]
which goes to $0$, as $n\rightarrow\infty$. Then
\begingroup\allowdisplaybreaks
\begin{align*}
\mathbb{P}\Bigl(C_{|\textbf{n}_\epsilon|};\sup_{m\geq1}d(P_{{|\textbf{n}_\epsilon|}+m}-&P_{|\textbf{n}_\epsilon|})>\epsilon\Bigr)=\int_{C_{|\textbf{n}_\epsilon|}}\mathbb{P}\Bigl(\sup_{m\geq1}d(P_{{|\textbf{n}_\epsilon|}+m}-P_{|\textbf{n}_\epsilon|})>\epsilon\bigr|X_1,\ldots,X_{|\textbf{n}_\epsilon|}\Bigr)(\omega)\mathbb{P}(d\omega)\\
&\leq\sum_{j=1}^k\int_{C_{|\textbf{n}_\epsilon|}}\mathbb{P}\Bigl(\sup_{m\geq1}\bigl|P_{{|\textbf{n}_\epsilon|}+m}(\{j\})-P_{|\textbf{n}_\epsilon|}(\{j\})\bigr|>2\epsilon/k\bigr| X_1,\ldots,X_{|\textbf{n}_\epsilon|}\Bigr)(\omega)\mathbb{P}(d\omega)\\
&\leq 2k\exp\biggl\{-\frac{2\epsilon^2}{k^2\sum_{m={|\textbf{n}_\epsilon|}+1}^\infty\bigl(\frac{\bar{f}}{\theta+m\underline{f}}\bigr)^2}\biggr\}\mathbb{P}(C_{|\textbf{n}_\epsilon|})\\
&<\mathbb{P}(C_{|\textbf{n}_\epsilon|}),
\end{align*}
\endgroup
where again we take a multiple of $|\textbf{n}|_\epsilon$ if necessary to guarantee the last inequality. Therefore,
\[\mathbb{P}\Bigl(C_{|\textbf{n}_\epsilon|};\sup_{m\geq1}d(P_{{|\textbf{n}_\epsilon|}+m}-P_{|\textbf{n}_\epsilon|})\leq\epsilon\Bigr)>0,\]
which implies that $t\in\mbox{supp}(Q)$. Thus, ultimately,
\begin{equation}\label{proof:results:other:cid-balanced:eq3}
\mbox{supp}(Q)=\mbox{conv}\{\bar{R}_x:x\in\mathbb{X}\}.
\end{equation}~ 

\noindent\textit{Step 2 (structure of $R$).}
It follows from \eqref{proof:results:other:cid-balanced:eq2} and \eqref{proof:results:other:cid-balanced:eq3} that
\begin{equation}\label{proof:results:other:cid-balanced:eq4}
\bar{R}_x(R(g))=\bar{R}_x(g)\bar{R}_x(f),\qquad\mbox{for all }x\in\mathbb{X}\mbox{ and }g:\mathbb{X}\rightarrow\mathbb{R}.
\end{equation}
Since $(X_n)_{n\geq1}$ is not i.i.d., then $\bar{R}_i\neq\bar{R}_j$ for at least one pair $i\neq j$, so $\mbox{dim}(\mbox{supp}(Q))\geq1$. Let $p_1,p_2\in\mbox{supp}(Q)$ be such that $p_1\neq p_2$. As $\mbox{supp}(Q)$ is convex, then $\frac{p_1+p_2}{2}\in\mbox{supp}(Q)$ and, applying \eqref{proof:results:other:cid-balanced:eq2} with $g=f$ to $p_1,p_2$ and $\frac{p_1+p_2}{2}$, we get
\begingroup\allowdisplaybreaks
\begin{align*}
(p_1(f))^2+(p_2(f))^2&=p_1(R(f))+p_2(R(f))=2\biggl(\Bigl(\frac{p_1+p_2}{2}\Bigr)(f)\biggr)^2=\frac{1}{2}\bigl((p_1(f))^2+2\,p_1(f)p_2(f)+(p_2(f))^2\bigr),
\end{align*}
\endgroup
which implies that $(p_1(f)-p_2(f))^2=0$. Therefore, $p(f)=c>0$ is constant, for all $p\in\mbox{supp}(Q)$. In particular, $\bar{R}_x(f)=c$, so from \eqref{proof:results:other:cid-balanced:eq4},
\begin{equation}\label{proof:results:other:cid-balanced:eq5}
R_x(R(g))=c\cdot R_x(g),\qquad\mbox{for all }x\in\mathbb{X}\mbox{ and }g:\mathbb{X}\rightarrow\mathbb{R}.
\end{equation}
On the other hand, $\tilde{P}(f)=c$ a.s., so from \eqref{proof:results:other:cid-balanced:eq0} and \eqref{proof:results:other:cid-balanced:eq1},
\begin{equation}\label{proof:results:other:cid-balanced:eq6}
\nu(R(g))=c\cdot\nu(g)\qquad\mbox{for all }g:\mathbb{X}\rightarrow\mathbb{R}.
\end{equation}

Let us consider $R$ as a $k\times k$ matrix, and $\nu$ as a $k$-dimensional vector. It follows from \eqref{proof:results:other:cid-balanced:eq5} that $R/c$ is a non-negative idempotent matrix whose rows are nonzero. Furthermore, \eqref{proof:results:other:cid-balanced:eq6} implies that no column of $R/c$ is zero. According to Theorem 2 in \cite{flor1969}, $R/c$ and, as a consequence $R$, becomes a block-diagonal matrix after a permutation of the coordinates, where each block is a positive rank-one idempotent matrix. Let us partition $\mathbb{X}$ according to the blocks $B_1,\ldots,B_m$ in $R$, for some $m\in\{2,\ldots,k\}$, where the case $m=1$ is excluded, since it leads to an i.i.d. sequence. It follows from the structure of $R$ that, for each $j\in\{1,\ldots,m\}$, there exists a positive probability measure $p^{B_j}$ on $B_j$ such that
\begin{equation}\label{proof:results:other:cid-balanced:eq7}
R_x(\cdot)=f(x)p^{B_j}(\cdot)\qquad\mbox{for all }x\in B_j.
\end{equation}

Fix $j\in\{1,\ldots,m\}$. Let $A\subseteq B_j$. It follows from \eqref{proof:results:other:cid-balanced:eq6} and \eqref{proof:results:other:cid-balanced:eq7} that
\[c\cdot\nu(A)=\nu(R(A))=\nu(f\cdot\mathbbm{1}_{B_j})p^{B_j}(A).\]
In particular, $c\cdot\nu(B_j)=\nu(f\cdot\mathbbm{1}_{B_j})$, so combining both expressions gives
\[p^{B_j}(A)=\frac{c}{\nu(f\cdot\mathbbm{1}_{B_j})}\nu(A)=\nu(A|B_j).\]
Therefore,
\begin{equation}\label{proof:results:other:cid-balanced:eq9}
R_x(\cdot)=f(x)\cdot\nu(\cdot\mid B_j),\qquad\mbox{for all }x\in B_j\mbox{ and }j=1,\ldots,m.
\end{equation}~

\noindent\textit{Step 3 (distribution of $f$).} Let $j\in\{1,\ldots,m\}$ and $n\in\mathbb{N}_0$. Since $(X_1,\ldots,X_n,X_{n+2})\overset{d}{=}(X_1,\ldots,X_n,X_{n+1})$, we obtain from \eqref{proof:results:other:cid-balanced:eq9} that, on $\{X_1\in B_j,\ldots,X_n\in B_j\}$ a.s.,
\begingroup\allowdisplaybreaks
\begin{align*}
\frac{\theta\nu(B_j)+\sum_{i=1}^nf(X_i)}{\theta+\sum_{i=1}^nf(X_i)}&=\mathbb{P}(X_{n+1}\in B_j|X_1,\ldots,X_n)\\
&=\mathbb{E}\bigl[\mathbb{P}(X_{n+2}\in B_j|X_1,\ldots,X_{n+1})|X_1,\ldots,X_n\bigr]\\
&=\int_\mathbb{X}\frac{\theta\nu(B_j)+\sum_{i=1}^nf(X_i)+f(x)\cdot\mathbbm{1}_{B_j}(x)}{\theta+\sum_{i=1}^nf(X_i)+f(x)}\frac{\theta\nu(dx)+\sum_{i=1}^nf(X_i)\nu(dx|B_j)}{\theta+\sum_{i=1}^nf(X_i)},
\end{align*}
\endgroup
which upon cancellation of $1/(\theta+\sum_{i=1}^nf(X_i))$, some simple algebra, and setting $h:=1/(\theta+\sum_{i=1}^nf(X_i)+f(x))$ becomes
\begingroup\allowdisplaybreaks
\begin{align*}
\theta&\nu(B_j)+\sum_{i=1}^nf(X_i)\\
&\begin{aligned}\;=\Bigl(\theta\nu(B_j)+\sum_{i=1}^nf(X_i)\Bigr)\theta\nu(h)&+\theta\nu(h\cdot f\cdot\mathbbm{1}_{B_j})\\
&+\Bigl(\theta\nu(B_j)+\sum_{i=1}^nf(X_i)\Bigr)\sum_{i=1}^nf(X_i)\nu(h|B_j)+\sum_{i=1}^nf(X_i)\nu(h\cdot f\cdot\mathbbm{1}_{B_j}|B_j)\end{aligned}\\
&\begin{aligned}\;=\nu(B_j)\Bigl(\theta+\sum_{i=1}^n\frac{f(X_i)}{\nu(B_j)}\Bigr)\theta\nu(h)&+\theta\nu(h\cdot f\cdot\mathbbm{1}_{B_j})\\
&+\nu(B_j)\Bigl(\theta+\sum_{i=1}^n\frac{f(X_i)}{\nu(B_j)}\Bigr)\sum_{i=1}^nf(X_i)\nu(h|B_j)+\sum_{i=1}^n\frac{f(X_i)}{\nu(B_j)}\nu(h\cdot f\cdot\mathbbm{1}_{B_j}).\end{aligned}
\end{align*}
\endgroup
Upon further cancellation of $\theta+\sum_{i=1}^n\frac{f(X_i)}{\nu(B_j)}$, we get
\begingroup\allowdisplaybreaks
\begin{align*}
\nu(B_j)&=\nu(B_j)\theta\nu(h)+\theta\nu(h\cdot f\cdot\mathbbm{1}_{B_j})+\nu(B_j)\sum_{i=1}^nf(X_i)\nu(h|B_j)\\
&=\nu(B_j)\theta\nu(h)+\nu(B_j)\nu\Bigl(h\cdot\Bigl(\sum_{i=1}^nf(X_i)+f\Bigr)\mid B_j\Bigr),
\end{align*}
\endgroup
so
\[\nu(B_j)\theta\nu(h)=\nu(B_j)\nu\Bigl(1-h\cdot\Bigl(\sum_{i=1}^nf(X_i)+f\Bigr)\mid B_j\Bigr)=\nu(B_j)\nu\bigl((h\cdot\theta)\mid B_j\bigr),\]
or, equivalently, $\nu(h)=\nu(h|B_j)$. Multiplying both sides by $\theta+\sum_{i=1}^nf(X_i)$ and subtracting $1$ gives
\begin{equation}\label{proof:results:other:cid-balanced:eq10}
\nu\biggl(\frac{f}{\theta+\sum_{i=1}^nf(x_i)+f}\biggr)=\nu\biggl(\frac{f}{\theta+\sum_{i=1}^nf(x_i)+f}\mid B_j\biggr),
\end{equation}
for all $x_1,\ldots,x_n\in B_j$ and $n\in\mathbb{N}_0$.

Suppose that the distinct values of $f$ are $a_1,\ldots,a_L\in(0,\infty)$. Let us define $p_l:=\nu(f=a_l)$ and $p_{jl}:=\nu(f=a_l|B_j)$, for $l=1,\ldots,L$ and $j=1,\ldots,m$. Fix $j\in\{1,\ldots,m\}$. Define $C_j:=\{\theta+\sum_{i=1}^nf(x_i):x_1,\ldots,x_n\in B_j,n\in\mathbb{N}_0\}$, noting that $C_j$ is infinite. It follows from \eqref{proof:results:other:cid-balanced:eq10} that
\[\sum_{l=1}^L\frac{a_l}{c+a_l}(p_{jl}-p_l)=0,\qquad\mbox{for all }c\in C_j.\]
Multiplying both sides on all denominators, we get polynomials of the type
\begin{equation}\label{proof:results:other:cid-balanced:eq11}
P_j(c)=\sum_{l=1}^La_l(p_{jl}-p_l)\prod_{h\neq l}(c+a_h),\qquad\mbox{for }c\in C_j.
\end{equation}
Since $P_j(c)$ is a polynomial of degree $L-1$ and $P_j(c)=0$ for infinitely many $c$, then $P_j(c)\equiv 0$ for all $c\in\mathbb{R}$. In particular,
\[0=P_j(-a_i)=a_i(p_{ji}-p_i)\prod_{h\neq i}(-a_i+a_h),\qquad\mbox{for }i=1,\ldots,L.\]
But $a_h\neq a_i$, for $h\neq i$, and $a_i>0$. Therefore, $\nu(f=a_i|B_j)=p_{ji}=p_i=\nu(f=a_i)$, for all $i=1,\ldots,L$ and $j=1,\ldots,m$.\\

Regarding the converse result, suppose that $(X_n)_{n\geq1}$ is an MVPS$(\theta,\nu,R)$ such that $R$ satisfies \textit{1.} and \textit{2.} from the statement of Theorem \ref{results:other:cid-balanced} w.r.t. some partition $B_1,\ldots,B_m$ of $\mathbb{X}$. Define $T_{n,j}:=\sum_{i=1}^nf(X_i)\cdot\mathbbm{1}_{B_j}(X_i)$ and $D_n:=\sum_{i=1}^nf(X_i)$, for $j=1,\ldots,m$ and $n\in\mathbb{N}$. Let $A\subseteq\mathbb{X}$. Since $R_x(A)=\sum_{j=1}^mf(x)\cdot\mathbbm{1}_{B_j}(x)\nu(A|B_j)$, then
\[P_n(A):=\mathbb{P}(X_{n+1}\in A|X_1,\ldots,X_n)=\sum_{j=1}^m\frac{\theta\nu(B_j)+T_{n,j}}{\theta+D_n}\nu(A|B_j).\]
In particular, $P_n(A\cap B_j)=\frac{\theta\nu(B_j)+T_{n,j}}{\theta+D_n}\nu(A|B_j)$, for any $j\in\{1,\ldots,m\}$, so
\[\mathbb{P}(X_{n+1}\in A\cap B_j|X_1,\ldots,X_n;X_{n+1}\in B_j)=\frac{P_n(A\cap B_j)}{P_n(B_j)}=\nu(A|B_j).\]
From this, we get
\begingroup\allowdisplaybreaks
\begin{align}\label{proof:results:other:cid-balanced:eq12}
\mathbb{P}(X_{n+2}\in A|X_1,\ldots,X_n)&=\sum_{j=1}^m\mathbb{P}(X_{n+2}\in A\cap B_j|X_1,\ldots,X_n)\nonumber\\
&=\sum_{j=1}^m\mathbb{E}\bigl[\mathbb{P}(X_{n+2}\in A\cap B_j|X_1,\ldots,X_{n+1};X_{n+2}\in B_j)P_{n+1}(B_j)\mid X_1,\ldots,X_n\bigr]\nonumber\\
&=\sum_{j=1}^m\mathbb{P}(X_{n+2}\in B_j|X_1,\ldots,X_n)\nu(A|B_j).
\end{align}
\endgroup

On the other hand, for all $j\in\{1,\ldots,m\}$ and $a\in(0,\infty)$,
\begingroup\allowdisplaybreaks
\begin{align*}
\mathbb{P}(X_{n+1}\in B_j,f(X_{n+1})=a|X_1,\ldots,X_n)&=\sum_{x\in B_j:f(x)=a}\frac{\theta\nu(B_j)+T_{n,j}}{\theta+D_n}\nu(\{x\}|B_j)\\
&=\frac{\theta\nu(B_j)+T_{n,j}}{\theta+D_n}\nu(f=a|B_j)\\
&=\frac{\theta\nu(B_j)+T_{n,j}}{\theta+D_n}\nu(f=a)\\
&=\mathbb{P}(X_{n+1}\in B_j|X_1,\ldots,X_n)\nu(f=a);
\end{align*}
\endgroup
thus, $f(X_{n+1})$ and $\mathbbm{1}_{B_j}(X_{n+1})$ are conditionally independent given $(X_1,\ldots,X_n)$. Moreover, summing over $j\in\{1,\ldots,m\}$, we have $f(X_{n+1})\mid X_1,\ldots,X_n\sim\nu\circ f^{-1}$. Then
\begingroup\allowdisplaybreaks
\begin{align*}
\mathbb{P}(X_{n+2}\in B_j|X_1,\ldots,X_n)&=\mathbb{E}\Bigl[\frac{\theta\nu(B_j)+T_{n,j}+f(X_{n+1})\cdot\mathbbm{1}_{B_j}(X_{n+1})}{\theta+D_n+f(X_{n+1})}\bigr|X_1,\ldots,X_n\Bigr]\\
&=\bigl(\theta\nu(B_j)+T_{n,j}\bigr)\int_\mathbb{X}\frac{1}{\theta+D_n+f(x)}\nu(dx)+P_n(B_j)\int_{\mathbb{X}}\frac{f(x)}{\theta+D_n+f(x)}\nu(dx)\\
&=\frac{\theta\nu(B_j)+T_{n,j}}{\theta+D_n}\biggl(\int_\mathbb{X}\frac{\theta+D_n}{\theta+D_n+f(x)}\nu(dx)+\int_{\mathbb{X}}\frac{f(x)}{\theta+D_n+f(x)}\nu(dx)\biggr)\\
&=\frac{\theta\nu(B_j)+T_{n,j}}{\theta+D_n}.
\end{align*}
\endgroup
Plugging this into \eqref{proof:results:other:cid-balanced:eq12}, we get
\[\mathbb{P}(X_{n+2}\in A|X_1,\ldots,X_n)=\sum_{j=1}^m\frac{\theta\nu(B_j)+T_{n,j}}{\theta+D_n}\nu(A|B_j)=\mathbb{P}(X_{n+1}\in A|X_1,\ldots,X_n),\]
which completes the proof of the theorem.
\end{proof}

\section*{Acknowledgments}

This study is financed by the European Union-NextGenerationEU, through the National Recovery and Resilience Plan of the Republic of Bulgaria, project No. BG-RRP-2.004-0008.

\bibliography{mybib} 

\end{document}